\documentclass[1 [leqno,11pt]{amsart}
\usepackage{amssymb, amsmath}
\usepackage{hyperref}
\usepackage{xcolor}
\allowdisplaybreaks[4]

\let\p=\partial

\let\f=\frac
\let\wt=\widetilde
\let\wh=\widehat
\let\D=\Delta

\def\Z{\mathop{\mathbb Z\kern 0pt}\nolimits}
\def\N{\mathop{\mathbb N\kern 0pt}\nolimits}
\def\Q{\mathop{\mathbb Q\kern 0pt}\nolimits}
\def\R{\mathop{\mathbb R\kern 0pt}\nolimits}

\def\l{{\rm l}}

\def\na{\nabla}

\def\l3{L^3(\R^3)}

\def\ep{\varepsilon}

\def\e{\varepsilon}

\newcommand{\Rmnum}[1]{\uppercase\expandafter{\romannumeral #1} }

\newcommand{\beq}{\begin{equation}}
\newcommand{\eeq}{\end{equation}}
\newcommand{\ben}{\begin{eqnarray}}
\newcommand{\een}{\end{eqnarray}}
\newcommand{\beno}{\begin{eqnarray*}}
\newcommand{\eeno}{\end{eqnarray*}}

 \numberwithin{equation}{section}
\newcommand{\andf}{\quad\hbox{and}\quad}

\newtheorem{thm}{Theorem}[section]
\newtheorem{lem}{Lemma}[section]
\newtheorem{rmk}{Remark}[section]

\newtheorem{prop}{Proposition}[section]

\begin{document}

\title[Inhomogeneous Navier-Stokes equations]
{Solutions of the 3D inhomogeneous incompressible Navier-Stokes  system with initial velocity in $\rm VMO^{-1}$}

\author[R. Hu]{Ruilin Hu}
\address[R. Hu]
 {University of Bath, Bath, BA2 7AY, UK.} \email{rh2488@bath.ac.uk}
\author[Q.-H. Nguyen]{Quoc-Hung Nguyen}
\address[Q.-H. Nguyen]
 {Academy of Mathematics $\&$ Systems Science, The Chinese Academy of
Sciences, Beijing 100190, CHINA. } 
\email{qhnguyen@amss.ac.cn}
\author[F. Shao]{Feng Shao}
\address[F. Shao]
 {School of Mathematical Science, Peking University, Beijing 100871, CHINA.} 
 \email{fshao@stu.pku.edu.cn}
\author[D. Wei]{Dongyi Wei}
\address[D. Wei]
 {School of Mathematical Science, Peking University, Beijing 100871, CHINA.} 
 \email{jnwdyi@pku.edu.cn}
\author[P. Zhang]{Ping Zhang}%
\address[P. Zhang]
 {State Key Laboratory of Mathematical Sciences,  Academy of Mathematics $\&$ Systems Science
and  Hua Loo-Keng Key Laboratory of Mathematics, The Chinese Academy of
Sciences, Beijing 100190, CHINA, and School of Mathematical Sciences, University of Chinese Academy of Sciences, Beijing 100049, CHINA. }
\email{zp@amss.ac.cn}
\author[Z. Zhang]{Zhifei Zhang}
\address[Z. Zhang]
 {School of Mathematical Science, Peking University, Beijing 100871, CHINA.}
 \email{zfzhang@math.pku.edu.cn}%
\date{\today}

\begin{abstract}
In this paper, we establish local existence of strong solutions for the three-dimensional inhomogeneous incompressible Navier–Stokes equations with initial data $(\rho_0, u_0)$ lying in $(C^1, L^2 \cap \mathrm{VMO}^{-1})$, where $\rho_0$ has a positive lower bound. Furthermore, if $\rho_0\in C^2$ and $\|\rho_0-1\|_{L^\infty}+\|u_0\|_{\mathrm{BMO}^{-1}}$ is sufficiently small, we prove  global existence of the solution. To achieve this, we employ an estimate for the transport equation to obtain regularity for the density and apply a new freezing coefficient method for the momentum equation, as in \cite{KHN24}.\end{abstract}
\maketitle

\vskip-0.3cm

\setcounter{equation}{0}
	\section{Introduction}
	
In this paper, we study the existence of solutions to the initial value problem for the following three-dimensional inhomogeneous incompressible Navier–Stokes system:
\begin{equation}\label{eqinns}
\begin{cases}
\partial_t \rho + u \cdot \nabla \rho = 0, \qquad (t,x) \in \mathbb{R}^+ \times \mathbb{R}^3,\\[4pt]
\partial_t (\rho u) + \operatorname{div}(\rho u \otimes u) - \Delta u + \nabla P = 0,\\[4pt]
\nabla \cdot u = 0,\\[4pt]
(\rho,u)|_{t=0} = (\rho_0, u_0).
\end{cases}
\end{equation}
Here $(\rho, u, P)$ denote the density, velocity, and pressure, respectively. This system can be used to describe a mixture of several immiscible incompressible fluids with varying densities; we refer to \cite{PL96} for a detailed derivation of \eqref{eqinns}.

A simplified case is $\rho \equiv 1$, which reduces the system to the classical Navier–Stokes equations. One  observes that, just as in the classical case, system \eqref{eqinns} also possesses a scaling invariance. More precisely, if $(\rho, u, P)$ is a solution to \eqref{eqinns} on $[0,T]$, then
\begin{equation}
\label{S1eq1} (\rho(t,x), u(t,x), P(t,x)) \mapsto \bigl(\rho(\lambda^2 t, \lambda x), \lambda u(\lambda^2 t, \lambda x), \lambda^2 P(\lambda^2 t, \lambda x)\bigr)
\end{equation}
is also a solution of \eqref{eqinns} on $[0, T/\lambda^2]$. A Banach space $X$ is called \textit{critical} if its norm is invariant under the scaling  transformation \eqref{S1eq1}. Some classical critical spaces in three dimensions include the Sobolev space $\dot{H}^{\frac12}$, the Lebesgue space $L^3$, the Besov space $\dot{B}_{p,q}^{-1+\frac{3}{p}}$, and the spaces $\mathrm{BMO}^{-1}$ and $\mathrm{VMO}^{-1}$, which will be recalled later.

Let us first survey some known results for the classical Navier–Stokes equations $(NS)$. The research dates back to \cite{J34}, where Leray established the existence of global weak solutions for $L^2$ data. In the critical setting, Fujita and Kato \cite{FK64} proved local well-posedness in $\dot{H}^{\frac12}$, and Kato \cite{K84} proved local well-posedness for $L^3$ data. For data with negative regularity,  the authors \cite{CMP94} established the global well-posedness of $(NS)$ with small initial data in the non-endpoint case $\dot{B}_{p,\infty}^{-1+\frac{3}{p}}$. Although the endpoint space $\dot{B}_{\infty,\infty}^{-1}$ is ill-posed \cite{BP08}, Koch and Tataru \cite{KT01} proved the global well-posedness of $(NS)$ with small initial data in the subspace $\mathrm{BMO}^{-1}$.

 It is natural to expect that if $\rho_0$ is close to a constant, then some of these well-posedness  results should also hold for the inhomogeneous system. Danchin \cite{D03} proved that if $\|u_0\|_{\dot{B}_{2,1}^{-1+\frac{d}{2}}}$ is sufficiently small and $\rho_0$ is close to a constant in the sense of $\dot{B}_{2,1}^{\frac{d}{2}}$, then \eqref{eqinns} admits a unique global solution. The existence result was extended to $\dot{B}_{p,1}^{-1+\frac{d}{p}}$ and $\dot{B}_{p,1}^{\frac{d}{p}}$ for $p \in (1, 2d)$, and the uniqueness result for $p \in (1, d]$ in \cite{A07}.  Corresponding results were extended to the half-space case in \cite{DM09}. The gap for the uniqueness with $p\in (d, 2d)$ was solved by Danchin and Mucha \cite{DM12} 
 in a more general functional framework.
The restriction on the initial density can in fact be relaxed and need not be confined to a small neighborhood of a constant. The study of such problems originated with Kazhikhov \cite{K74}, who in 1974 established global existence of strong solutions for densities bounded away from zero and initial velocity $u_0\in H^1$. Subsequently, \cite{AGZ12} proved local and global well-posedness for data $(u_0, \rho_0)$ with $u_0 \in \dot B_{2,1}^{\frac{1}{2}}$ and $\rho_0 \in \dot B_{2,1}^{\frac{3}{2}}$ possessing a positive lower bound. This result was later extended to the $\dot B_{p,1}^{-1+\frac{3}{p}}$ framework in \cite{AGZ13}; see also \cite{AGZ24,QZ21} for further developments. In \cite{PZZ13}, Paicu, Zhang, and Zhang obtained global existence in two dimensions for initial density in $L^\infty$ with a positive lower bound and initial velocity in $H^s(\mathbb{R}^2)$ for any $s>0$. In three dimensions, they proved global existence for $H^1(\mathbb{R}^3)$ velocity under the assumption that $\|u_0\|_{L^2}\|\nabla u_0\|_{L^2}$ is sufficiently small. The three-dimensional result was later improved to $H^{s}$ velocity for $s>\frac{1}{2}$ in \cite{CZZ16}. The first result allowing the initial density $\rho_0$ to be merely bounded (and bounded away from zero) while the velocity lies in a critical space is due to Zhang \cite{Z20}, who established global existence for small initial velocity in $\dot B_{2,1}^{\frac{1}{2}}$. Uniqueness for this class of solutions was subsequently proved in \cite{DW23}, where the authors also generalized the existence result to initial velocities in the larger critical spaces $\dot B_{p,1}^{-1+\frac{d}{p}}$. An interesting related direction shows that smallness of two velocity components  can also suffice for global existence; see \cite{CMZ14,DZ14, HPZ13, Z20} for details.

%All results above have the restriction for density to be away from $0$, since vacuum will lead to further difficulties. In 1990, Simon extended the global weak solution theories before in \cite{S90} and removed the request for lower bound of density. Lions \cite{PL96} further extended the results before to a density-dependent viscosity (with a renormalized density equation), and proved the existence for global weak solutions with $L^2$ data. In \cite{DM17}, Danchin and Mucha proved global existence for $H^1(\Omega)$ for $L^\infty$ density and some bounded domain $\Omega$. Recently, Hao et.al proved in \cite{HSWZ} for global existence for $L^\infty$ density and $L^2(\mathbb{R}^2)$ velocity in 2D, and corresponding result in 3D for initial velocity small in $H^{\frac{1}{2}}(\mathbb{R}^3)$. The result soon has been extended to the case when $\sqrt{\rho_0}u_0$ is small in $\dot B_{p,\infty}^{-1+\frac{3}{p}}$ in \cite{HSWZZ}.

All the aforementioned results require the density to be bounded away from zero, as the presence of vacuum introduces substantial additional difficulties. A significant advance was made by Simon \cite{S90} in 1990, who improved the results of \cite{K74} by removing the requirement of a positive lower bound for the density. Lions \cite{PL96} further generalized these results to the case of density-dependent viscosity (formulated via a renormalized continuity equation), establishing the existence of global weak solutions for initial density in $L^\infty$ and initial velocity  in $L^2$. Subsequently, Danchin and Mucha \cite{DM17} proved global existence for $L^\infty$ densities and velocities in $H^1(\Omega)$ on a bounded domain $\Omega$. More recently, Hao et al. \cite{HSWZ} obtained a global existence result in two dimensions for $L^\infty$ densities and velocities in $L^2(\mathbb{R}^2)$; in three dimensions, they required the initial velocity to be small in $\dot{H}^{\frac{1}{2}}(\mathbb{R}^3)$. This result was soon extended in \cite{HSWZZ}, where global existence was established under the smallness condition $\|\rho_0-1\|_{L^\infty}+\|\mathbb{P}(\rho_0u_0)\|_{\dot{B}_{p,\infty}^{-1+\frac{3}{p}}}\ll 1$ for $3<p<\infty$.

\subsection{Functional setting and main results}
We recall the $\mathrm{BMO}^{-1}$ setting as in \cite{KT01}. For $b>0$, we define the localized $\mathrm{BMO}_b^{-1}$ norm
\begin{equation*}\label{bmo}
\|u\|_{\mathrm{BMO}^{-1}_b} := \sup_{t \leq b^2} t^{\frac{1}{2}} \|e^{t\Delta} u\|_{L^\infty} + \sup_{x \in \mathbb{R}^3,\, 0 < R < b} R^{-\frac{3}{2}} \Bigl( \int_0^{R^2} \int_{B(x,R)} |e^{t\Delta} u(y)|^2 \, dy \, dt \Bigr)^{1/2}.
\end{equation*}
For the case $b = \infty$, the norm reduces to the $\mathrm{BMO}^{-1}$ norm, and we denote the corresponding norm simply by $\mathrm{BMO}^{-1}$. The space $\mathrm{VMO}^{-1}$ is defined as
\begin{equation*}
\mathrm{VMO}^{-1} := \bigl\{ f \in \mathcal{S}':\ \  \lim_{b \to 0} \|f\|_{\mathrm{BMO}^{-1}_b} = 0 \bigr\}.
\end{equation*}
For simplicity, we introduce the following norm for a function $u: [0,T] \times \mathbb{R}^3 \to \mathbb{R}$:
\begin{equation*}
\|u\|_{V_T^b} := \sup_{x \in \mathbb{R}^3,\, 0 < R < b} R^{-\frac{3}{2}} \Bigl( \int_0^{\min\{T, R^2\}} \int_{B(x,R)} |u(t,y)|^2 \, dy \, dt\Bigr)^{1/2},
\end{equation*}
which satisfies $\|e^{t\Delta} u_0\|_{V_\infty^b} \leq \|u_0\|_{\mathrm{BMO}^{-1}_b}$.

To capture both critical and subcritical behaviors, we introduce the norms
\begin{equation*}
\begin{aligned}
\|u\|_{X_T^1} &:= \sup_{t \leq T} t^{\frac{1}{2}} \|u(t)\|_{L^\infty} + \|u\|_{V_T^b} + \sup_{t \leq T} t^{1+\frac{\kappa}{2}} \|\nabla u(t)\|_{\dot{C}^{\kappa}} \\
&\quad + \sup_{s < t \leq T} \Bigl( s^{\frac{1}{2} + \varkappa_1} \frac{\|u(t) - u(s)\|_{L^\infty}}{(t-s)^{\varkappa_1}} \Bigr),
\end{aligned}
\end{equation*}
where $\varkappa_1 = \frac{18}{35} > \frac{1}{2}$ and $\kappa = \frac{1}{700}$, and
\begin{equation*}
\begin{aligned}
\|u\|_{X_T^2} &:= \sup_{t \leq T} \Bigl( t^{\frac{5}{12}} \|u(t)\|_{L^{12}} + \sup_{s < t \leq T} s^{\frac{5}{12} + \varkappa_2} \frac{\|u(t) - u(s)\|_{L^{12}}}{(t-s)^{\varkappa_2}} \Bigr),
\end{aligned}
\end{equation*}
with $\varkappa_2 = \frac{5}{7} > \varkappa_1$. The values of $\kappa$, $\varkappa_1$, and $\varkappa_2$ remain fixed throughout the paper. It is easy to verify that the norm $\|\cdot\|_{X_T^1}$ is critical under the scaling \eqref{S1eq1} when $b = \infty$, while $\|\cdot\|_{X_T^2}$ is a subcritical norm. We define the following mixed norm
\begin{equation}\label{defxt}
\|u\|_{X_T} := \|u\|_{X_T^1} + \|u\|_{X_T^2}
\end{equation}
for some $T \leq 1$. We will use the subcritical norm $\|\cdot\|_{X_T^2}$ to estimate the remainder term $u_R$ in \eqref{decop} below, for which we may gain nearly one additional derivative. Further details will be explained in the next subsection.

By the properties of the transport equation, if the Lipschitz norm of the convection velocity field is integrable in time, then the regularity of the solution is inherited from the initial data (see \cite{BCD}). However, in our case, the $L^\infty$ norm of the gradient to the velocity behaves like $t^{-1}$, which is not integrable, leading to a loss of space regularity in the density. In what follows, we will use the auxiliary norm
\begin{equation}\label{auxnorm}
\|u\|_{T} := \sup_{t \leq T} \left( t^{\frac{1}{2}} \|u(t)\|_{L^\infty} + t \|u(t)\|_{\dot{C}^1} \right)
\end{equation}
to control the flow map associated with $u$.

For the lower-order remainder terms that appear in the next subsection, we define another mixed norm without the $\mathrm{BMO}_b^{-1}$ structure, namely
\begin{equation*}
\begin{aligned}
\|u\|_{Y_T} &:= \sup_{t \leq T} \Big( t^{\frac{2}{5}} \|u(t)\|_{L^\infty} + t^{\frac{5}{12}} \|u(t)\|_{L^{12}} + t^{1+\frac{\kappa}{2}} \|\nabla u(t)\|_{\dot{C}^{\kappa}} \Big) \\
&\quad + \sup_{s < t \leq T} \Bigl( s^{\frac{1}{2} + \varkappa_1} \frac{\|u(t) - u(s)\|_{L^\infty}}{(t-s)^{\varkappa_1}} \Bigr) + \sup_{s < t \leq T} \Bigl( s^{\frac{5}{12} + \varkappa_2} \frac{\|u(t) - u(s)\|_{L^{12}}}{(t-s)^{\varkappa_2}} \Bigr).
\end{aligned}
\end{equation*}
Since $Y_T$ contains subcritical terms, it follows readily from the definition of the $\|\cdot\|_{V_T^b}$ norm that for any $T < 1$,
\begin{equation*}
\|u\|_{X_T} \lesssim \|u\|_{Y_T}.
\end{equation*}

Our first main result is stated as follows.

\begin{thm}\label{mainthm}
Let $\rho_0 \in C^1$ satisfy
\begin{equation}\label{bdrho0}
0 < C_1 < \rho_0 < C_2 < \infty
\end{equation}
for some positive constants $C_1$ and $C_2$. Assume that $u_0 \in L^2 \cap \mathrm{VMO}^{-1}$. There exist $b \in (0,1)$ and $\varepsilon_0 > 0$ such that if
\begin{equation*}
\|u_0\|_{\mathrm{BMO}^{-1}_b} \leq \varepsilon_0,
\end{equation*}
then system \eqref{eqinns} admits a solution $(\rho, u, P)$ on $[0, T(\|\rho_0\|_{C^1},u_0, \varepsilon_0, b)]$ satisfying
\begin{equation}\label{S1eqw}
\|u\|_{X_T} \leq C\epsilon_0^{\frac{5}{6}}\|u_0\|_{L^2}^{\frac{1}{6}}\|\rho_0\|_{C^1},
\end{equation}
and
\begin{equation}\label{estL2u}
\sup_{t \leq T} \left( \bigl\|\bigl(u(t), t^{\frac{1}{2}}\nabla u(t), t u_t(t)\bigr)\bigr\|_{L^2}+ \bigl\|\bigl(\na u, t^{\frac{1}{2}} u_t, t\na u_t\bigr)\bigr\|_{L^2_t(L^2)} \right) \leq C (\ep_0,\|u_0\|_{L^2}).
\end{equation}
\end{thm}

\begin{rmk}
Moreover, we note that $T$ can be chosen such that $C_1' b^2 \leq T \leq C_2' b^2$ for some universal constants $C_1'$ and $C_2'$ (in the sequel, we write $T \sim b^2$). This scaling is consistent with the parabolic nature of the equations.
\end{rmk}

\begin{rmk}
At present, establishing the uniqueness of solutions under the assumptions that $\rho_0 \in C^1$ has a positive lower bound, $u_0 \in L^2 \cap \mathrm{VMO}^{-1}$, and the quantity $\|\rho_0 - 1\|_{L^\infty} + \|u_0\|_{\mathrm{BMO}^{-1}}$ is sufficiently small remains a significant open problem. As noted in \cite[Remark 1.2]{HSWZZ}, the methodology employed in that paper cannot be adapted to prove weak-strong uniqueness in the $\mathrm{BMO}^{-1}$ setting. The principal difficulty arises from the lack of an effective decomposition $f = f_1 + f_2$ for a function $f \in L^2 \cap \mathrm{BMO}^{-1}$, where $f_1$ belongs to a suitable subcritical space. It is worth mentioning that weak-strong uniqueness can be established under the stronger assumption $u_0 \in L^p \cap \mathrm{BMO}^{-1}$ for some $p > 2$. In this case, one can approximate $u_0$ by a sequence $\{u_{0,N}\} \subset L^q$ with $q > 3$ such that $\|u_0 - u_{0,N}\|_{L^2} \to 0$, allowing the application of the method from \cite{HSWZZ}. This argument, however, fails for the critical case $p = 2$.

For the classical Navier–Stokes equations, uniqueness for small initial data in $\mathrm{BMO}^{-1}$ was proved by Koch and Tataru \cite{KT01}. Their approach, which relies on a fixed-point argument that simultaneously yields global existence and uniqueness, can not be extended to our inhomogeneous setting, as our existence proof does not imply uniqueness. To the best of our knowledge, even in the classical Navier–Stokes setting, weak-strong uniqueness for $u_0 \in L^2 \cap \mathrm{BMO}^{-1}$ (without a smallness condition) remains a major open problem. For recent progress, we refer to \cite{Barker2018, Chemin2011, DZ2014}. Notably, Coiculescu and Palasek \cite{CP2025} have very recently proved the non-uniqueness of solutions to the classical Navier–Stokes equations for large initial data in $\mathrm{BMO}^{-1}$.
\end{rmk}

By assuming that the initial density is nearly constant and the initial velocity is sufficiently small in $\mathrm{BMO}^{-1}$, we obtain the following global existence result.

\begin{thm}\label{thm.global-existence}
Let $M > 0$. Assume that $u_0 \in L^2$ and $\rho_0 \in C^2$ satisfy $\|u_0\|_{L^2} + \|\rho_0\|_{C^2} \leq M$. There exists a constant $\varepsilon_2 \in (0, 1/2)$, depending only on $M$, such that if
\begin{equation}\label{glocond}
\|\rho_0 - 1\|_{L^\infty} + \|u_0\|_{\mathrm{BMO}^{-1}} < \varepsilon_2,
\end{equation}
then \eqref{eqinns} admits a global strong solution.
\end{thm}
\begin{rmk}
    Here in Theorem \ref{thm.global-existence}, we require $\rho_0\in C^2$ to have higher regularity. This is because in the a priori estimate of Theorem \ref{mainthm}, the maximal existence time $T$ depends on $\ep_0$ even if $\ep_0$ is small. More explanation can be seen in Section \ref{sec7} and Remark \ref{rmk6.1}.
\end{rmk}

\begin{rmk}
A notable aspect of Theorem \ref{thm.global-existence} is that the required smallness of the initial data is measured in terms of the norms $\|u_0\|_{L^2}$ and $\|\rho_0\|_{C^2}$. Whether an ``unconditional'' global existence result can be established remains an open problem.
\end{rmk}

\subsection{Ideas and outline}\label{outline}
We explain the main ideas of the proof and the motivation behind the choice of the norm $\|u\|_{X_T}$.

We apply the freezing coefficient method developed in \cite{KHN24}. For any fixed $x_0$, applying the Leray projection operator $\mathbb{P} = \mathrm{Id} - \nabla \Delta^{-1} \nabla \cdot$ to the momentum equation in \eqref{eqinns} and freezing the coefficient at $x_0$, we obtain
\begin{equation*}
\left\{
\begin{array}{l}
\partial_t \rho + u \cdot \nabla \rho = 0,\\[4pt]
\rho_0(x_0) \partial_t u - \Delta u = -\mathbb{P} \operatorname{div}(\rho u \otimes u) - \mathbb{P} \partial_t \bigl[(\rho - \rho_0) u\bigr] \\
\qquad\qquad\qquad\qquad+ \nabla \Delta^{-1}\bigl( \nabla \rho_0 \cdot \partial_t u\bigr) - (\rho_0 - \rho_0(x_0)) u_t,\\[4pt]
\nabla \cdot u = 0,\\[4pt]
(\rho, u)|_{t=0} = (\rho_0, u_0),
\end{array}
\right.
\end{equation*}
where $\rho_0 = \rho_0(y)$ denotes the initial density, and we have used the notation $\rho_0 = \rho_0(y)$ in the convolution kernels below.

Let $\varrho := \rho - \rho_0$ and let $K_{x_0}(t,x)$ be the fundamental solution to
\begin{equation}\label{defKx0}
\left\{
\begin{array}{l}
\rho_0(x_0) \partial_t K_{x_0}(t,x) - \Delta K_{x_0}(t,x) = 0,\\[4pt]
K_{x_0}(t,x)|_{t=0} = \delta.
\end{array}
\right.
\end{equation}
With this definition, we can write
\begin{equation}\label{decop}
\begin{aligned}
u(t) &= K_{x_0}(t) \ast u_0 - \int_0^t K_{x_0}(t-\tau) \ast \mathbb{P} \operatorname{div}(\rho u \otimes u)(\tau) \, d\tau \\
&\quad + \int_0^t K_{x_0}(t-\tau) \ast \nabla \Delta^{-1} \partial_\tau (\nabla \rho_0 \cdot u)(\tau) \, d\tau \\
&\quad - \int_0^t K_{x_0}(t-\tau) \ast \mathbb{P} \partial_\tau (\varrho u)(\tau) \, d\tau \\
&\quad - \int_0^t \int K_{x_0}(t-\tau, x-y) (\rho_0(y) - \rho_0(x_0)) \partial_\tau u(\tau, y) \, dy \, d\tau \\
&=: u_{L,x_0} + u_{N,x_0} + \sum_{i=1}^3 u_{Ri,x_0},
\end{aligned}
\end{equation}
for any $x_0 \in \mathbb{R}^3$. Since \eqref{decop} holds for every $x_0 \in \mathbb{R}^3$, it holds in particular for $x_0 = x$. If we define $\tilde{u}_{\star}(t,x) = u_{\star,x_0}(t,x)|_{x_0=x}$ for $\star \in \{L, N, R1, R2, R3\}$, then
\begin{equation}\label{deftilu}
u(t,x) = \sum_{\star = L, N, R1, R2, R3} \tilde{u}_{\star}(t,x).
\end{equation}
With a slight abuse of notation, we will use $u_R$ and $\tilde{u}_R$ to denote any of the components $u_{R1,x_0}, u_{R2,x_0}, u_{R3,x_0}, \tilde{u}_{R1}, \tilde{u}_{R2}, \tilde{u}_{R3}$ in the sequel. From \eqref{decop}, we observe that $u_{L,x_0}$ and $u_{N,x_0}$ are the critical terms, while the remainder terms $u_R$ enjoy supercritical regularity due to the extra smoothing from $\rho_0$ and $\varrho$.

Here and throughout, we denote $\delta_\alpha f(t,x) := f(t,x+\alpha) - f(t,x)$ for any $\alpha \in \mathbb{R}^3$, and $\delta_a^t f(t,x) := f(t+a,x) - f(t,x)$ for $a \in \mathbb{R}$. A key observation is that
\begin{equation*}
(\delta_\alpha \nabla u_{R3})|_{x_0=x} \neq \delta_\alpha \nabla (u_{R3}|_{x_0=x}).
\end{equation*}
From \eqref{decop}, we obtain the following estimates:
\begin{equation}\label{estnorm}
\begin{aligned}
&\|u(t)\|_{L^p} \lesssim \sum_{\star = L, N, R1, R2, R3} \|\tilde{u}_{\star}(t)\|_{L^p}, \quad \forall p \in [1,\infty],\\
&\|\na|D|^{\kappa} u(t)\|_{L^\infty} \lesssim \sum_{\star = L, N, R1, R2, R3} \| (\na|D|^{\kappa} u_{\star,x_0})|_{x_0=x} \|_{L^\infty},\\
&\frac{\|u(t) - u(s)\|_{L^p}}{(t-s)^\alpha} \lesssim \sum_{\star = L, N, R1, R2, R3} \frac{\|\tilde{u}_{\star}(t) - \tilde{u}_{\star}(s)\|_{L^p}}{(t-s)^\alpha}.
\end{aligned}
\end{equation}
Here and below, we always denote $|D|^{\alpha}$ to be Fourier multiplier with symbol $|\xi|^{\alpha}.$

For the $V_T^b$ norm, note that for any fixed $z$, \eqref{decop} yields
\begin{equation*}
|u(t,x)| \lesssim \sum_{\star = L, N, R1, R2, R3} |u_{\star,z}(t,x)|,
\end{equation*}
which implies
\begin{equation*}
\int_0^{\min\{T,R^2\}} \int_{B(z,R)} |u(t,x)|^2 \, dx \, d\tau \lesssim \sum_{\star = L, N, R1, R2, R3} \int_0^{\min\{T,R^2\}} \int_{B(z,R)} |u_{\star,z}(t,x)|^2 \, dx \, d\tau
\end{equation*}
for any $R < b$. Consequently,
\begin{equation*}
\|u\|_{V_T^b} \lesssim \sup_z \sum_{\star = L, N, R1, R2} \|u_{\star,z}\|_{V_T^b} + \sup_{z \in \mathbb{R}^3} \sup_{R < b} R^{-\frac{3}{2}} \Bigl( \int_0^{\min\{T,R^2\}} \int_{B(z,R)} |u_{R3,z}(t,x)|^2 \, dx \, dt \Bigr)^{\frac{1}{2}}.
\end{equation*}

The estimates for the density $\rho$ established in Section \ref{sec2} allow us to complete the proof, provided that we control the norm $\|u\|_T$, and in particular the $C^1$ component. According to the decomposition \eqref{decop}, the solution $u$ comprises three parts: the linear term from the initial data $u_{L,x_0}$, the bilinear term $u_{N,x_0}$, and the remainder $u_{R,x_0}$. The terms $u_{L,x_0}$ and $u_{N,x_0}$ constitute the critical components of the analysis. To establish the required a priori bounds for $\|u\|_{L^\infty}$ and $\|u\|_{\dot{C}^{1+\kappa}}$, we employ the strategy introduced by Koch and Tataru. This analysis is carried out in Appendix \ref{sub7.2}.

The term $u_{R2}$ can be expressed as
\begin{equation*}
u_{R2} = \int_0^t \int \partial_t K_{x_0}(t-\tau, x-y) \mathbb{P}(\varrho u)(\tau, y) \, dy \, d\tau + \mathbb{P}(\varrho u)(t,x).
\end{equation*}
Although the first integral can be estimated by making a difference, the second term—despite its favorable temporal decay—poses difficulties due to its limited spatial regularity. To circumvent this issue, we adapt the method from \cite{KHN}, specifically Lemma \ref{CVPDE}, by introducing time-weighted Hölder norms in $X_T^1$ and $X_T^2$. This allows us to compensate for the spatial regularity loss from $\rho$ by fully exploiting the available time regularity.

A second challenge lies in estimating the time-weighted Hölder norm of
\begin{equation*}
u_{R1} = \int_0^t \int \partial_t K_{x_0}(t-\tau, x-y) \nabla \Delta^{-1} (\nabla \rho_0 \cdot u)(\tau, y) \, dy \, d\tau + \nabla \Delta^{-1} (\nabla \rho_0 \cdot u)(t,x),
\end{equation*}
particularly the latter term. To address this, we employ energy estimates (developed in Appendix \ref{sec3}) to obtain
\begin{equation*}
\sup_{t \leq T} t \|u_t(t)\|_{L^2} < \infty,
\end{equation*}
which, via interpolation, supplies sufficient regularity for the desired bound.

The inclusion of the $\|u(t)\|_{L^{12}}$ norm in our functional setting is motivated by the fact that the projection $\mathbb{P}$ does not preserve $L^\infty$ but is bounded on $L^p$ for any $p \in (1, \infty)$. All newly introduced norms for $u_{N,x_0}$ and $u_{L,x_0}$ can be controlled using the energy inequality in Lemmas \ref{eng1}-\ref{eng3} together with standard interpolation.

The rest of the paper is organized as follows. Section \ref{sec2} establishes estimates for the transport equation governing the density. Sections \ref{estcri} and \ref{sec5} provide the analysis for the critical terms and the remainder terms, respectively. Section \ref{sec6} synthesizes these a priori estimates and establishes the local existence of a solution, namely, Theorem \ref{mainthm}. Section \ref{sec7} extends the result to the global case, namely, Theorem \ref{thm.global-existence}. Finally, Appendix \ref{append} presents several weighted $L^2$ energy estimates and proves a bilinear estimate of Koch–Tataru type.

\section{Preliminaries}\label{sec2}
In this section, we derive a priori estimates for the density $\rho$, which satisfies the following transport equation:
\begin{equation}\label{eqtrans}
\begin{cases}
\partial_t \rho(t,x) + (u \cdot \nabla \rho)(t,x) = 0,\\[4pt]
\rho|_{t=0}(x) = \rho_0(x),
\end{cases}
\end{equation}
where $u : \mathbb{R}_+ \times \mathbb{R}^3 \to \mathbb{R}^3$ is assumed to be divergence-free and Lipschitz for any $t > 0$. Below we always assume that
\begin{equation}\label{S2eq1}
\|\nabla \rho_0\|_{L^\infty} < \infty \quad \text{and} \quad
\|u\|_T \leq \varepsilon_1,
\end{equation}
where $\|\cdot\|_T$ is defined in \eqref{auxnorm}, and $\varepsilon_1$ is a small constant to be determined later.

We first establish the spatial Hölder regularity of $\rho$ for small times.

\begin{prop}\label{regdiff}
Let $\rho$ be a smooth enough solution of \eqref{eqtrans}. Then, under condition \eqref{S2eq1}, for any $T \leq 1$ and $\gamma \in (0,1)$, there exists $\varepsilon_1 > 0$ depending only on $\gamma$ such that if $\|u\|_T \leq \varepsilon_1$, one has
\begin{equation}\label{S2eqa}
\sup_{t \leq T} \|\rho(t)\|_{C^\gamma} \leq C \|\rho_0\|_{C^1},
\end{equation}
where $C > 1$ is a constant depending only on $\gamma$.
\end{prop}

\begin{proof}
Let $X(t,x)$ be the trajectory map associated with $u$, defined by
\begin{equation}\label{Eq.trajectory}
\begin{aligned}
\partial_t X(t,x) = u(t, X(t,x)), \quad X(0,x) = x.
\end{aligned}
\end{equation}
Then the transport equation implies that
\begin{equation}\label{transol}
\rho(t, X(t,x)) = \rho_0(x), \quad \text{i.e.,} \quad \rho(t,x) = \rho_0(X_t^{-1}(x)),
\end{equation}
where $X_t : \mathbb{R}^3 \to \mathbb{R}^3$ is defined by $X_t(x) = X(t,x)$.

Note that
\begin{equation*}
-\|\nabla u(t)\|_{L^\infty} |X(t,x) - X(t,y)| \leq \frac{d}{dt} |X(t,x) - X(t,y)| \leq \|\nabla u(t)\|_{L^\infty} |X(t,x) - X(t,y)|.
\end{equation*}
By Grönwall's lemma, we obtain
\begin{equation}\label{IFM}
|X(s,x) - X(s,y)| \leq \exp\left( \int_s^t \|\nabla u(\tau)\|_{L^\infty} \, d\tau \right) |X(t,x) - X(t,y)|, \quad \forall s < t.
\end{equation}
Moreover, from the definition \eqref{auxnorm} of $\|u\|_T$, we have
\begin{equation*}
\int_s^t \|\nabla u(\tau)\|_{L^\infty} \, d\tau \leq C\|u\|_T \log\left( \frac{t}{s} \right),
\end{equation*}
which together with \eqref{IFM} yields
\begin{equation}\label{diffrho}
|X(s,x) - X(s,y)| \leq \exp\Bigl( C\|u\|_T \log\bigl( \frac{t}{s} \bigr) \Bigr) |X(t,x) - X(t,y)|, \quad \forall s < t.
\end{equation}    

	Now we use this inequality to derive the Hölder estimate of $\rho$. Let $x, y \in \mathbb{R}^3$. Then there exist $\overline{x}, \overline{y} \in \mathbb{R}^3$ such that
\begin{equation*}
\begin{aligned}
&\rho(t,x) = \rho_0(\overline{x}), \quad x = X(t, \overline{x}),\\
&\rho(t,y) = \rho_0(\overline{y}), \quad y = X(t, \overline{y}).
\end{aligned}
\end{equation*}
Consequently, for any $\tau < t$, by \eqref{Eq.trajectory} and \eqref{transol} we have
\begin{equation*}
\begin{aligned}
|\rho(t,x) - \rho(t,y)| &\leq \|\nabla \rho_0\|_{L^\infty} |\overline{x} - \overline{y}| \\
&\leq \|\nabla \rho_0\|_{L^\infty} |X(\tau, \overline{x}) - X(\tau, \overline{y})| + 2 \|\nabla \rho_0\|_{L^\infty} \int_0^\tau \|u(\tau)\|_{L^\infty} \, d\tau \\
&\leq \|\nabla \rho_0\|_{L^\infty} |X(\tau, \overline{x}) - X(\tau, \overline{y})| + 2 \|\nabla \rho_0\|_{L^\infty} \|u\|_T \tau^{\frac{1}{2}}.
\end{aligned}
\end{equation*}
For any $\gamma \in (0,1)$ and $t \in [0,T]$, we take $\tau = t \frac{|x-y|^{2\gamma}}{1 + 2|x-y|^{2\gamma}}$. Then it follows from \eqref{diffrho} and assumption \eqref{S2eq1} that
\begin{equation*}
\begin{aligned}
|\rho(t,x) - \rho(t,y)| &\leq \|\nabla \rho_0\|_{L^\infty} \exp\left( \log\Bigl(2 + \frac{1}{|x-y|^{2\gamma}}\Bigr) \|u\|_T \right) |X(t, \overline{x}) - X(t, \overline{y})| \\
&\quad + 2\sqrt{t} \|\nabla \rho_0\|_{L^\infty} \|u\|_T |x-y|^\gamma \\
&\leq \|\nabla \rho_0\|_{L^\infty} \left( \exp\bigl(-C_0 \gamma \varepsilon_1 \log(|x-y|)\bigr) |x-y| + 2\sqrt{t} \varepsilon_1 |x-y|^\gamma \right) \\
&\lesssim \|\nabla \rho_0\|_{L^\infty} |x-y|^\gamma, \quad \forall |x-y| < 1/2,
\end{aligned}
\end{equation*}
for some universal $C_0$, provided we choose $\varepsilon_1$ sufficiently small such that $1 - C_0 \gamma \varepsilon_1 \geq \gamma$. Note that a direct corollary of \eqref{transol} is
\begin{equation}\label{lowerbd}
\|\rho(t)\|_{L^\infty} \leq \|\rho_0\|_{L^\infty}, \quad \forall t \geq 0.
\end{equation}
This completes the proof \eqref{S2eqa}.
\end{proof}

Furthermore, it is important to establish the time regularity of $\rho$. In the remainder of this paper, we denote
\begin{equation}\label{eq3.1}
|D|^\eta f(x) := (-\Delta)^{\frac{\eta}{2}} f(x) = \mathcal{F}^{-1}(|\cdot|^\eta \widehat{f}(\cdot))(x) = C_{d,\eta} \int \frac{\delta_\alpha f(x)}{|\alpha|^{d+\eta}} \, d\alpha, \quad \forall \eta \in (0,2),
\end{equation}
noting that
\[
\int_{\mathbb{R}^d} \frac{1 - e^{i\xi\cdot\alpha}}{|\alpha|^{d+\eta}} \, d\alpha = C_{d,\kappa} |\xi|^{\eta}.
\]
Since for $\kappa \in (0,1)$ we have $\|u\|_{\dot{C}^{\kappa}} = \|u\|_{\dot{B}_{\infty,\infty}^{\kappa}}$, Proposition \ref{regdiff} implies that
\begin{equation*}
\||D|^\eta \rho\|_{L_T^\infty L^\infty} < \infty, \quad \forall \eta \in (0,1).
\end{equation*}
We next quantify the time regularity of fractional derivatives of $\rho$.

\begin{prop}\label{Holtrho}
Let $\alpha \in (0,1)$ and $0 \leq \alpha' < \alpha$. Assume that $\rho$ is the solution of \eqref{eqtrans} with $\||D|^{\alpha}\rho\|_{L_T^\infty L^\infty} + \|u\|_T < \infty$. Then for any $0 < s < t < T$, we have
\begin{equation*}
\||D|^{\alpha'}\rho(t) - |D|^{\alpha'}\rho(s)\|_{L^\infty} \lesssim \||D|^{\alpha}\rho\|_{L_T^\infty L^\infty} (1 + \|u\|_T) \min\Bigl\{ (t-s)^{\frac{\alpha-\alpha'}{2}}, s^{-\frac{\alpha-\alpha'}{2}} (t-s)^{\alpha-\alpha'} \Bigr\}.
\end{equation*}
\end{prop}

\begin{proof}
Let $\psi \in \mathcal{S}(\mathbb{R}^3)$ be a standard mollifier. Denote $\psi_\varepsilon = \varepsilon^{-3} \psi(\cdot/\varepsilon)$ and $\rho_\varepsilon = \rho \ast \psi_\varepsilon$. Then
\begin{align*}
|D|^{\alpha'}\rho(t) - |D|^{\alpha'}\rho(s) = &\bigl( |D|^{\alpha'}\rho(t) - |D|^{\alpha'}\rho_\varepsilon(t) \bigr) \\
&- \bigl( |D|^{\alpha'}\rho(s) - |D|^{\alpha'}\rho_\varepsilon(s) \bigr) + \bigl( |D|^{\alpha'}\rho_\varepsilon(t) - |D|^{\alpha'}\rho_\varepsilon(s) \bigr).
\end{align*}
For the mollification error, using standard inequalities, we obtain
\begin{equation*}
\begin{aligned}
||D|^{\alpha'}\rho(x) - |D|^{\alpha'}\rho_\varepsilon(x)| &= \left| \int \bigl( |D|^{\alpha'}\rho(x) - |D|^{\alpha'}\rho(y) \bigr) \psi_\varepsilon(x-y) \, dy \right| \\
&\lesssim \||D|^\alpha \rho\|_{L_T^\infty L^\infty} \| |\cdot|^{\alpha-\alpha'} \phi_\varepsilon(\cdot) \|_{L^1} \lesssim \||D|^\alpha \rho\|_{L_T^\infty L^\infty} \varepsilon^{\alpha-\alpha'}.
\end{aligned}
\end{equation*}
Moreover, the mollified function $\rho_\varepsilon$ satisfies
\begin{equation*}
\partial_t \rho_\varepsilon + \psi_\varepsilon \ast (u \cdot \nabla \rho_\varepsilon) + \psi_\varepsilon \ast \nabla \cdot \bigl( u (\rho - \rho_\varepsilon) \bigr) = 0,
\end{equation*}
hence
\begin{equation*}
\begin{aligned}
||D|^{\alpha'}\rho_\varepsilon(t) - |D|^{\alpha'}\rho_\varepsilon(s)|
&\lesssim \Bigl ||D|^{\alpha'} \psi_\varepsilon \ast \int_s^t (\nabla \rho_\varepsilon \cdot u)(\tau) \, d\tau \Bigr|
+ \Bigl| |D|^{\alpha'} \nabla \psi_\varepsilon \ast \int_s^t (\rho - \rho_\varepsilon) u(\tau) \, d\tau \Bigr| \\
&\lesssim \||D|^\alpha \rho\|_{L^\infty} \|u\|_T \int_s^t \varepsilon^{\alpha-\alpha'-1} \tau^{-\frac{1}{2}} \, d\tau \\
&\lesssim \||D|^\alpha \rho\|_{L^\infty} \|u\|_T \, \varepsilon^{\alpha-\alpha'-1} \min\bigl\{ s^{-\frac{1}{2}} (t-s), (t-s)^{\frac{1}{2}} \bigr\},
\end{aligned}
\end{equation*}
where we used the estimate
\begin{equation*}
|\nabla \rho_\varepsilon| \lesssim ||D|^\alpha \rho \ast \nabla |D|^{-\alpha} \psi_\varepsilon| \lesssim \||D|^\alpha \rho\|_{L^\infty} \varepsilon^{-1+\alpha}.
\end{equation*}
Optimizing with respect to $\varepsilon$ by taking $\varepsilon = \min\bigl\{ s^{-\frac{1}{2}} (t-s), (t-s)^{\frac{1}{2}} \bigr\}$ yields the desired estimate.
\end{proof}

For any $\eta > -d$, we call $\mathcal{P}^{\eta}$ an $\eta$-order Fourier multiplier if $\mathcal{P}^{\eta} f = \mathcal{F}^{-1}(P^{\eta} \widehat{f})$, where $P^{\eta}$ is an $\eta$-order homogeneous function. We list some properties of the heat kernel $K_{x_0}(t,x)$ in the following.

\begin{prop}\label{S2prop3}
Assume that $0 < C_1 < \rho_0(x) < C_2 < \infty$. Let $K_{x_0}(t,x)$ be determined by \eqref{defKx0}. Then we have the pointwise estimate
\begin{equation}\label{esthtk}
\sup_{x_0} |\partial_t^{m} \mathcal{P}^{\eta} K_{x_0}(t,x)| \leq \frac{C_{m,\eta}(C_1,C_2)}{(\sqrt{t} + |x|)^{d + 2m + \eta}}, \quad \forall m \in \mathbb{N}, \, \eta > -d.
\end{equation}
Furthermore, we have the time Hölder estimate
\begin{equation}\label{htkint}
\sup_{x_0} \| \delta_a^t \partial_t^{m} \mathcal{P}^{\eta} K_{x_0}(t,x) \|_{L_x^1} \leq C_{m,\eta}(C_1,C_2) t^{-m - \frac{\eta}{2}} \min\left\{ 1, \frac{a}{t} \right\}, \quad \forall m \in \mathbb{N}, \, \eta > 0.
\end{equation}
\end{prop}

\begin{proof}
The estimate \eqref{esthtk} follows from a classical decomposition into low and high frequencies; details can be found in \cite{KHN24}. For \eqref{htkint}, we observe that
\begin{equation*}
\delta_a^t \partial_t^{m} \mathcal{P}^{\eta} K_{x_0}(t,x) = \int_t^{t+a} \partial_t^{m+1} \mathcal{P}^{\eta} K_{x_0}(s,x) \, ds,
\end{equation*}
and then \eqref{htkint} follows directly from \eqref{esthtk}.
\end{proof}    

    \section{Estimates for critical terms}\label{estcri}
    In this section, we establish the main estimates for $u_{L,x_0}$ and $u_{N,x_0}$ defined in \eqref{decop}, following the approach of Koch and Tataru in \cite{KT01}.

We will frequently use the following property, whose proof can be found in Section 2.4 of \cite{BCD}:
\begin{equation*}
\|u\|_{\dot{C}^{\kappa}} \sim \|u\|_{\dot{B}_{\infty,\infty}^{\kappa}} \sim \||D|^{\kappa} u\|_{\dot{B}_{\infty,\infty}^{0}} \lesssim \||D|^{\kappa} u\|_{L^\infty}, \quad \forall \kappa \in (0,1).
\end{equation*}

For the linear term $u_{L,x_0}$, by the properties of the heat kernel, we have the following proposition:

\begin{prop}\label{uL}
Let $u_0 \in L^2$ satisfy $\|u_0\|_{\mathrm{BMO}_b^{-1}} \leq \varepsilon_0$. Then there exists $T \sim b^2$ such that for $\tilde{u}_L$ defined in \eqref{deftilu}, the following estimates hold:
\begin{equation}\label{estuLx}
\begin{aligned}
&\sup_{x_0} \Bigl( \sup_{t \leq T} t^{\frac{1}{2}} \|u_{L,x_0}(t)\|_{L^\infty} + \|u_{L,x_0}\|_{V_T^b} + \sup_{s < t < T} s^{\frac{1}{2} + \varkappa_1} \frac{\|u_{L,x_0}(t) - u_{L,x_0}(s)\|_{L^\infty}}{(t-s)^{\varkappa_1}} \Bigr) \\
&\quad + \sup_{t \leq T} t^{1+\frac{\kappa}{2}} \| (\nabla D^{\kappa} u_{L,x_0})|_{x_0=x}(t) \|_{L^\infty}
+ \sup_{t \leq T} t^{\frac{5}{12}} \|\tilde{u}_L(t)\|_{L^{12}} \\
&\quad + \sup_{s < t < T} s^{\frac{5}{12} + \varkappa_2} \frac{\|\tilde{u}_L(t) - \tilde{u}_L(s)\|_{L^{12}}}{(t-s)^{\varkappa_2}} \\
&\lesssim \|u_0\|_{\mathrm{BMO}_b^{-1}} + \|u_0\|_{\mathrm{BMO}_b^{-1}}^{\frac{5}{6}} \|u_0\|_{L^2}^{\frac{1}{6}}.
\end{aligned}
\end{equation}
\end{prop}

\begin{proof} We first deduce from \eqref{defKx0} that $\wh{K}_{x_0}(t,\xi)=e^{-\f{t}{\rho_0(x_0)}|\xi|^2},$ from which, we infer
\begin{equation}\label{S3eq5}
\begin{aligned}
\sup_{t \leq T} t^{\frac{1}{2}} \|u_{L,x_0}(t)\|_{L^\infty}
=&\sqrt{\rho_0(x_0)}\sup_{t \leq T} \left( \Bigl(\frac{t }{\rho_0(x_0)}\Bigr)^{\f12}\bigl\|e^{\f{t}{\rho_0(x_0)}\D}u_0\bigr\|_{L^\infty}\right)\\
\leq & \sqrt{\rho_0(x_0)}\|u_0\|_{\mathrm{BMO}_b^{-1}},\quad \mbox{if}\ \ T\leq \rho(x_0)b^2,
\end{aligned}
\end{equation}
and
\begin{align*}
\sup_{t \leq T}& t^{1+\frac{\kappa}{2}} \| (|D|^{1+\kappa} u_{L,x_0})(t) \|_{L^\infty} \\
=&\rho_0^{1+\f\kappa2}(x_0)\sup_{t \leq T} \left(\Bigl(\f{t}{\rho_0(x_0)}\Bigr)^{\f12} \Bigl(\f{t}{\rho_0(x_0)}\Bigr)^{\frac{1+\kappa}{2}} \bigl\| (|D|^{1+\kappa}e^{\f{t}{2\rho_0(x_0)}\D} e^{\f{t}{2\rho_0(x_0)}\D}u_0\bigr\|_{L^\infty} \right)\\
\lesssim &\rho_0^{1+\f\kappa2}(x_0)\sup_{t \leq T} \left(\Bigl(\f{t}{\rho_0(x_0)}\Bigr)^{\f12} \bigl\|e^{\f{t}{2\rho_0(x_0)}\D}u_0\bigr\|_{L^\infty}  \right)\\
\lesssim &\rho_0^{1+\f\kappa2}(x_0)\|u_0\|_{\mathrm{BMO}_b^{-1}},\quad \mbox{if}\ \ T\leq 2\rho(x_0)b^2. \end{align*}
{ Observe that for $t\geq s$, if $t-s\leq \frac{s}{2}$, then
\begin{equation*}\label{S3eq4}
    \begin{aligned}
        \left\|e^{\frac{t}{\rho_0(x_0)}\Delta}u_0- e^{\frac{s}{\rho_0(x_0)}\Delta}u_0\right\|_{L^\infty}&=
        \Bigl\|\int_{\frac{s}{2}}^{t-\frac{s}{2}}\partial_\tau \left(e^{\frac{\tau}{\rho_0(x_0)}\Delta}\right)e^{\frac{s}{2\rho_0(x_0)}\Delta}u_0d\tau\Bigr\|_{L^\infty}\\
        &\lesssim \frac{1}{\rho_0(x_0)}\int_{\frac{s}{2}}^{t-\frac{s}{2}}\bigl\|e^{\frac{\tau}{\rho_0(x_0)}\Delta}\Delta e^{\frac{s}{2\rho_0(x_0)}\Delta}u_0\bigr\|_{L^\infty}d\tau\\
        &\lesssim s^{-1}(t-s)\|e^{\frac{s}{4\rho_0(x_0)}\Delta}u_0\|_{L^\infty}\\
        &\lesssim \sqrt{\rho_0(x_0)}s^{-\frac{3}{2}}(t-s)\|u_0\|_{BMO^{-1}_b},\quad \mbox{if}\ \ T\leq 4\rho(x_0)b^2.
    \end{aligned}
\end{equation*}
While it follows from \eqref{S3eq5} that
\begin{equation*}
    \left\|e^{\frac{t}{\rho_0(x_0)}\Delta}u_0- e^{\frac{s}{\rho_0(x_0)}\Delta}u_0\right\|_{L^\infty}\lesssim \sqrt{\rho_0(x_0)}s^{-\frac{1}{2}}\|u_0\|_{BMO_b^{-1}},\quad \mbox{if}\ \ s\leq t\leq T\leq \rho(x_0)b^2.
\end{equation*}
As a result, it comes out
\begin{equation}\label{S3eq6}
    \left\|e^{\frac{t}{\rho_0(x_0)}\Delta}u_0- e^{\frac{s}{\rho_0(x_0)}\Delta}u_0\right\|_{L^\infty}\lesssim \sqrt{\rho_0(x_0)}\min\{1,s^{-1}(t-s)\}s^{-\frac{1}{2}}\|u_0\|_{BMO_b^{-1}}.
\end{equation}
For the case $t-s\geq \frac{s}{2}$, \eqref{S3eq6} follows from \eqref{S3eq5}.}

By virtue of \eqref{S3eq6}, we deduce that
\begin{equation*}\label{S3eq11}
\begin{split}
\sup_{s < t \leq T}  s^{\frac{1}{2} + \varkappa_1} \frac{\|u_{L,x_0}(t) - u_{L,x_0}(s)\|_{L^\infty}}{(t-s)^{\varkappa_1}}  \lesssim& \f{\min\{1,s^{-1}(t-s)\}}{(t-s)^{\varkappa_1}}  s^{\varkappa_1}\|u_0\|_{BMO_b^{-1}}\\
\lesssim &\|u_0\|_{BMO_b^{-1}},
\end{split}
\end{equation*}
where we used the fact that if $t-s\leq s,$ there holds
\[ \f{\min\{1,s^{-1}(t-s)\}}{(t-s)^{\varkappa_1}}  s^{\varkappa_1}\leq (t-s)^{-1+\varkappa_1}
s^{1-\varkappa_1}\leq 1\quad \forall \varkappa_1\in (0,1),
 \]
 whereas if $t-s\geq s,$  one has
\[ \f{\min\{1,s^{-1}(t-s)\}}{(t-s)^{\varkappa_1}}  s^{\varkappa_1}\leq (t-s)^{-\varkappa_1}
s^{\varkappa_1}\leq 1\quad \forall \varkappa_1\in (0,1).
 \]

 For $\|u_{L,x_0}\|_{V_T^b},$ we get, by using changes of variable, that
    \begin{equation*}
        \begin{aligned}
         \|u_{L,x_0}\|_{V_T^b}=   &R^{-\frac{3}{2}}\Bigl(\int_0^{\min\{T,R^2\}}\int_{B(x,R)}\bigl|e^{\frac{t}{\rho_0(x_0)}\Delta}u_0(y)\bigr|^2\,dy\,dt\Bigr)^{\f12}\\
         \lesssim &\rho_0(x_0)R^{-\frac{3}{2}}\Bigl(\int_0^{\min\bigl\{\frac{T}{\rho_0(x_0)}, \frac{R^2}{\rho_0(x_0)}\bigr\}}\int_{B(x,R)}\bigl|e^{s\Delta}u_0(y)\bigr|\,dy\,ds\Bigr)^{\f12}\\\
             \lesssim & C(\rho_0(x_0))R^{-\frac{3}{2}}\Bigl(\int_0^{\min\bigl\{\frac{T}{\rho_0(x_0)},\frac{R^2}{\rho_0(x_0)}\bigr\}}\int_{B\bigl(x,\frac{R}{\sqrt{\rho_0(x_0)}}\bigr)}\bigl|e^{s\Delta}u_0(y)\bigr|^2\,dy\,\Bigr)^{\f12}\\\
             \lesssim &\|u_0\|_{BMO_b^{-1}}
        \end{aligned}
    \end{equation*}
    where  in the last step, we used the fact that the ball with radius $R$ can be covered by finite balls with radius $\frac{R}{\sqrt{\rho_0(x_0)}}$, and with a constant multiple to $T,b$ with respect to $\|\rho_0^{-1}\|_{L^\infty}$.
    
 As a consequence, we deduce that,
 \begin{equation}\label{S3eq1}
\begin{aligned}
&\sup_{x_0} \Bigl( \sup_{t \leq T} \left( t^{\frac{1}{2}} \|u_{L,x_0}(t)\|_{L^\infty} + t^{1+\frac{\kappa}{2}} \| (|D|^{1+\kappa} u_{L,x_0})(t) \|_{L^\infty} \right) + \|u_{L,x_0}\|_{V_T^b} \Bigr) \\
&\quad + \sup_{x_0} \sup_{s < t \leq T} s^{\frac{1}{2} + \varkappa_1} \frac{\|u_{L,x_0}(t) - u_{L,x_0}(s)\|_{L^\infty}}{(t-s)^{\varkappa_1}}  \lesssim \|u_0\|_{\mathrm{BMO}_b^{-1}}.
\end{aligned}
\end{equation}

On the other hand, we observe that for any function $K(t,s,x,z)\geq 0$, 
\begin{align*}
    &\left(\int_{\mathbb{R}^3}\Bigl(\int_{\mathbb{R}^3}K(t,s,x-y,x)u_0(y)\,dy\Bigr)^2\,dx\right)^{\frac{1}{2}}\\
            &\lesssim \left(\int_{\mathbb{R}^3}\Bigl(\Bigl(\int_{\mathbb{R}^3} K(t,s,x-y,x)|u_0(y)|^2\,dy \Bigr)^{\frac{1}{2}}\Bigl(\int_{\mathbb{R}^3}K(t,s,x-y,x)\,dy \Bigr)^{\frac{1}{2}}\Bigr)^2dx\right)^{\frac{1}{2}}\\
            &\lesssim \Bigl(\sup_{z\in\mathbb{R}^3}\|K(t,s,\cdot,z)\|_{L^1}\Bigr)^{\frac{1}{2}}\left(\int_{\mathbb{R}^3}\int_{\mathbb{R}^3} K(t,s,x-y,x)|u_0(y)|^2\,dy dx\right)^{\frac{1}{2}}\\
            &\lesssim \|\sup_{z\in\mathbb{R}^3}K(t,s,\cdot,z)\|_{L^1}\|u_0\|_{L^2}.
    \end{align*}
We now take $K(t,s,x,z)=K(t,x,z)=\frac{\rho_0^{\frac{3}{2}}(z)}{(4\pi t)^{\frac{3}{2}}}\exp(-\frac{\rho_0(z)|x|^2}{4t}),$ then  it follows from \eqref{bdrho0} that 
\begin{equation*}
  \sup_{z\in\R^3}K(t,s,x,z)\leq \frac{C_2^{\frac{3}{2}}}{(4\pi t)^{\frac{3}{2}}}\exp\Bigl(-\frac{C_1|x|^2}{4t}\Bigr) \andf  \|\sup_{z\in\mathbb{R}^3}K(t,\cdot,z)\|_{L^1}\lesssim 1,\quad \mbox{for any}\ t\geq 0.
\end{equation*}
Hence we obtain
\begin{equation*}\label{S3eq9}
    \sup_{t\leq T}\|\tilde{u}_L(t)\|_{L^2}\lesssim \|u_0\|_{L^2}
\end{equation*}
from which and  \eqref{S3eq1}, we infer
\beq \label{S3eq3}
\begin{split}
\sup_{t \leq T} &t^{\frac{5}{12}} \left\| (K_{x_0} \ast u_0)|_{x_0=x}(t,x) \right\|_{L_x^{12}}\\
\leq &\Bigl(\sup_{x_0\in\R^3}\sup_{t \leq T} t^{\f12} \left\| (K_{x_0} \ast u_0)(t,x) \right\|_{L_x^{\infty}}\Bigr)^{\f56}\Bigl(\sup_{x_0\in\R^3}\sup_{t \leq T} \| \wt{u}_L(t,x)\|_{L_x^{2}}\Bigr)^{\f16}\\
\lesssim & \|u_0\|_{\mathrm{BMO}_b^{-1}}^{\frac{5}{6}} \|u_0\|_{L^2}^{\frac{1}{6}}.\end{split} \eeq

For the term involving time difference of $\wt{u}_L,$  we take 
\begin{equation*}
    \begin{aligned}
      K(t,s,x,z)&=\frac{(\rho_0(z))^{\frac{3}{2}}}{(4\pi t)^{-\frac{3}{2}}}\exp\Bigl(-\frac{\rho_0(z)|x|^2}{4t}\Bigr)-\frac{(\rho_0(z))^{\frac{3}{2}}}{(4\pi s)^{-\frac{3}{2}}}\exp\Bigl(-\frac{\rho_0(z)|x|^2}{4s}\Bigr)  \\
      &=\int_s^t\partial_\tau \Bigl(\frac{(\rho_0(z))^{\frac{3}{2}}}{(4\pi \tau)^{-\frac{3}{2}}}\exp\Bigl(-\frac{\rho_0(z)|x|^2}{4\tau}\Bigr)\Bigr)\,d\tau,
    \end{aligned}
\end{equation*}
from which and \eqref{bdrho0}, we infer
\begin{equation*}
    \bigl\|\sup_{z\in\mathbb{R}^3}K(t,s,\cdot,z)\bigr\|_{L^1}\lesssim \min\bigl\{1,s^{-1}(t-s)\bigr\},\quad \mbox{for any}\ t\geq s\geq 0,
\end{equation*}
so that we obtain
\begin{align*}
&\|\tilde{u}_L(t)-\tilde{ u}_L(s)\|_{L^2}\lesssim \min\bigl\{1,s^{-1}(t-s)\bigr\}\|u_0\|_{L^2},\\
&\|\tilde{u}_L(t)-\tilde{ u}_L(s)\|_{L^\infty}\min\bigl\{1,s^{-1}(t-s)\bigr\}\|u_0\|_{BMO_b^{-1}}.
\end{align*}
As a consequence, we deduce that
\beq \label{S3eq12}
\begin{split}
   & \sup_{s<t<T}s^{\f5{12}+\varkappa_2}\frac{\|\tilde{u}_L(t)-\tilde{ u}_L(s)\|_{L^{12}}}{(t-s)^{\varkappa_2}}\\
    &\lesssim \sup_{s<t<T}\Bigl(\f{s^{\varkappa_2}}{(t-s)^{\varkappa_2}} \bigl(s^{\f12}\|\tilde{u}_L(t)-\tilde{ u}_L(s)\|_{L^\infty}\bigr)^{\f56}\|\tilde{u}_L(t)-\tilde{ u}_L(s)\|_{L^2}^{\f16} \Bigr)\\
    &\lesssim  \f{s^{\varkappa_2}}{(t-s)^{\varkappa_2}}  {\min\{1,s^{-1}(t-s)\}}\|u_0\|_{BMO_b^{-1}}^{\f56}\|u_0\|_{L^2}^{\f16}\\
&\lesssim \|u_0\|_{BMO_b^{-1}}^{\f56}\|u_0\|_{L^2}^{\f16}, \end{split}\eeq

By summarizing the estimates \eqref{S3eq1}, \eqref{S3eq3} and \eqref{S3eq12}, we complete the proof of  \eqref{uL}.
\end{proof}

For $u_{N,x_0}$, we have the following estimates:

\begin{prop}\label{uN}
Let $T \sim b^2 < 1$. Let $(\rho, u, P)$ be a smooth enough solution of \eqref{eqinns} on $[0, T]$ and $u_{N,x_0}$ be defined in \eqref{decop}.  We assume that $\|u\|_{X_T} < \varepsilon_0$ for some sufficiently small $\varepsilon_0 > 0.$ Then we have the following estimates:
\begin{equation}\label{S3eq16}
\begin{aligned}
&\sup_{x_0} \Bigl( \sup_{t \leq T} t^{\frac{1}{2}} \|u_{N,x_0}(t)\|_{L^\infty} + \|u_{N,x_0}\|_{V_T^b} + \sup_{s < t < T} s^{\frac{1}{2} + \varkappa_1} \frac{\|u_{N,x_0}(t) - u_{N,x_0}(s)\|_{L^\infty}}{(t-s)^{\varkappa_1}} \Bigr) \\
&\quad + \sup_{t \leq T} t^{1 + \frac{\kappa}{2}} \| (\nabla |D|^\kappa u_{N,x_0})|_{x_0=x}(t) \|_{L^\infty}
+ \sup_{t \leq T} t^{\frac{5}{12}} \|\tilde{u}_N(t)\|_{L^{12}} \\
&\quad + \sup_{s < t < T} s^{\frac{5}{12} + \varkappa_2} \frac{\|\tilde{u}_N(t) - \tilde{u}_N(s)\|_{L^{12}}}{(t-s)^{\varkappa_2}} \lesssim \|\rho_0\|_{C^1} \|u\|_{X_T}^2.
\end{aligned}
\end{equation}
\end{prop}   

\begin{proof}
Define the norms
\begin{equation*}
\begin{aligned}
&\|f\|_{Z_T^{1,b}} := \sup_{t \leq T} t^{\frac{1}{2}} \|f(t)\|_{L^\infty} + \sup_{x_0 \in \mathbb{R}^3,\, R \leq b} R^{-\frac{3}{2}} \Bigl( \int_0^{\min\{T, R^2\}} \int_{B_{x_0}(R)} |f(\tau,y)|^2 \, dy \, d\tau \Bigr)^{\frac{1}{2}},\\[6pt]
&\|f\|_{Z_T^{2,b}} := \sup_{t \leq T} t \|f(t)\|_{L^\infty} + \sup_{x_0 \in \mathbb{R}^3,\, R \leq b} R^{-3} \int_0^{\min\{T, R^2\}} \int_{B_{x_0}(R)} |f(\tau,y)| \, dy \, d\tau.
\end{aligned}
\end{equation*}

 Thanks to the approach of Koch and Tataru, we have the following lemma (see also \cite[Lemma 5.36]{BCD}). 

\begin{lem}\label{lemtta}
Let 
\[
B(f,g):= \int_0^t K_{x_0}(t-\tau) \ast \mathbb{P} \nabla \cdot (f \otimes g)(\tau) \, d\tau.
\]
Then, for some $T \sim b^2$, we have
\[
\|B(f,g)\|_{Z_T^{1,b}} \leq C \|f \otimes g\|_{Z_T^{2,b}} \leq C \|f\|_{Z_T^{1,b}} \|g\|_{Z_T^{1,b}},
\]
where the implicit constant is independent of $x_0$.
\end{lem}

The proof of this lemma is postponed to the appendix. 

It follows from Lemma \ref{lemtta} that
 \begin{equation}\label{S3eq17}
\sup_{x_0\in\R^3} \sup_{t \leq T} t^{\frac{1}{2}} \|u_{N,x_0}(t)\|_{L^\infty} + \sup_{x_0\in\R^3} \|u_{N,x_0}\|_{V_T^b} \lesssim \|\rho_0\|_{L^\infty} \|u\|_{X_T}^2.
\end{equation}

Let us turn to the estimate  of the remaining norms of $u_{N,x_0}$ in the space $X_T,$ which we will handle term by term below.\smallskip

\noindent\textbf{(i): The  $C^{1+\kappa}$ norm of $u_{N,x_0}$}. \smallskip

We split the time interval into two parts as
\begin{align*}
|D|^\kappa \nabla u_{N,x_0}(t,x) =& \left( \int_0^{t/2} + \int_{t/2}^t \right) \int |D|^\kappa \mathbb{P} \nabla^2 K_{x_0}(t-\tau, x-y) (\rho u \otimes u)(\tau, y) \, dy \, d\tau\\
 =:&u_{N1,x_0} + u_{N2,x_0}.
\end{align*}
Note that by Proposition \ref{regdiff}, if $\|u\|_{X_T} < \varepsilon_0$ for some sufficiently small $\varepsilon_0 > 0$, then
\begin{equation*}\label{rhoreg}
\|\rho\|_{C^\gamma} \lesssim C(\gamma, \varepsilon_0) \|\rho_0\|_{C^1}, \quad \forall \gamma < 1.
\end{equation*}

For $u_{N1}$, we use a method similar to that in the Appendix \ref{sub7.2}.  In order to do so, for all $(t,x) \in [0,+\infty) \times \mathbb{R}^3,$ we define
\begin{equation}\label{defnewker}
\Gamma_R^1(t,x) := \mathbf{1}_{|x| > R} \frac{1}{|x|^{3+\alpha}}, \qquad 
\Gamma_R^2(t,x) := \mathbf{1}_{|x| < R} \frac{1}{(\sqrt{t} + |x|)^{3+\alpha}},
\end{equation}
for some $\alpha>0$. It is easy to observe that  \[
||D|^\kappa \mathbb{P} \nabla^2 K_{x_0}(t, x)| \lesssim \Gamma_R^1(t,x) + \Gamma_R^2(t,x), \] with $\alpha = 2 + \kappa$. Then
\begin{align*}
|u_{N1,x_0}| \lesssim& \int_0^{t/2} \int_{\R^3} \Gamma_R^1(x-y) |\rho u \otimes u|(\tau, y) \, dy \, d\tau\\ &+ \int_0^{t/2} \int_{\R^3} \Gamma_R^2(t-\tau, x-y) |\rho u \otimes u|(\tau, y) \, dy \, d\tau := u_{N1,x_0}^1 + u_{N1,x_0}^2.
\end{align*}
Taking $R = \sqrt{t}$ in \eqref{5.37h} of Lemma \ref{lem5.37} gives
\begin{equation*}
\|u_{N1,x_0}^1(t)\|_{L^\infty} \lesssim t^{-1 - \frac{\kappa}{2}} \|\rho u \otimes u\|_{Z_T^{2,b}}.
\end{equation*}
Note that
\begin{align*}
\|u_{N1,x_0}^2(t)\|_{L^\infty} &\lesssim \int_0^{t/2} \int_{|y| \leq \sqrt{t}} \frac{1}{(\sqrt{t})^{5+\kappa}} |\rho u \otimes u|(\tau, y) \, d\tau \\
&\lesssim \frac{t^{\f32}}{t^{\frac{5+\kappa}{2}}} \|\rho u \otimes u\|_{Z_T^{2,b}} \lesssim t^{-1 - \frac{\kappa}{2}} \|\rho u \otimes u\|_{Z_T^{2,b}}.
\end{align*}
 Combining the estimates for $u_{N1}^1$ and $u_{N1}^2$, and noting that $\|\rho u \otimes u\|_{Z_T^{2,b}} \lesssim \|u\|_{Z_T^{1,b}}^2 \|\rho_0\|_{L^\infty}$, we obtain
\begin{equation}\label{uN1}
\sup_{t \leq T} t^{1 + \frac{\kappa}{2}} \|u_{N1,x_0}(t)\|_{L^\infty} \lesssim \|u\|_{X_T}^2 \|\rho_0\|_{L^\infty}.
\end{equation}

%It follows from \eqref{eq3.1} that
%\begin{equation}\label{eq3.6}\|D^{1-\eta} (\rho u \otimes u)\|_{L^\infty} \lesssim t^{-\frac{3-\eta}{2}} \|u\|_{X_T}^2 \|\rho_0\|_{C^1} (1 + t^{\frac{1-\eta}{2}}).
%\end{equation}
For $u_{N2}$,  we have
\begin{align*}
|u_{N2,x_0}(t,x)| &\lesssim \Bigl| \int_{t/2}^t \int_{\mathbb{R}^3} \mathbb{P} \nabla^2 |D|^\kappa K_{x_0}(t-\tau, x-y) \bigl( (\rho u \otimes u)(\tau, y) - (\rho u \otimes u)(\tau, x) \bigr) \, dy \, d\tau \Bigr| \\
&\lesssim \int_{t/2}^t \int_{\mathbb{R}^3} \left| \mathbb{P} \nabla^2 |D|^\kappa K_{x_0}(t-\tau, x-y) \right| |x-y|^{\kappa'} \|\rho u \otimes u(\tau)\|_{\dot{C}^{\kappa'}} \, dy \, d\tau \\
&\lesssim \int_{t/2}^t (t-\tau)^{-1 - \frac{\kappa}{2} + \frac{\kappa'}{2}} \|\rho u \otimes u(\tau)\|_{\dot{C}^{\kappa'}} \, d\tau.
\end{align*}
for some $\kappa' \in (\kappa, 1)$.
Yet we get, by using interpolation inequality in H\"older space and the definition of $\|\cdot\|_{X_T}$ given by \eqref{defxt} that
\begin{align*}
\|u(t)\|_{\dot{C}^{\kappa'}}\lesssim \|u(t)\|_{L^\infty}^{1-\f{\kappa'}{1+\kappa}}\|u(t)\|_{\dot{C}^{1+\kappa}}^{\f{\kappa'}{1+\kappa}}
\lesssim t^{\f12\left(1+\kappa'\right)}\|u\|_{X_T},
\end{align*}
from which and the law of product in H\"older space (see \cite{BCD}), we infer
\beq \label{S3eq13}
\begin{split}
 \|\rho u \otimes u(\tau)\|_{\dot{C}^{\kappa'}} \lesssim &\|\rho u\|_{L^\infty}\|u\|_{\dot{C}^{\kappa'}}
 +\|\rho u\|_{\dot{C}^{\kappa'}}\|u\|_{L^\infty}\\
 \lesssim &\|\rho_0\|_{L^\infty}\|u\|_{L^\infty}\|u\|_{\dot{C}^{\kappa'}}
+\|\rho\|_{\dot{C}^{\kappa'}}\|u\|_{L^\infty}^2
\lesssim \tau^{-1-\frac{\kappa'}2}\|u\|_{X_T}^2 \|\rho_0\|_{C^1}.
\end{split}\eeq
As a consequence, 
we obtain
\begin{equation}\label{uN2}
|u_{N2,x_0}(t,x)|\\
\lesssim t^{-1 - \frac{\kappa}{2}} \|u\|_{X_T}^2 \|\rho_0\|_{C^1} 
\end{equation}

Thus, by \eqref{uN1} and \eqref{uN2}, we have proved
\begin{equation*}
\sup_{t \leq T} t^{1 + \frac{\kappa}{2}} \|\nabla |D|^\kappa u_{N,x_0}(t)\|_{L^\infty} \lesssim \|u\|_{X_T}^2 \|\rho_0\|_{C^1} .
\end{equation*}
Note that all of the estimates above are uniform with respect to $x_0$. Therefore, for $T \sim b^2$, we have
\begin{equation*}
\sup_{t \leq T} t^{1 + \frac{\kappa}{2}} \| (\nabla |D|^\kappa u_{N,x_0})|_{x_0=x}(t) \|_{L^\infty} \lesssim \|u\|_{X_T}^2 \|\rho_0\|_{C^1}.
\end{equation*}

\noindent \textbf{(ii): The non-critical $L^{12}$ norm}. \smallskip

We first observe from Proposition \ref{S2prop3} that
\begin{align*}
|\na_x K_{x_0}(t,x)|\leq \f{C(C_1,C_2)}{\left(\sqrt{t}+|x|\right)^4} \andf \|\na_x K_{x_0}(t,\cdot)\|_{L^1}\lesssim t^{-\f12},
\end{align*}
from which and 
 \eqref{deftilu},  we get, by direct computation, that
\begin{equation*}
\begin{aligned}
|\tilde{u}_N(t,x)| &\lesssim \int_0^t \int \bigl| \nabla_y K_{x_0}(t-\tau, x, x-y) \bigr|^{\frac{11}{12}} \Bigl( \bigl| \nabla_y K_{x_0}(t-\tau, x, x-y) \bigr|^{\frac{1}{12}} |\rho u \otimes u(\tau, y)| \Bigr) dy \, d\tau \\
&\lesssim \int_0^t (t-\tau)^{-\frac{11}{24}} \Bigl( \int_{\R^3} \bigl| \nabla_y K_{x_0}(t-\tau, x, x-y) \bigr| |\rho u \otimes u(\tau, y)|^{12} \, dy \Bigr)^{\frac{1}{12}} d\tau.
\end{aligned}
\end{equation*}
Taking the $L^{12}$ norm with respect to $x$, we obtain
\begin{equation*}
\begin{aligned}
\|\tilde{u}_N(t)\|_{L^{12}}&\lesssim \int_0^t (t-\tau)^{-\frac{11}{24}} \Bigl( \int_{\R^3}\int_{\R^3} \bigl| \nabla_y K_{x_0}(t-\tau, x, x-y) \bigr| |\rho u \otimes u(\tau, y)|^{12} \, dy\,dx \Bigr)^{\frac{1}{12}} \,d\tau\\
&\lesssim \int_0^t (t-\tau)^{-\frac{1}{2}} \|\rho u \otimes u(\tau) \|_{L^{12}}\, d\tau \\
&\lesssim \int_0^t (t-\tau)^{-\frac{1}{2}} \tau^{-\frac{11}{12}} \, d\tau \, \|u\|_{X_T}^2 \|\rho\|_{L_T^\infty L^\infty} \\
&\lesssim t^{-\frac{5}{12}} \|u\|_{X_T}^2 \|\rho_0\|_{L^\infty},
\end{aligned}
\end{equation*}
where we used the fact that
\begin{align*}
\|u(\tau)\|_{L^{12}}^2\leq \|u(\tau)\|_{L^{\infty}}\|u(\tau)\|_{L^{12}}\lesssim \tau^{-\f{11}{12}} \|u\|_{X_T}^2.
\end{align*}

    \noindent  \textbf{(iii): The time difference for $L^\infty$}. \smallskip
    
For the quantity $\frac{\|u_{N,x_0}(t)-u_{N,x_0}(s)\|_{L^\infty}}{(t-s)^{\varkappa_1}}$, it suffices to consider the case $t-s < \frac{s}{4}$, since otherwise, we have
\begin{align*}
\sup_{s<t\leq T}s^{\f12+\varkappa_1}\frac{\|u_{N,x_0}(t)-u_{N,x_0}(s)\|_{L^\infty}}{(t-s)^{\varkappa_1}}
\lesssim &\sup_{s<t\leq T}s^{\f12}{\|u_{N,x_0}(t)-u_{N,x_0}(s)\|_{L^\infty}}\\
\lesssim &\sup_{t<T} t^{\frac{1}{2}} \|u_{N,x_0}(t)\|_{L^\infty}.
\end{align*}

 Let $a := t-s$ and write
\begin{equation}\label{decopuN}
\begin{aligned}
u_{N,x_0}(t) - u_{N,x_0}(s) &= \int_s^t \mathbb{P} \nabla K_{x_0}(t-\tau) \ast (\rho u \otimes u)(\tau) \, d\tau \\
&\quad + \left( \int_0^{s/2} + \int_{s/2}^s \right) \delta_a^t \mathbb{P} \nabla K_{x_0}(s-\tau) \ast (\rho u \otimes u)(\tau) \, d\tau \\
&=: Du_{N,x_0}^1(t,s) + Du_{N,x_0}^2(t,s) + Du_{N,x_0}^3(t,s).
\end{aligned}
\end{equation}

By virtue of \eqref{S3eq13},  for any $\eta\in (0,1),$ we get, by first taking $\kappa'\in (1-\eta,1) $ and then using the interpolation inequality in Besov spaces, that 
\beq\label{S3eq14}
\begin{split}
\|D^{1-\eta}(\rho u \otimes u)(\tau)\|_{L^\infty} \lesssim &\|(\rho u \otimes u)(\tau)\|_{\dot{B}^{1-\eta}_{\infty,1}} 
\lesssim \|(\rho u \otimes u)(\tau)\|_{L^\infty}^{1-\theta}\|\rho u \otimes u(\tau)\|_{\dot{C}^{\kappa'}}^{\theta}\\
 \lesssim &\bigl(\|\rho_0\|_{L^\infty}\tau\|u(\tau)\|_{L^\infty}^2\bigr)^{1-\theta}
 \tau^{1-\theta}\times \tau^{-\left(1+\f{\kappa'}2\right)\theta}\bigl(\|\rho_0\|_{C^1}\|u\|_{X_T}^2\bigr)^\theta\\ 
\lesssim &\tau^{-\f{3-\eta}2}\|\rho_0\|_{C^1}\|u\|_{X_T}^2,
\end{split}\eeq
where $\theta=\f{1-\eta}{\kappa'}.$

While  for $\Gamma_\eta(t) := \mathbb{P} \nabla |D|^{\eta-1} K_{x_0}(t)$ with some $\eta \ll \kappa$ sufficiently small, it follows from Proposition \ref{S2prop3} that
\begin{align*}
|\Gamma_\eta(t,x)|\leq \f{C(C_1,C_2)}{(\sqrt{t}+|x|)^{3+\eta}} \andf \|\Gamma_\eta(t,\cdot)\|_{L^1}\lesssim t^{-\f\eta2}.
\end{align*}

As a consequence, we deduce that
\begin{equation*}
\|Du_{N,x_0}^1(t,s)\|_{L^\infty} \lesssim \int_s^t \|\Gamma_\eta(t-\tau)\|_{L^1} \|D^{1-\eta}(\rho u \otimes u)(\tau)\|_{L^\infty} \, d\tau \lesssim a^{1-\frac{\eta}{2}} s^{-\frac{3-\eta}{2}} \|u\|_{X_T}^2 \|\rho_0\|_{C^1}.
\end{equation*}

 Similarly, for $Du_{N,x_0}^3(t,s)$, by \eqref{htkint} we have
\begin{equation}\label{diffker}
\|\delta_a^t \Gamma_\eta(t)\|_{L^1} \lesssim \min\left\{ t^{-\frac{\eta}{2}}, a t^{-1-\frac{\eta}{2}} \right\}.
\end{equation}
In particular, for $0<\varkappa_1<1-\f{\eta}2,$ we deduce from \eqref{S3eq14} that
\begin{align*}
\|Du_{N,x_0}^3(t,s)\|_{L^\infty} \lesssim & \int_{s/2}^s \|\delta_a^t \Gamma_\eta(s-\tau)\|_{L^1} \|D^{1-\eta}(\rho u \otimes u)(\tau)\|_{L^\infty} \, d\tau\\
\lesssim&a^{\varkappa_1}\int_{s/2}^s (s-\tau)^{-\f\eta2-\varkappa_1}\tau^{-\f{3-\eta}2}\,d\tau \|u\|_{X_T}^2 \|\rho_0\|_{C^1}\\
 \lesssim &a^{\varkappa_1} s^{-\frac{1}{2}-\varkappa_1} \|u\|_{X_T}^2 \|\rho_0\|_{C^1} .
\end{align*}

For $Du_{N,x_0}^2(t,s)$, we write
\begin{equation*}
Du_{N,x_0}^2(t,s) = \int_0^{s/2} \left( \int_s^t \partial_\gamma \mathbb{P} \nabla K_{x_0}(\gamma-\tau) \, d\gamma \right) \ast (\rho u \otimes u)(\tau) \, d\tau.
\end{equation*}
Denote $H(\tau) := \partial_\gamma \mathbb{P} \nabla K_{x_0}(\tau)$; then $H(\gamma-\tau) \lesssim \frac{1}{(\sqrt{\gamma-\tau} + |y|)^6}$. Using the same decomposition as in Lemma \ref{lem5.37}, we write $H = H^1 + H^2$, with $H^1 \lesssim \mathbf{1}_{|y| \geq R} \frac{1}{|y|^6}$ and $H^2 \lesssim \mathbf{1}_{|y| \leq R} \frac{1}{(\sqrt{\gamma-\tau} + |y|)^6}$. Take $R = \sqrt{s} < b$. By \eqref{5.37h} in Lemma \ref{lem5.37} with $\alpha = 3$, we obtain
\begin{equation*}
\left| \int_0^{s/2} \int_s^t H^1(\gamma-\tau) \ast (\rho u \otimes u)(\tau) \, d\gamma \, d\tau \right| \lesssim \frac{t-s}{R^3} \|u\|_{X_T}^2 \|\rho_0\|_{C^1}.
\end{equation*}
Moreover, for the part concerning $H^2,$ we have
\begin{equation*}
\begin{aligned}
&\left| \int_0^{s/2} \int_s^t H^2(\gamma-\tau) \ast (\rho u \otimes u)(\tau) \, d\gamma \, d\tau \right| \\
&\quad \lesssim \int_0^{s/2} \int_s^t \int_{|y| \leq R} \frac{1}{(\gamma-\tau)^3} |\rho u \otimes u|(\tau, x-y) \, dy \, d\gamma \, d\tau \\
&\quad \lesssim \int_s^t \frac{R^3}{(\gamma - \frac{s}{2})^3} \, d\gamma \, \|u\|_{X_T}^2 \|\rho_0\|_{C^1} \lesssim \frac{(t-s) R^3}{s^3} \|u\|_{X_T}^2 \|\rho_0\|_{C^1}.
\end{aligned}
\end{equation*}
As a result, since $R=\sqrt{s}$, it comes out 
\[\|Du_{N,x_0}^3(t,s)\|_{L^\infty} \lesssim \frac{t-s}{s^{3/2}} \|u\|_{X_T}^2 \|\rho_0\|_{C^1}.
% \lesssim \left( \frac{t-s}{s} \right)^{1-\frac{\eta}{2}} s^{-1/2} \|u\|_{X_T}^2 \|\rho_0\|_{C^1}.
\]
Since all the estimates above are uniform with respect to $x_0$ and $t-s < \frac{s}{4}$, we conclude
\begin{equation*}
\sup_{x_0} \sup_{s < t < T} s^{\frac{1}{2} + \varkappa_1} \frac{\|u_{N,x_0}(t) - u_{N,x_0}(s)\|_{L^\infty}}{(t-s)^{\varkappa_1}} \lesssim \|u\|_{X_T}^2 \|\rho_0\|_{C^1}.
\end{equation*}

\noindent\textbf{(iv): The time difference for $L^{12}$}. \smallskip

For any $\eta\in (0,1),$ we get, by using \eqref{eq3.1}, that
\begin{align*}
    \left\||D|^{1-\eta}(\rho u\otimes u)\right\|_{L^{12}}&\lesssim\int_{\mathbb{R}^3}\frac{\left\|\delta_a(\rho u\otimes u)\right\|_{L^{12}}}{|\alpha|^{4-\eta}}d\alpha\\
    &\lesssim \int_{\R^3}\frac{\|\delta_\alpha \rho\|_{L^\infty}\|u\|_{L^\infty}\|u\|_{L^{12}}}{|\alpha|^{4-\eta}}d\alpha+\int_{\R^3}\frac{\|\rho\|_{L^\infty}\|\delta_\alpha u\|_{L^\infty}\|u\|_{L^{12}}}{|\alpha|^{4-\eta}}d\alpha\\
    &\qquad+\int_{\R^3}\frac{\|\rho\|_{L^\infty}\|\delta_\alpha u\|_{L^\infty}\sup_{\alpha}\|u(\cdot-\alpha)\|_{L^{12}}}{|\alpha|^{4-\eta}}d\alpha
\end{align*}
As a result, it comes out
 \beq\label{S3eq18}
 \begin{split}
\| |D|^{1-\eta}&(\rho u \otimes u)(t) \|_{L^{12}} \lesssim \|u(t)\|_{L^\infty} \|u(t)\|_{L^{12}} \|\rho(t)\|_{C^{1-\eta}} \\
&+ \|u(t)\|_{L^\infty}^{\eta}\|u(t)\|_{\dot C^1}^{1-\eta} \|u(t)\|_{L^{12}} \|\rho\|_{L^\infty}
\lesssim \|\rho_0\|_{C^1}t^{-\left(\f{17}{12}-\f\eta2\right)}\|u\|_{X_T}^2.
\end{split}
\eeq

Let us now turn to the estimate of  $\sup_{s < t < T} s^{\frac{5}{12} + \varkappa_2} \frac{\|\tilde{u}_N(t) - \tilde{u}_N(s)\|_{L^{12}}}{(t-s)^{\varkappa_2}}.$ We again use the corresponding decomposition in \eqref{decopuN}. Let $a = t-s < \frac{s}{4}$. Then we deduce from \eqref{esthtk} and \eqref{S3eq18} that
\begin{align*}
\|D\tilde{u}_N^1(t,s)\|_{L^{12}} &\lesssim \int_s^t \| \tilde{\Gamma}_\eta(t-\tau, z, y)\|_{L^1_y(L^\infty_z)} \| |D|^{1-\eta}(\rho u \otimes u)(\tau) \|_{L^{12}} \, d\tau \\
&\lesssim a^{1-\frac{\eta}{2}} s^{-\frac{2-\eta}{2} - \frac{5}{12}} \|\rho_0\|_{C^1} \|u\|_{X_T}^2.
\end{align*}
Here, we denote $\tilde{\Gamma}_\eta(t, z, y) := (\nabla |D|^{\eta-1} K_{x_0}(t, y))|_{x_0=z}$ for $0 < \eta \ll \kappa$ and $T \leq 1$. 

Similarly, we deduce from \eqref{htkint} and \eqref{S3eq18} that
\begin{align*}
\|D\tilde{u}_N^3\|_{L^{12}} &\lesssim \int_{s/2}^s \| \delta_a^t \tilde{\Gamma}_\eta(s-\tau, z, y) \|_{L^1_y L^\infty_z} \| D^{1-\eta}(\rho u \otimes u) \|_{L^{12}} \, d\tau \\
&\lesssim a^{1-\theta} \int_{s/2}^s (s-\tau)^{-\theta-\f\eta2}\tau^{-\left(\f{17}{12}-\f\eta2\right)}\,d\tau\|\rho_0\|_{C^1} \|u\|_{X_T}^2\\
&\lesssim a^{\theta} s^{-\theta- \frac{5}{12}} \|\rho_0\|_{C^1} \|u\|_{X_T}^2,
\end{align*}
for any $\theta<1-\f\eta2.$

For $u_{N,x_0}^2$, using the fact that $\|\rho u \otimes u(t)\|_{L^{12}} \lesssim t^{-\frac{11}{12}} \|u\|_{X_T}^2$ and \eqref{diffker}, we infer
\begin{align*}
\|D\tilde{u}_N^2(t,s)\|_{L^{12}} \lesssim& \int_0^{s/2} \| \delta_a^t \tilde{\Gamma}_1(s-\tau, z, y) \|_{L^1_y (L^\infty_z)} \|\rho u \otimes u(\tau)\|_{L^{12}} \, d\tau\\
\lesssim&a^{\varkappa_2}\int_0^{s/2}(s-\tau)^{-\f12-\varkappa_2}\tau^{-\f{11}{12}}\,d\tau\|u\|_{X_T}^2 \|\rho_0\|_{C^1}\\
 \lesssim & a^{\varkappa_2} s^{-\varkappa_2 - \frac{5}{12}} \|u\|_{X_T}^2 \|\rho_0\|_{C^1}.
\end{align*}
Thus we conclude that
\begin{equation*}
\sup_{s < t < T} s^{\frac{5}{12} + \varkappa_2} \frac{\|\tilde{u}_N(t) - \tilde{u}_N(s)\|_{L^{12}}}{(t-s)^{\varkappa_2}} \lesssim \|u\|_{X_T}^2 \|\rho_0\|_{C^1}.
\end{equation*}

Combining the results (i)–(iv) above with  \eqref{S3eq17}, we  complete the proof of Proposition \ref{uN}.
\end{proof}

 \section{Estimates for remainder terms}\label{sec5}
 
	In this section we derive estimates for the remainder terms $u_{R_i,x_0}$ $(i=1,2,3)$ that appear in the decomposition \eqref{decop}. Recall the definitions
\begin{equation}\label{S4eq0}
\begin{aligned}
&u_{R1,x_0}(t,x) = \int_0^t \int_{\R^3} K_{x_0}(t-\tau, x-y) \nabla \Delta^{-1} (\nabla \rho_0 \cdot \partial_\tau u)(\tau, y) \, dy \, d\tau,\\
&u_{R2,x_0}(t,x) = -\int_0^t \int_{\R^3} K_{x_0}(t-\tau, x-y) \mathbb{P} \partial_\tau (\varrho u)(\tau, y) \, dy \, d\tau,\\
&u_{R3,x_0}(t,x) = -\int_0^t \int_{\R^3} K_{x_0}(t-\tau, x-y) \partial_\tau \bigl( (\rho_0(y) - \rho_0(x_0)) u(\tau) \bigr) \, dy \, d\tau.
\end{aligned}
\end{equation}
We denote $\tilde{u}_{Ri}(t,x) := (u_{Ri,x_0})|_{x_0=x}$ and $K(t,x,y) := (K_{x_0}(t,y))|_{x_0=x}$. These terms arise from the spatial inhomogeneity of the density $\rho$ and the localization in $x_0$. We shall prove that each $u_{R_i,x_0}$ satisfies an estimate of the same type as \eqref{estuLx} and \eqref{S3eq16},
  with additional small factors depending on the regularity of $\rho$. All constants appearing in this section are independent of $x_0$ and $b$.

To prove the estimates of time differences, we introduce the following technical lemma. The proof is direct, and a similar version can be found in Lemma 2.3 of \cite{KHN} and Lemma 2.2 of \cite{KHN22}.

\begin{lem}\label{CVPDE}
Let $0 \leq \eta, \sigma < 1$ and $1 \leq p \leq \infty$. Define
\begin{equation*}
g_{x_0}(t,x) := \int_0^t \int \partial_\tau K_{x_0}(t-\tau, x-y) f(\tau, y) \, dy \, d\tau,
\end{equation*}
where $K_{x_0}$ is defined in \eqref{defKx0}, and define
\begin{equation*}
\tilde{g}(t,x) := g_{x_0}(t,x)|_{x_0=x}.
\end{equation*}
Then we have
\begin{equation}\label{lem51inf}
\begin{aligned}
&\sup_{x_0} \sup_{t \leq T} t^{\eta} \|g_{x_0}(t)\|_{L^\infty}
+ \sup_{x_0} \sup_{s < t \leq T} s^{\eta + \sigma} \frac{\|g_{x_0}(t) - g_{x_0}(s)\|_{L^\infty}}{(t-s)^{\sigma}} \\
&\qquad \lesssim \sup_{t \leq T} t^{\eta} \|f(t)\|_{L^\infty}
+ \sup_{s < t \leq T} s^{\eta + \sigma} \frac{\|f(t) - f(s)\|_{L^\infty}}{(t-s)^{\sigma}},
\end{aligned}
\end{equation}
and for any $p \in [1, \infty)$,
\begin{equation}\label{lem51p}
\begin{aligned}
&\sup_{t \leq T} t^{\eta} \|\tilde{g}(t)\|_{L^p}
+ \sup_{s < t \leq T} s^{\eta + \sigma} \frac{\|\tilde{g}(t) - \tilde{g}(s)\|_{L^p}}{(t-s)^{\sigma}} \\
&\qquad \lesssim \sup_{t \leq T} t^{\eta} \|f(t)\|_{L^p}
+ \sup_{s < t \leq T} s^{\eta + \sigma} \frac{\|f(t) - f(s)\|_{L^p}}{(t-s)^{\sigma}}.
\end{aligned}
\end{equation}
\end{lem}
\begin{proof}
		For simplicity, we only prove \eqref{lem51p}, since \eqref{lem51inf} can be proved along the same line. Denote that $\partial_t K(t,x,y) = (\partial_t K_{x_0}(t,y))|_{x_0=x}$. First, for the $L^p$ norm, we write
\begin{equation*}
\begin{aligned}
\tilde{g}(t,x) =& \int_0^{t/2} \int_{\R^3} \partial_\tau K(t-\tau, x, x-y) f(\tau, y) \, dy \, d\tau \\
&+ \int_{t/2}^t \int_{\R^3} \partial_\tau K(t-\tau, x, x-y) \bigl( f(\tau, y) - f(t, y) \bigr) \, dy \, d\tau \\
&+ f(t,x) - \int_{\R^3} K(t/2, x, x-y) f(t, y) \, dy.
\end{aligned}
\end{equation*}
Note that for any $h(x)$ and $m,n\geq 0$, if we denote $\mathbf{h}(t,x) := \int_{\R^3} \partial_t^{m_1} \nabla_y^{m_2} K(t, x, x-y) h(y) \, dy$, then by Hölder's inequality,
\begin{equation*}
\begin{aligned}
|\mathbf{h}(t,x)| &\lesssim \int_{\R^3} |\partial_t^{m_1} \nabla_y^{m_2} K(t, x, x-y)|^{\frac{1}{p'}} |\partial_t^{m_1} \nabla_y^{m_2} K(t, x, x-y)|^{\frac{1}{p}} |h(y)| \, dy \\
&\lesssim \bigl\| \sup_z \partial_t^{m_1} \nabla_x^{m_2} K(t, z, x) \bigr\|_{L_x^1}^{\frac{1}{p'}} \left( \int_{\R^3} |\partial_t^{m_1} \nabla_y^{m_2} K(t, x, x-y)| |h(y)|^p \, dy \right)^{\frac{1}{p}},
\end{aligned}
\end{equation*}
which leads to
\begin{equation}\label{refY}
\begin{aligned}
\|\mathbf{h}(t)\|_{L^p} &\lesssim \bigl\| \sup_z \partial_t^{m_1} \nabla_x^{m_2} K(t, z, x) \bigr\|_{L_x^1}^{\frac{1}{p'}} \left( \iint_{\R^3\times\R^3} |\partial_t^{m_1} \nabla_y^{m_2} K(t, x, x-y)| |h(y)|^p \, dy \, dx \right)^{\frac{1}{p}} \\
&\lesssim \bigl\| \sup_z \partial_t^{m_1} \nabla_x^{m_2} K(t, z, x) \bigr\|_{L_x^1} \|h\|_{L^p}.
\end{aligned}
\end{equation}
Note that, for any $x_0$ and $t$, one has
$
    \left\|K(t,x_0,x)\right\|_{L_x^1}=1,$
so that it follows from \eqref{refY} that
\begin{equation*}
    \left\|\int_{\mathbb{R}^d}K({t}/{2},x,x-y)f(t,y)\,dy\right\|_{L_x^p}\lesssim \|f(t,\cdot)\|_{L^p}\lesssim t^{-\eta}\sup_{t\leq T}t^\eta\|f(t)\|_{L^p}.
\end{equation*}
By \eqref{refY} and the pointwise estimate \eqref{esthtk}, we find
\begin{equation}\label{S4eq1}
\begin{aligned}
\|\tilde{g}(t)\|_{L^p} &\lesssim \Bigl( \int_0^{t/2} (t-\tau)^{-1} \tau^{-\eta} \, d\tau + t^{-\eta} \Bigr) \sup_{t \leq T} t^{\eta} \|f(t)\|_{L^p} \\
&\quad + \int_{t/2}^t (t-\tau)^{-1+\sigma} \tau^{-\sigma-\eta} \, d\tau \sup_{s < t \leq T} s^{\eta+\sigma} \frac{\|f(t) - f(s)\|_{L^p}}{(t-s)^{\sigma}} \\
&\lesssim t^{-\eta} \Bigl( \sup_{t \leq T} t^{\eta} \|f(t)\|_{L^p} + \sup_{s < t \leq T} s^{\eta+\sigma} \frac{\|f(t) - f(s)\|_{L^p}}{(t-s)^{\sigma}} \Bigr).
\end{aligned}
\end{equation}
For the time differences, let $a:= t-s$. For the same reason as in the previous section, we only need to consider the case $a < \frac{s}{4}$. We write
\begin{equation*}
\begin{aligned}
\tilde{g}(t,x) &= \int_0^{t/2} \int_{\R^3} \partial_\tau K(t-\tau, x, x-y) f(\tau, y) \, dy \, d\tau \\
&\quad + \int_{t/2}^t \int_{\R^3} \partial_\tau K(t-\tau, x, x-y) f(\tau, y) \, dy \, d\tau:= I(t,x) + II(t,x).
\end{aligned}
\end{equation*}
For $I(t)$, since
\begin{align*}
I(t) - I(s) = &\int_0^{s/2} \int_{\R^3} \delta_a^t \partial_\tau K(s-\tau, x, x-y) f(\tau, y) \, dy \, d\tau\\
& + \int_{s/2}^{t/2} \int_{\R^3} \partial_\tau K(t-\tau, x, x-y) f(\tau, y) \, dy \, d\tau,
\end{align*}
by \eqref{esthtk} and \eqref{refY},  one obtains
\begin{align*}
 {\|I(t) - I(s)\|_{L^p}} \lesssim& \Bigl(\int_0^{s/2}\min\bigl\{(s-\tau)^{-1}, a(s-\tau)^{-2}\big\}\tau^{-\eta}\,d\tau\\
 &+\int_{s/2}^{t/2}(t-\tau)^{-1}\tau^{-\eta}\,d\tau\Bigr)
 \sup_{t \leq T} t^{\eta} \|f(t)\|_{L^p}\\
 \lesssim & a^{\sigma}s^{-\sigma-\eta}\sup_{t \leq T} t^{\eta} \|f(t)\|_{L^p},
\end{align*}
where we used the fact that 
\begin{align*}
\int_{s/2}^{t/2}(t-\tau)^{-1}\tau^{-\eta}\,d\tau\lesssim a t^{-1}s^{-\eta}
\lesssim a^{\sigma}s^{-\sigma-\eta} \Bigl(\f{a}t\Bigr)^{1-\sigma}\lesssim a^\sigma s^{-\sigma-\eta}.
\end{align*}
Furthermore, notice that
\begin{align*}
II(t) - II(s) &= -\int_{s/2}^{t/2} \int_{\R^3} \partial_\tau K(s-\tau, x, x-y) f(\tau, y) \, dy \, d\tau \\
&\quad + \int_{t/2}^{s} \int_{\R^3} \delta_a^t \partial_\tau K(s-\tau, x, x-y) f(\tau, y) \, dy \, d\tau \\
&\quad + \int_{s}^{t} \int_{\R^3} \partial_\tau K(t-\tau, x, x-y) f(\tau, y) \, dy \, d\tau \\
&= -\int_{s/2}^{t/2} \int_{\R^3} \partial_\tau K(s-\tau, x, x-y) f(\tau, y) \, dy \, d\tau \\
&\quad + \int_{t/2}^{s} \int_{\R^3} \delta_a^t \partial_\tau K(s-\tau, x, x-y) \left(f(\tau, y)-f(s,y)\right) \, dy \, d\tau \\
&\quad + \int_{s}^{t} \int_{\R^3} \partial_\tau K(t-\tau, x, x-y) \left(f(\tau, y)-f(t,y)\right) \, dy \, d\tau \\
&\quad +\left(\int_{t/2}^{s} \int_{\R^3} \delta_a^t \partial_\tau K(s-\tau, x, x-y)f(s,y) \, dy \, d\tau\right.\\
&\left.\qquad+\int_{s}^{t} \int_{\R^3} \partial_\tau K(t-\tau, x, x-y)f(t,y) \, dy \, d\tau\right)\\
&:= DII_1(t,s) + DII_2(t,s) + DII_3(t,s)+DII_4(t,s).
\end{align*}
Observing that
\begin{equation*}
    \sup_x\|\partial_\tau K(s-\tau,x,y)\|_{L_y^1}\lesssim (s-\tau)^{-1},
\end{equation*}
and on the time interval $(s/2, t/2)$, there holds $(s-\tau)^{-1}\lesssim s^{-1}$,  we deduce from \eqref{refY} and  $a<\frac{s}{4}$,  that
    \begin{align*}
        \|DII_1(t,s)\|_{L^p}&\lesssim s^{-1}\int_{\frac{s}{2}}^{\frac{t}{2}}\|f(\tau)\|_{L^p}\,d\tau\lesssim s^{-1}\int_{\frac{s}{2}}^{\frac{t}{2}}\tau^{-\eta}\,d\tau\sup_{\tau\leq T}\tau^{\eta}\|f(\tau)\|_{L^p}\\
        &\lesssim (t-s)s^{-1-\eta}\sup_{\tau\leq T}\tau^{\eta}\|f(\tau)\|_{L^p}\lesssim (t-s)^{\sigma}s^{-\sigma-\eta}\sup_{\tau\leq T}\tau^{\eta}\|f(\tau)\|_{L^p}.
    \end{align*}

By \eqref{esthtk} and \eqref{refY} (pointwisely for each $\tau$) we have
    \begin{align*}
       &\bigl\| \int_{t/2}^{s} \int_{\R^3} \delta_a^t \partial_\tau K(s-\tau, x, x-y) \bigl( f(\tau, y) - f(s, y) \bigr) \, dy \, d\tau\bigr\|_{L^p}\\
        &\lesssim \int_{\frac{t}{2}}^s(s-\tau)^{-1+\sigma} \min\Bigl\{1,\frac{a}{s-\tau}\Bigr\}\,\tau^{-\sigma-\eta}\,d\tau \sup_{s < t \leq T} s^{\eta+\sigma} \frac{\|f(t) - f(s)\|_{L^p}}{(t-s)^{\sigma}}\\
        &\lesssim a^{\sigma}s^{-\sigma-\eta}\sup_{s < t \leq T} s^{\eta+\sigma} \frac{\|f(t) - f(s)\|_{L^p}}{(t-s)^{\sigma}},
    \end{align*}
    where we used the fact that
    \begin{align*}
    \int_{\frac{t}{2}}^s(s-\tau)^{-1+\sigma} \min\Bigl\{1,\frac{a}{s-\tau}\Bigr\}\,\tau^{-\sigma-\eta}\,d\tau 
    = &a\int_{\frac{t}{2}}^{2s-t}(s-\tau)^{-2+\sigma} \tau^{-\sigma-\eta}\,d\tau \\
&+\int_{2s-t}^s(s-\tau)^{-1+\sigma} \tau^{-\sigma-\eta}\,d\tau 
\lesssim a^{\sigma}s^{-\sigma-\eta}.
\end{align*}
Then along the same line to the proof of \eqref{S4eq1}, we find
\begin{align*}
&\|DII_2(t,s)\|_{L^p} + \|DII_3(t,s)\|_{L^p}\\
&\quad \lesssim (t-s)^{\sigma} s^{-\eta-\sigma} \Bigl( \sup_{t \leq T} t^{\eta} \|f(t)\|_{L^p} + \sup_{s < t \leq T} s^{\eta+\sigma} \frac{\|f(t) - f(s)\|_{L^p}}{(t-s)^{\sigma}} \Bigr).
\end{align*}
For the term $DII_4(t,s)$, we write
    \begin{align*}
        DII_4(t,s)=&\Bigl(\int_{\mathbb{R}^3}K(t-s,x,x-y)\bigl(f(s,y)dy-f(t,y)\bigr)\,dy\Bigr)\\
        &+\Bigl(\int_{\mathbb{R}^3}K(s-\frac{t}{2},x,x-y)f(s,y)dy-\int_{\mathbb{R}^3}K(\frac{t}{2},x,x-y)f(s,y)dy\Bigr)\\
        &+\left(f(t,x)-f(s,x)\right).
    \end{align*}
Then by \eqref{diffker} and \eqref{refY}, we deduce that
\begin{equation*}
    \begin{aligned}
        \|DII_4(t,s)\|_{L^p}&\lesssim \|f(t)-f(s)\|_{L^p}+\|f(s)\|_{L^p}\sup_{x}\left\|K(s-{t}/{2},x,y)-K({t}/{2},x,y)\right\|_{L_y^1}\\
        &\lesssim a^\sigma s^{-\sigma-\eta}\Bigl(\sup_{t \leq T} t^{\eta} \|f(t)\|_{L^p} + \sup_{s < t \leq T} s^{\eta+\sigma} \frac{\|f(t) - f(s)\|_{L^p}}{(t-s)^{\sigma}}  \Bigr).
    \end{aligned}
\end{equation*}
 This completes the proof of Lemma \ref{CVPDE}.
	\end{proof}
	
  We will prove the {\it a priori} estimates  for the remainder terms in the rest of this section. The term $u_{R1,x_0}$ is the most difficult one, and the next two propositions will be used separately in local and global existence.

\begin{prop}\label{uR1}
Let $T \sim b^2 < 1$. Let  $(\rho, u, P)$      be a smooth enough solution of  \eqref{eqinns} on $[0, T]$ with  initial data satisfying \eqref{bdrho0}, $u_0 \in L^2$, $\rho_0 \in C^1$. Then we have
\begin{equation}\label{S4eq3}
\begin{aligned}
&\sup_{x_0} \Bigl( \sup_{t \leq T} t^{\frac{1}{2}} \|u_{R1,x_0}(t)\|_{L^\infty} + \|u_{R1,x_0}\|_{V_T^b} + \sup_{s < t < T} s^{\frac{1}{2} + \varkappa_1} \frac{\|u_{R1,x_0}(t) - u_{R1,x_0}(s)\|_{L^\infty}}{(t-s)^{\varkappa_1}} \Bigr) \\
&\quad + \sup_{t \leq T} t^{1 + \frac{\kappa}{2}} \| (\nabla |D|^\kappa u_{R1,x_0})|_{x_0=x}(t) \|_{L^\infty}
+ \sup_{t \leq T} t^{\frac{5}{12}} \|\tilde{u}_{R1}(t)\|_{L^{12}} \\
&\quad + \sup_{s < t < T} s^{\frac{5}{12} + \varkappa_2} \frac{\|\tilde{u}_{R1}(t) - \tilde{u}_{R1}(s)\|_{L^{12}}}{(t-s)^{\varkappa_2}} \leq CT^{\f1{4}} \|\rho_0\|_{C^1} \mathbf{C}^2(u),
\end{aligned}
\end{equation}
where the constant $C$ depends only on $C_1$ and $C_2$ in \eqref{bdrho0}, and we denote
\begin{equation*}
\mathbf{C}(u) := \|u\|_{X_T} + \|u_0\|_{L^2}+1.
\end{equation*}
\end{prop}

\begin{proof} We divide the proof into the following two steps:\smallskip

\noindent \textbf{i) Estimates for spatial integrability.}\smallskip

In view of \eqref{S4eq0}, we write
\begin{equation}\label{estR11n}
\begin{aligned}
u_{R1,x_0}(t,x) &= \int_0^t K_{x_0}(t-\tau) \ast \nabla \Delta^{-1} \partial_\tau \bigl( \nabla \rho_0 \cdot (u(\tau) - u(t)) \bigr) \, d\tau \\
&= \int_0^t \partial_t K_{x_0}(t-\tau) \ast \nabla \Delta^{-1} \bigl( \nabla \rho_0 \cdot (u(\tau) - u(t)) \bigr) \, d\tau \\
&\quad - K_{x_0}(t) \ast \nabla \Delta^{-1} \bigl( \nabla \rho_0 \cdot (u_0 - u(t)) \bigr)\\
&:=u_{R1,x_0}^1(t,x)+u_{R1,x_0}^2(t,x),
\end{aligned}
\end{equation}
and correspondingly we define $\tilde{u}_{R1}^i=(u_{R1,x_0}^i)|_{x_0=x}$ as in \eqref{deftilu}. Furthermore, by the formula \eqref{defKx0}, $u_{R1,x_0}^1(t,x)$ satisfies
\begin{equation*}
    u_{R1,x_0}^1(t,x)=\frac{1}{\rho_0(x_0)}\int_0^t \nabla K_{x_0}(t-\tau) \ast  \bigl( \nabla \rho_0 \cdot (u(\tau) - u(t)) \bigr) \, d\tau.
\end{equation*}
It follows from  Lemmas \ref{eng1}, \ref{eng2} and \ref{eng3} that there holds \eqref{estL2u}.
Below we will frequently use the following interpolation inequalities:
\begin{equation}\label{intpoeq}
\begin{aligned}
    &\|\nabla \Delta^{-1} u\|_{L^\infty} \lesssim \min\{\|u\|_{L^2}^{\frac{2}{3}} \|u\|_{L^\infty}^{\frac{1}{3}},\|u\|_{L^2}^{\f35}\|u\|_{L^{12}}^{\f25}\}, \quad &&\|\nabla \Delta^{-1}\nabla  u\|_{L^\infty} \lesssim\|u\|_{L^2}^{\frac{2\kappa}{2\kappa+3}}\|u\|_{\dot C^{\kappa}}^{\frac{3}{2\kappa+3}} ,\\
& \|\nabla \Delta^{-1} u\|_{L^{12}} \lesssim \min\{\|u\|_{L^2}^{\frac{5}{6}} \|u\|_{L^\infty}^{\frac{1}{6}},\|u\|_{L^2}^{\f45}\|u\|_{L^{12}}^{\f15}\},\quad &&\|\nabla \Delta^{-1}\nabla u\|_{L^{12}} \lesssim \|u\|_{L^{12}}.
\end{aligned}
\end{equation}

First we consider the case $\rho_0\in C^1$. By \eqref{estR11n} and \eqref{estL2u}, we find
\begin{align*}
\|u_{R1,x_0}^1(t)\|_{L^\infty} \lesssim& \int_0^t (t-\tau)^{-\f12} \|\rho_0\|_{C^1} \|(u(\tau) - u(t))\|_{L^\infty} \, d\tau
%&+ \bigl\| K_{x_0}(t) \ast \bigl( \nabla \rho_0 (u_0 - u(t)) \bigr) \bigr\|_{L^2}^{\frac{2}{3}}\bigl\| K_{x_0}(t) \ast \bigl( \nabla \rho_0 (u_0 - u(t)) \bigr) \bigr\|_{L^\infty}^{\frac{1}{3}}\\
\lesssim \|\rho_0\|_{C^1}\|u\|_{X_T}\int_0^t(t-\tau)^{-\f12}
\tau^{-\f12}\,d\tau
%&+\|\rho_0\|_{C^1}\|u_0\|_{L^2}\|K_{x_0}(t)\|_{L^2}^{\f13},
 \end{align*}
 which implies
\begin{equation}\label{uS4eq20}
\begin{aligned}
\sup_{x_0} \|u_{R1,x_0}^1(t)\|_{L^\infty} \lesssim   \mathbf{C}(u)\|\rho_0\|_{C^1}\|u\|_{X_T}.
\end{aligned}
\end{equation}
Along the same line, we get,  by using  \eqref{refY}, that
\begin{align*}
\|\tilde{u}_{R1}^1(t)\|_{L^{12}} \lesssim& \int_0^t (t-\tau)^{-\f12} \|\rho_0\|_{C^1}  \|u(\tau) - u(t)\|_{L^{12}}  \, d\tau 
%&+ \bigl\| \bigl((K_{x_0}(t,\cdot))\big|_{x_0=x} \ast (\nabla \rho_0 (u_0 - u(t))) \bigr) \bigr\|_{L_x^{2}}^{\frac{5}{6}} \\
%&\qquad \times \bigl\| \bigl( (K_{x_0}(t,\cdot))|_{x_0=x} \ast (\nabla \rho_0 (u_0 - u(t))) \bigr) \bigr\|_{L_x^{\infty}}^{\frac{1}{6}} \\
\lesssim\|\rho_0\|_{C^1}\|u\|_{X_T}\int_0^t(t-\tau)^{-\f12}\tau^{-\frac{5}{12}}\,d\tau%+\|\sup_{x_0\in\mathbb{R}^3}K_{x_0}(t,\cdot)\|_{L^2}^{\f16}\|u_0\|_{L^2}^{\f16},
 \end{align*}
which leads to
\begin{equation}\label{S4eq21}
\|\tilde{u}_{R1}^1(t)\|_{L^{12}}\lesssim t^{\frac{1}{12}} \mathbf{C}(u) \|\rho_0\|_{C^1}\|u\|_{X_T}.
\end{equation}

Furthermore, for the $C^{1+\kappa}$ norm, we have
\begin{equation}\label{S4eq22}
\begin{aligned}
\sup_{x_0} \| |D|^{1+\kappa} u_{R1,x_0}^1(t) \|_{L^\infty} \lesssim
& \int_0^t (t-\tau)^{ -1- \frac{\kappa}{2}} \|\rho_0\|_{C^1}\|u(\tau) - u(t)\|_{L^\infty} \, d\tau \\
%&\quad +  \|\sup_{x_0} \nabla \Delta^{-1} |D|^{1+\kappa} K_{x_0}(t) \|_{L^2} \| \nabla \rho_0 (u_0 - u(t)) \|_{L^2} \\
 \lesssim &\int_0^t (t-\tau)^{-1 - \frac{\kappa}{2}} \tau^{-\frac{1}{2}}\min\{1,(t-\tau)^{\varkappa_1}\tau^{-\varkappa_1}\}d\tau\|\rho_0\|_{C^1}\|u\|_{X_T}\\
  \lesssim& t^{-\frac{1 + \kappa}{2}} \mathbf{C}(u) \|\rho_0\|_{C^1}\|u\|_{X_T} %+t^{-\frac{3 + 2\kappa}{4}}\|\rho_0\|_{C^1}\|u_0\|_{L^2}
.
\end{aligned}
\end{equation}

For $u_{R1,x_0}^2$, by \eqref{estR11n} and \eqref{intpoeq}, we have the $L^\infty$ estimate
\begin{equation}\label{S4eq14}
    \begin{aligned}
        \|u_{R1,x_0}^2(t)\|_{L^\infty}&\lesssim \bigl\| K_{x_0}(t) \ast \bigl( \nabla \rho_0 (u_0 - u(t)) \bigr) \bigr\|_{L^2}^{\frac{2}{3}}\bigl\| K_{x_0}(t) \ast \bigl( \nabla \rho_0 (u_0 - u(t)) \bigr) \bigr\|_{L^\infty}^{\frac{1}{3}}\\
        &\lesssim \|\rho_0\|_{C^1}\|u_0\|_{L^2}\|K_{x_0}(t)\|_{L^2}^{\f13}\lesssim t^{-\frac{1}{4}}\|\rho_0\|_{C^1}\mathbf{C}(u),
    \end{aligned}
\end{equation}
the corresponding $L^{12}$ estimate
\begin{equation}\label{S4eq15}
    \begin{aligned}
        \|\tilde{u}_{R1}^2(t)\|_{L^{12}}&\lesssim \bigl\| \bigl((K_{x_0}(t,\cdot))\big|_{x_0=x} \ast (\nabla \rho_0 (u_0 - u(t))) \bigr) \bigr\|_{L_x^{2}}^{\frac{5}{6}} \bigl\| \bigl( (K_{x_0}(t,\cdot))|_{x_0=x} \ast (\nabla \rho_0 (u_0 - u(t))) \bigr) \bigr\|_{L_x^{\infty}}^{\frac{1}{6}}\\
        &\lesssim\|\sup_{x_0\in\mathbb{R}^3}K_{x_0}(t,\cdot)\|_{L^2}^{\f16}\|\rho_0\|_{C^1}\|u_0\|_{L^2}\lesssim t^{-\frac{1}{8}}  \|\rho_0\|_{C^1}\mathbf{C}(u),
    \end{aligned}
\end{equation}
and the $C^{1+\kappa}$ estimate
\begin{equation}\label{S4eq16}
    \begin{aligned}
        \sup_{x_0}\||D|^{1+\kappa}u_{R1,x_0}^2(t)\|_{L^\infty}&\lesssim \|\sup_{x_0} \nabla \Delta^{-1} |D|^{1+\kappa} K_{x_0}(t) \|_{L^2} \| \nabla \rho_0 (u_0 - u(t)) \|_{L^2}\\
        &\lesssim\|\sup_{x_0} \nabla \Delta^{-1} |D|^{1+\kappa} K_{x_0}(t) \|_{L^2} \| \nabla \rho_0 (u_0 - u(t)) \|_{L^2}\\
        &\lesssim t^{-\frac{3 + 2\kappa}{4}}\|\rho_0\|_{C^1}\mathbf{C}(u).
    \end{aligned}
\end{equation}

So combine \eqref{uS4eq20} and \eqref{S4eq14}, we have
\begin{equation}\label{S4eq68}
    \sup_{x_0}\sup_{t\leq T}t^{\frac{1}{2}}\|u_{R1,x_0}(t)\|_{L^\infty}\lesssim T^{\frac{1}{4}}\|\rho_0\|_{C^1}\mathbf{C}(u).
\end{equation}
Combine \eqref{S4eq21} and \eqref{S4eq15}, we have
\begin{equation}\label{S4eq69}
    \sup_{t\leq T}t^{\frac{5}{12}}\|\tilde{u}_{R1}(t)\|_{L^{12}}\lesssim T^{\frac{7}{24}}\|\rho_0\|_{C^1}\mathbf{C}(u).
\end{equation}
Combine \eqref{S4eq22} and \eqref{S4eq16}, we have
\begin{equation}\label{S4eq70}
    \sup_{x_0}\sup_{t\leq T}t^{1+\frac{\kappa}{2}}\||D|^{1+\kappa}u_{R1,x_0}(t)\|_{L^\infty}\lesssim T^{\frac{1}{4}}\|\rho_0\|_{C^1}\mathbf{C}(u).
\end{equation}

Note that
\begin{equation}\label{uR1bmo}
\|u_{R1,x_0}\|_{V_T^b} \lesssim T^{\f1{12}}\sup_{t \leq T} t^{\frac{5}{12}} \|u_{R1,x_0}\|_{L^\infty}
 \lesssim \sup_{t \leq T} t^{\frac{5}{12}} \|u_{R1,x_0}\|_{L^\infty}, \quad \forall T < 1.
\end{equation}
Hence, similar estimates as \eqref{uR1Linf} also implies the estimate for $\sup_{x_0} \|u_{R1,x_0}\|_{V_T^b}$.\\
\textbf{ii) Estimates for time differences.}\\
We still use \eqref{estR11n}. For $u_{R1,x_0}^1$ and $s<t<T$, let $a=t-s$, we denote
\begin{align*}
    u_{R1,x_0}^1(t,x)-u_{R1,x_0}^1(s,x)=&\frac{1}{\rho_0(x_0)}\int_s^t\nabla K_{x_0}(t-\tau)\ast (\nabla\rho_0 (u(\tau)-u(t)))d\tau\\
    &+\frac{1}{\rho_0(x_0)}\int_0^s\nabla K_{x_0}(t-\tau)\ast (\nabla\rho_0 (u(t)-u(s)))d\tau\\
    &+\frac{1}{\rho_0(x_0)}\int_0^s\nabla \delta_a^t K_{x_0}(s-\tau)\ast (\nabla\rho_0 (u(\tau)-u(s)))d\tau.
\end{align*}
Then separately, we have the $L^\infty$ estimate
\begin{align*}
   \|u_{R1,x_0}^1(t)-u_{R1,x_0}^1(s)\|_{L^\infty} \lesssim &\int_s^t(t-\tau)^{-\frac{1}{2}}(t-\tau)^{\varkappa_1}\tau^{-\frac{1}{2}-\varkappa_1}d\tau\|\rho_0\|_{C^1}\|u\|_{X_T}\\
   & +\int_0^s(t-\tau)^{-\frac{1}{2}}(t-s)^{\varkappa_1}s^{-\frac{1}{2}-\varkappa_1}d\tau\|\rho_0\|_{C^1}\|u\|_{X_T}\\
   &+\int_0^s(s-\tau)^{-\frac{1}{2}}\min\{1,a(s-\tau)^{-1}\}\\
   &\qquad\times\tau^{-\frac{1}{2}}\min\{1,(s-\tau)^{\varkappa_1}\tau^{-\varkappa_1}\}d\tau\|\rho_0\|_{C^1}\|u\|_{X_T}\\
   \lesssim &a^{\varkappa_1}s^{-\varkappa_1}\|\rho_0\|_{C^1}\|u\|_{X_T},
\end{align*}
and the $L^{12}$ estimate
\begin{align*}
   \|\tilde u_{R1}^1(t)-\tilde u_{R1}^1(s)\|_{L^{12}} \lesssim &\int_s^t(t-\tau)^{-\frac{1}{2}}(t-\tau)^{\varkappa_2}\tau^{-\frac{5}{12}-\varkappa_2}d\tau\|\rho_0\|_{C^1}\|u\|_{X_T}\\
   & +\int_0^s(t-\tau)^{-\frac{1}{2}}(t-s)^{\varkappa_2}s^{-\frac{5}{12}-\varkappa_2}d\tau\|\rho_0\|_{C^1}\|u\|_{X_T}\\
   &+\int_0^s(s-\tau)^{-\frac{1}{2}}\min\{1,a(s-\tau)^{-1}\}\\
   &\qquad\times\tau^{-\frac{5}{12}}\min\{1,(s-\tau)^{\varkappa_2}\tau^{-\varkappa_2}\}d\tau\|\rho_0\|_{C^1}\|u\|_{X_T}\\
   \lesssim &a^{\varkappa_2}s^{\frac{1}{12}-\varkappa_2}\|\rho_0\|_{C^1}\|u\|_{X_T}.
\end{align*}
Combine the two estimates above, we can obtain that
\begin{align}
    &\sup_{x_0}\sup_{s\leq T\leq T}s^{\frac{1}{2}+\varkappa_1}\frac{\|u_{R1,x_0}^1(t)-u_{R1,x_0}^1(s)\|_{L^{\infty}}}{(t-s)^{\varkappa_1}}\lesssim T^{\frac{1}{2}}\|\rho_0\|_{C^1}\|u\|_{X_T},\label{S4eq41}\\
    &\sup_{s\leq T\leq T}s^{\frac{5}{12}+\varkappa_2}\frac{\|\tilde u_{R1}^1(t)-u_{R1}^1(s)\|_{L^{12}}}{(t-s)^{\varkappa_2}}\lesssim T^{\frac{1}{2}}\|\rho_0\|_{C^1}\|u\|_{X_T}\nonumber
\end{align}

For $u_{R1,x_0}^2$, by \eqref{estR11n}, we have
\begin{align}
   u_{R1,x_0}^2(t,x)-u_{R1,x_0}^2(s,x)=&\left(K_{x_0}(t)-K_{x_0}(s)\right)\ast \nabla\Delta^{-1}\left(\nabla\rho_0(u_0-u(t))\right)(x)\label{S4eq40}\\
   &+K_{x_0}(s)\ast \nabla\Delta^{-1}\left(\nabla\rho_0(u(t)-u(s))\right)(x)\nonumber\\
   =&\frac{1}{\rho_0(x_0)}\int_s^t\nabla K_{x_0}(\gamma)d\gamma\ast \left(\nabla\rho_0(u_0-u(t))\right)(x)\nonumber\\
   &+\nabla\Delta^{-1}K_{x_0}(s)\ast \left(\nabla\rho_0(u(t)-u(s))\right)(x).\nonumber
\end{align}
 While it follows  from \eqref{estL2u} and Lemma \ref{eng3} that for $s<t,$
 \begin{align*} \|u(t) - u(s)\|_{L^2}\leq \int_s^t\|u_\tau(\tau)\|_{L^2}\,d\tau\leq \int_s^t\tau^{-1}\,d\tau\|\tau u_\tau(\tau)\|_{L^\infty_t(L^2)}\leq s^{-1}(t-s)\|u_0\|_{L^2}, \end{align*}
so that
\beq\label{S4eqp}
\|u(t) - u(s)\|_{L^2}\lesssim \min\bigl\{1, s^{-1}(t-s)\bigr\}\|u_0\|_{L^2}.
\eeq
By \eqref{S4eq40}, \eqref{S4eqp} and Lemma \ref{eng1}, we have
\begin{equation*}
    \begin{aligned}
        \|u_{R1,x_0}^2(t)-u_{R1,x_0}^2(s)\|_{L^\infty}\lesssim& \int_s^t\|\sup_{x_0}|\nabla K_{x_0}(\gamma)|\|_{L^2}d\gamma\|\rho_0\|_{C^1}\|u_0\|_{L^2}\\
        &+\|\sup_{x_0}|\nabla\Delta^{-1}K_{x_0}(s)|\|_{L^2}\|\rho_0\|_{C^1}\|u_0\|_{L^2}\min\{1,(t-s)s^{-1}\}\\
        \lesssim &s^{-\frac{1}{4}}\min\{1,(t-s)s^{-1}\}\|\rho_0\|_{C^1}\|u_0\|_{L^2},\\
        \|\tilde u_{R1}^2(t)-\tilde u_{R1}^2(s)\|_{L^{12}}\lesssim &\int_s^t\|\sup_{x_0}|\nabla K_{x_0}(\gamma)|\|_{L^{\frac{12}{7}}}d\gamma\|\rho_0\|_{C^1}\|u_0\|_{L^2}\\
        &+\|\sup_{x_0}|\nabla\Delta^{-1}K_{x_0}(s)|\|_{L^{\frac{12}{5}}}\|\rho_0\|_{C^1}\|u_0\|_{L^2}\min\{1,(t-s)s^{-1}\}\\
        \lesssim& s^{-\frac{1}{8}}\min\{1,(t-s)s^{-1}\}\|\rho_0\|_{C^1}\|u_0\|_{L^2}.
    \end{aligned}
\end{equation*}
These lead to
\begin{equation}\label{S4eq42}
    \begin{aligned}
        &\sup_{x_0}\sup_{s<t<T}s^{\frac{1}{2}+\varkappa_1}\frac{\|u_{R1,x_0}^2(t)-u_{R1,x_0}^2(s)\|_{L^\infty}}{(t-s)^{\varkappa_1}}\lesssim T^{\frac{1}{4}}\|\rho_0\|_{C^1}\|u_0\|_{C^2}\\
        &\sup_{s<t<T}s^{\frac{5}{12}+\varkappa_2}\frac{\|\tilde u_{R1}^2(t)-\tilde u_{R1}^2(s)\|_{L^{12}}}{(t-s)^{\varkappa_2}}\lesssim T^{\frac{7}{24}}\|\rho_0\|_{C^1}\|u_0\|_{C^2}
    \end{aligned}
\end{equation}

So \eqref{S4eq41} and \eqref{S4eq42} lead to
\begin{equation}\label{S4eq71}
    \begin{aligned}
        &\sup_{x_0}\sup_{s<t<T}s^{\frac{1}{2}+\varkappa_1}\frac{\|u_{R1,x_0}(t)-u_{R1,x_0}(s)\|_{L^\infty}}{(t-s)^{\varkappa_1}}\lesssim T^{\frac{1}{4}}\|\rho_0\|_{C^1}\mathbf{C}(u),\\
        &\sup_{s<t<T}s^{\frac{5}{12}+\varkappa_2}\frac{\|\tilde u_{R1}(t)-\tilde u_{R1}(s)\|_{L^{12}}}{(t-s)^{\varkappa_2}}\lesssim T^{\frac{7}{24}}\|\rho_0\|_{C^1}\mathbf{C}(u).
    \end{aligned}
\end{equation}

Now combine \eqref{S4eq68}, \eqref{S4eq69}, \eqref{S4eq70}, \eqref{uR1bmo} and \eqref{S4eq71}, we have proved \eqref{S4eq3}. 
\end{proof}
\begin{prop}\label{Prop4.2}
    Under the assumptions of Proposition \ref{uR1}, if additionally $\rho_0\in C^2$, then we have
\begin{equation}\label{S4eq12}
\begin{aligned}
&\sup_{x_0} \Bigl( \sup_{t \leq T} t^{\frac{1}{2}} \|u_{R1,x_0}(t)\|_{L^\infty} + \|u_{R1,x_0}\|_{V_T^b} + \sup_{s < t < T} s^{\frac{1}{2} + \varkappa_1} \frac{\|u_{R1,x_0}(t) - u_{R1,x_0}(s)\|_{L^\infty}}{(t-s)^{\varkappa_1}} \Bigr) \\
&\quad + \sup_{t \leq T} t^{1 + \frac{\kappa}{2}} \| (\nabla |D|^\kappa u_{R1,x_0})|_{x_0=x}(t) \|_{L^\infty}
+ \sup_{t \leq T} t^{\frac{5}{12}} \|\tilde{u}_{R1}(t)\|_{L^{12}} \\
&\quad + \sup_{s < t < T} s^{\frac{5}{12} + \varkappa_2} \frac{\|\tilde{u}_{R1}(t) - \tilde{u}_{R1}(s)\|_{L^{12}}}{(t-s)^{\varkappa_2}} \\
&\qquad\leq C\|\rho_0\|_{C^2}^2\mathbf{C}^3(u)\left(T^{\frac{1}{4}}(\|u\|_{X_T}+\|\rho_0-1\|_{L^\infty})+T^{\frac{3}{8}}\|u\|_{X_T}^{\f12}+T^{\frac{5}{8}}\right).
\end{aligned}
\end{equation}
\end{prop}
\begin{proof}
Note that from the decomposition \eqref{estR11n}, the estimates of $u_{R1,x_0}^1$ follows from \eqref{uS4eq20}, \eqref{S4eq21}, \eqref{S4eq22}, \eqref{uR1bmo} and \eqref{S4eq41}. We only need to get new estimates for $u_{R1,x_0}^2$. 

\noindent \textbf{i) Estimates for spatial integrability.}\smallskip

We will use the following formula for $u_{R1,x_0}^2$. By \eqref{eqinns}, $u_{R1,x_0}^2(t,x)$ satisfies
  \begin{align}
        u_{R1,x_0}^2(t,x)=&-\nabla\Delta^{-1}K_{x_0}(t)\ast\left(\nabla\rho_0\cdot(\rho_0u_0-\rho(t)u(t))\right)+\nabla\Delta^{-1}K_{x_0}(t)\ast\left(\nabla\rho_0\cdot(\rho_0-1) u_0\right)\nonumber\\
        &-\nabla\Delta^{-1}K_{x_0}(t)\ast\left(\nabla\rho_0\cdot(\rho(t)-1)u(t)\right)\nonumber\\
        =&\,\nabla\Delta^{-1}K_{x_0}(t)\ast\left(\nabla\rho_0\cdot\int_0^t\left(\Delta u-\nabla\cdot(\rho u\otimes u)+\nabla P\right)(\tau,x)d\tau\right)\label{S4eq11} \\
        &+\nabla\Delta^{-1}K_{x_0}(t)\ast\left(\nabla\rho_0\cdot(\rho_0-1)u_0\right)-\nabla\Delta^{-1}K_{x_0}(t)\ast\left(\nabla\rho_0\cdot(\rho(t)-1)u(t)\right)\nonumber\\
        &:=\sum_{i=1}^3u_{R1,x_0}^{2i},\nonumber
    \end{align}
and we denote $\tilde u_{R1}^{2i}=u_{R1,x_0}^{2i}|_{x_0=x}$. 
Note that for the error terms come from density, by \eqref{transol} and Lemma \ref{eng1}, we have
\begin{equation}\label{S4eq55}
    \begin{aligned}
        &\sup_{x_0}\left(\|u_{R1,x_0}^{22}\|_{L^\infty}+\|u_{R1,x_0}^{23}\|_{L^\infty}\right)\lesssim \|\sup_{x_0}|\nabla\Delta^{-1}K_{x_0}|(t)\|_{L^2}\|\rho_0-1\|_{L^\infty}\|u_0\|_{L^2}\\
        &\qquad\lesssim t^{-\frac{1}{4}}\|u_0\|_{L^2}\|\rho_0-1\|_{L^\infty},\\
        &\left(\|\tilde u_{R1}^{22}\|_{L^{12}}+\|\tilde u_{R1}^{23}\|_{L^{12}}\right)\lesssim \|\sup_{x_0}|\nabla\Delta^{-1}K_{x_0}|(t)\|_{L^{\frac{12}{7}}}\|\rho_0-1\|_{L^\infty}\|u_0\|_{L^2}\\
        &\qquad\lesssim t^{-\frac{1}{8}}\|u_0\|_{L^2}\|\rho_0-1\|_{L^\infty},\\
        &\sup_{x_0}\left(\||D|^{1+\kappa}u_{R1,x_0}^{22}\|_{L^\infty}+\||D|^{1+\kappa}u_{R1,x_0}^{23}\|_{L^\infty}\right)\lesssim \|\sup_{x_0}||D|^{1+\kappa}\nabla\Delta^{-1}K_{x_0}|(t)\|_{L^2}\|\rho_0-1\|_{L^\infty}\|u_0\|_{L^2}\\
        &\qquad\lesssim t^{-\frac{3+2\kappa}{4}}\|u_0\|_{L^2}\|\rho_0-1\|_{L^\infty}.
    \end{aligned}
\end{equation}
For $u_{R1,x_0}^{21}$, we further decompose $u_{R1,x_0}^{21}=u_{R1,x_0}^{21,m}+u_{R1,x_0}^{21,P}$, with
\begin{align}
    &u_{R1,x_0}^{21,m}=\nabla\Delta^{-1}K_{x_0}(t)\ast\left(\nabla\rho_0\cdot\int_0^t\left(\Delta u-\nabla\cdot(\rho u\otimes u)\right)(\tau,x)d\tau\right),\label{S4eq60}\\
    &u_{R1,x_0}^{21,P}=\nabla\Delta^{-1}K_{x_0}(t)\ast\left(\nabla\rho_0\cdot\int_0^t\nabla P(\tau,x)d\tau\right).\nonumber
\end{align}
For $u_{R1,x_0}^{21,m}$, using integration by parts and \eqref{esthtk}, we have the $L^\infty$ estimate
\begin{equation}\label{S4eq53}
    \begin{aligned}
        \sup_{x_0}\|u_{R1,x_0}^{21,m}\|_{L^\infty}&\lesssim \int_0^t\bigl\|\sup_{x_0}|\nabla\Delta^{-1}\nabla K_{x_0}(t)\|_{L^{\frac{4}{3}}} \|\nabla\rho_0\|_{L^\infty}\left\|\nabla u(\tau)-(\rho u\otimes u)(\tau)\right\|_{L^{4}}d\tau\\
        &\quad+\int_0^t\sup_{x_0}|\nabla\Delta^{-1}K_{x_0}(t)|\|_{L^2}\|\rho_0\|_{C^2}\|\nabla u(\tau)-(\rho u\otimes u)(\tau)\|_{L^2}d\tau \\
        &\lesssim \|\rho_0\|_{C^1}^2\int_0^tt^{-\frac{3}{8}}\tau^{-\frac{3}{4}}d\tau \mathbf{C}(u)^2\|u\|_{X_T}^{\frac{1}{2}}+\|\rho_0\|_{C^2}^2t^{-\frac{1}{4}}\int_0^t\tau^{-\frac{1}{2}}d\tau \mathbf{C}(u)^2\\
        &\lesssim \|\rho_0\|_{C^2}^2\mathbf{C}^2(u)\left(t^{-\frac{1}{8}}\|u\|_{X_T}^{\frac{1}{2}}+t^{\f14}\right),
    \end{aligned}
\end{equation}
the $L^{12}$ estimate
\begin{equation}\label{S4eq18}
    \begin{aligned}
        \|\tilde u_{R1}^{21,m}(t)\|_{L^{12}}&\lesssim \int_0^t\bigl\|K_{x_0}(t)\ast \left(\nabla\rho_0\cdot\left(\nabla u(\tau)-(\rho u\otimes u)(\tau) \right)\right)\bigr\|_{L^{12}}d\tau\\
        &\quad+\int_0^t\|\sup_{x_0}|\nabla\Delta^{-1}K_{x_0}(t)|\|_{L^{\frac{12}{7}}} \|\rho_0\|_{C^2}\|\nabla u(\tau)-(\rho u\otimes u)(\tau)\|_{L^2}d\tau \\
        &\lesssim \|\rho_0\|_{C^1}^2\int_0^t\tau^{-\frac{11}{12}}d\tau \mathbf{C}^2(u)\|u\|_{X_T}^{\frac{5}{6}}+\|\rho_0\|_{C^2}^2\int_0^tt^{-\frac{1}{8}}\tau^{-\frac{1}{2}}d\tau\mathbf{C}^2(u)\\
        &\lesssim \|\rho_0\|_{C^2}^2\mathbf{C}^2(u)\left(t^{\frac{1}{12}}\|u\|_{X_T}^{\frac{5}{6}}+t^{\frac{3}{8}}\right),
    \end{aligned}
\end{equation}
and the $C^{1+\kappa}$ estimate
\begin{align}
        &\sup_{x_0}\||D|^{1+\kappa}u_{R1,x_0}^{21,m}(t)\|_{L^\infty \label{S4eq19}}\\
       &\lesssim \int_0^t\|\sup_{x_0}\nabla\Delta^{-1}\nabla|D|^{1+\kappa}K_{x_0}(t)\|_{L^{\frac{4}{3}}}\|\rho_0\|_{C^1}^2\left(\|\nabla u(\tau)\|_{L^4}+\|(\rho u\otimes u)(\tau)\|_{L^4}\right)\nonumber\\
        &\quad+\int_0^t\|\sup_{x_0}\nabla\Delta^{-1}|D|^{1+\kappa}K_{x_0}(t)\|_{L^{\frac{4}{3}}}\|\rho_0\|_{C^2}^2\left(\|\nabla u(\tau)\|_{L^4}+\|(\rho u\otimes u)(\tau)\|_{L^4}\right)\nonumber\\
        &\lesssim \|\rho_0\|_{C^1}^2\int_0^tt^{-\frac{7+4\kappa}{8}}\tau^{-\frac{3}{4}}d\tau\mathbf{C}^2(u)\|u\|_{X_T}^{\frac{1}{2}}+\|\rho_0\|_{C^2}\int_0^tt^{-\frac{3+4\kappa}{8}}\tau^{-\frac{3}{4}}\mathbf{C}^2(u)\|u\|_{X_T}^{\frac{1}{2}}\nonumber\\
        &\lesssim t^{-\frac{5+4\kappa}{8}}\|\rho_0\|_{C^2}^2\mathbf{C}^2(u)\|u\|_{X_T}^{\frac{1}{2}}.\nonumber
    \end{align}
    For $u_{R1,x_0}^{21,P}$, since by \eqref{eqinns} and due to the divergence free condition of $u$ and $\partial_tu$, we have
    \begin{equation}\label{S4eq61}
        \begin{aligned}
            -\int_0^t\nabla P(\tau)d\tau&=\int_0^t\left(\nabla\Delta^{-1}\nabla\cdot\left(\rho\partial_tu+\rho u\cdot\nabla u\right)\right)(\tau)d\tau\\
            &=\int_0^t\nabla\Delta^{-1}\nabla\cdot\left((\rho-\rho_0)\partial_tu\right)(\tau)d\tau+\int_0^t\nabla\Delta^{-1}\nabla\cdot\left(\rho u\cdot\nabla u\right)(\tau)d\tau\\
            &\qquad+\nabla\Delta^{-1}\nabla\cdot \big((\rho_0-1)(u(t)-u_0)\big)\\
            &:=P_1(t)+P_2(t)+P_3(t).
        \end{aligned}
    \end{equation}
   For $P_1(t)$, we use Proposition \ref{Holtrho} and Lemma \ref{eng3} to show that
   \begin{align}
       &\left\|\nabla\Delta^{-1}K_{x_0}(t)\ast \left(\nabla\rho_0 P_1(t)\right)\right\|_{L^\infty}\lesssim \left\|\sup_{x_0}|\nabla\Delta^{-1}K_{x_0}(t)| \right\|_{L^2}\int_0^t\|\rho(\tau)-\rho_0\|_{L^\infty}\|\partial_\tau u(\tau)\|_{L^2}d\tau\label{S4eq50}\\
       &\qquad\lesssim t^{-\frac{1}{4}}\int_0^t\tau^{\frac{3}{8}}\tau^{-1}d\tau\|\rho_0\|_{C^1}\mathbf{C}^2(u)\lesssim t^{\frac{1}{8}}\|\rho_0\|_{C^1}\mathbf{C}^2(u),\nonumber\\
       &\left\|\left(\nabla\Delta^{-1}K_{x_0}(t)\ast \left(\nabla\rho_0 P_1(t)\right)\right)\big|_{x_0=x}\right\|_{L^{12}}\lesssim \left\|\sup_{x_0}|\nabla\Delta^{-1}K_{x_0}(t)| \right\|_{L^\frac{12}{7}}\int_0^t\|\rho(\tau)-\rho_0\|_{L^\infty}\|\partial_\tau u(\tau)\|_{L^2}d\tau\nonumber\\
       &\qquad\lesssim t^{-\frac{1}{8}}\int_0^t\tau^{\frac{3}{8}}\tau^{-1}d\tau\|\rho_0\|_{C^1}\mathbf{C}^2(u)\lesssim t^{\frac{1}{4}}\|\rho_0\|_{C^1}\mathbf{C}^2(u),\nonumber\\
       &\left\||D|^{1+\kappa}\nabla\Delta^{-1}K_{x_0}(t)\ast \left(\nabla\rho_0 P_1(t)\right)\right\|_{L^\infty}\lesssim \left\|\sup_{x_0}||D|^{1+\kappa}\nabla\Delta^{-1}K_{x_0}(t)| \right\|_{L^2}\int_0^t\|\rho(\tau)-\rho_0\|_{L^\infty}\|\partial_\tau u(\tau)\|_{L^2}d\tau\nonumber\\
       &\qquad\lesssim t^{-\frac{3+2\kappa}{4}}\int_0^t\tau^{\frac{3}{8}}\tau^{-1}d\tau\|\rho_0\|_{C^1}\mathbf{C}^2(u)\lesssim t^{-\frac{3+4\kappa}{8}}\|\rho_0\|_{C^1}\mathbf{C}^2(u).\nonumber
   \end{align}
   For $P_3(t)$, by Lemma \ref{eng1}, we simply have
   \begin{align}
       &\left\|\nabla\Delta^{-1}K_{x_0}(t)\ast \left(\nabla\rho_0 P_3(t)\right)\right\|_{L^\infty}\lesssim \left\|\sup_{x_0}|\nabla\Delta^{-1}K_{x_0}(t)| \right\|_{L^2}\|\rho_0-1\|_{L^\infty}\|u_0\|_{L^2}\label{S4eq54}\\
       &\qquad\lesssim t^{-\frac{1}{4}}\|\rho_0-1\|_{L^\infty}\mathbf{C}^2(u),\nonumber\\
       &\left\|\left(\nabla\Delta^{-1}K_{x_0}(t)\ast \left(\nabla\rho_0 P_3(t)\right)\right)\big|_{x_0=x}\right\|_{L^{12}}\lesssim \left\|\sup_{x_0}|\nabla\Delta^{-1}K_{x_0}(t)| \right\|_{L^\frac{12}{7}}\|\rho_0-1\|_{L^\infty}\|u_0\|_{L^2}\nonumber\\
       &\qquad\lesssim t^{-\frac{1}{8}}\|\rho_0-1\|_{L^\infty}\mathbf{C}^2(u),\nonumber\\
       &\left\||D|^{1+\kappa}\nabla\Delta^{-1}K_{x_0}(t)\ast \left(\nabla\rho_0 P_3(t)\right)\right\|_{L^\infty}\lesssim \left\|\sup_{x_0}||D|^{1+\kappa}\nabla\Delta^{-1}K_{x_0}(t)| \right\|_{L^2}\|\rho_0-1\|_{L^\infty}\|u_0\|_{L^2}d\tau\nonumber\\
       &\qquad\lesssim t^{-\frac{3+2\kappa}{4}}\|\rho_0-1\|_{L^\infty}\mathbf{C}^2(u).\nonumber
   \end{align}
   Finally, for $P_2$, we use $\rho_0$ to approximate $\rho$, and write
   \begin{equation}\label{S4eq63}
       P_2(t)=\int_0^t\nabla\Delta^{-1}\nabla\cdot((\rho(\tau)-\rho_0)u\cdot\nabla u)(\tau)d\tau+\int_0^t\nabla\Delta^{-1}\nabla\cdot(\rho_0 u\cdot\nabla u)(\tau)d\tau:=P_2^1(t)+P_2^2(t).
   \end{equation}
   For $P_2^1$, since by Lemma \ref{eng1} and \ref{eng2}, $\|u\cdot\nabla u\|_{L^2}(\tau)\lesssim \tau^{-1}$, we can prove similar estimate as \eqref{S4eq50}, and obtain
   \begin{align}
       &\left\|\nabla\Delta^{-1}K_{x_0}(t)\ast \left(\nabla\rho_0 P_2^1(t)\right)\right\|_{L^\infty}\lesssim t^{\frac{1}{8}}\|\rho_0\|_{C^1}\mathbf{C}^2(u),\label{S4eq51}\\
       &\left\|\left(\nabla\Delta^{-1}K_{x_0}(t)\ast \left(\nabla\rho_0 P_2^1(t)\right)\right)\big|_{x_0=x}\right\|_{L^{12}} \lesssim t^{\frac{1}{4}}\|\rho_0\|_{C^1}\mathbf{C}^2(u),\nonumber\\
       &\left\||D|^{1+\kappa}\nabla\Delta^{-1}K_{x_0}(t)\ast \left(\nabla\rho_0 P_2^1(t)\right)\right\|_{L^\infty}\lesssim t^{-\frac{3+4\kappa}{8}}\|\rho_0\|_{C^1}\mathbf{C}^2(u).\nonumber
   \end{align}
   For $P_2^2$, also by the divergence free condition of $u$, we note that
   \begin{equation*}
       P_2^2(t)=\int_0^t\nabla\Delta^{-1}\nabla\cdot\operatorname{div}(\rho_0u\otimes u)(\tau)d\tau-\int_0^t\nabla\Delta^{-1}\nabla\cdot(u\nabla\rho_0\cdot u)(\tau)d\tau, 
   \end{equation*}
   which implies
   \begin{align}
       \nabla\Delta^{-1}K_{x_0}(t)\ast \nabla\rho_0 P_2^2(t)=&\nabla\Delta^{-1}\nabla K_{x_0}(t)\ast \left(\nabla\rho_0\int_0^t\Delta^{-1}\nabla\cdot\operatorname{div}(\rho_0 u\otimes u)(\tau)d\tau\right)\label{S4eq64}\\
       &-\nabla\Delta^{-1} K_{x_0}(t)\ast \left(\nabla\rho_0\int_0^t\nabla\Delta^{-1}\nabla\cdot(u\nabla\rho_0\cdot u)(\tau)d\tau\right)\nonumber\\
       &-\nabla\Delta^{-1} K_{x_0}(t)\ast \left(\nabla^2\rho_0\int_0^t\left(\Delta^{-1}\nabla\cdot\operatorname{div}(\rho_0 u\otimes u)d\tau\right)\right)\nonumber.
   \end{align}
   By this way, we can prove that
   \begin{align}
       &\left\|\nabla\Delta^{-1}K_{x_0}(t)\ast \left(\nabla\rho_0 P_2^2(t)\right)\right\|_{L^\infty}\lesssim \left\|\sup_{x_0}|\nabla^2\Delta^{-1}K_{x_0}(t)| \right\|_{L^2}\int_0^t\|\rho_0\|_{C^1}^2\left(\|u\otimes u(\tau)\|_{L^2}  \right)d\tau\label{S4eq52}\\
       &\qquad\qquad+ \left\|\sup_{x_0}|\nabla\Delta^{-1}K_{x_0}(t)| \right\|_{L^2}\|\rho_0\|_{C^2}^2\int_0^t\|u(\tau)\|_{L^\infty}\|u(\tau)\|_{L^2}d\tau \nonumber\\
       &\qquad\lesssim t^{-\frac{3}{4}}\int_0^t\tau^{-\frac{1}{2}}d\tau\|\rho_0\|_{C^1}\mathbf{C}^2(u)\|u\|_{X_T}+t^{-\frac{1}{4}}\|\rho_0\|_{C^2}\int_0^t\tau^{-\frac{1}{2}}d\tau\mathbf{C}^2(u)\|u\|_{X_T}\nonumber\\
       &\qquad\lesssim t^{-\frac{1}{4}}\|\rho_0\|_{C^2}^2\mathbf{C}^2(u)\|u\|_{X_T},\nonumber\\
       &\left\|\left(\nabla\Delta^{-1}K_{x_0}(t)\ast \left(\nabla\rho_0 P_2^2(t)\right)\right)\big|_{x_0=x}\right\|_{L^{12}}\lesssim \left\|\sup_{x_0}|\nabla^2\Delta^{-1}K_{x_0}(t)| \right\|_{L^\frac{12}{7}}\int_0^t\|\rho_0\|_{C^1}^2\|u\otimes u(\tau)\|_{L^2}d\tau\nonumber\\
       &\qquad\qquad+ \left\|\sup_{x_0}|\nabla\Delta^{-1}K_{x_0}(t)| \right\|_{L^\frac{12}{7}}\|\rho_0\|_{C^2}^2\int_0^t\|u(\tau)\|_{L^\infty}\|u(\tau)\|_{L^2}d\tau\nonumber\\
       &\qquad\lesssim t^{-\frac{5}{8}}\int_0^t\tau^{-\frac{1}{2}}d\tau\|\rho_0\|_{C^1}\mathbf{C}^2(u)\|u\|_{X_T}+t^{-\frac{1}{8}}\|\rho_0\|_{C^2}\int_0^t\tau^{-\frac{1}{2}}d\tau\mathbf{C}^2(u)\|u\|_{X_T}\nonumber\\
       &\qquad\lesssim t^{-\frac{1}{8}}\|\rho_0\|_{C^1}\mathbf{C}^2(u)\|u\|_{X_T},\nonumber\\
       &\left\||D|^{1+\kappa}\nabla\Delta^{-1}K_{x_0}(t)\ast \left(\nabla\rho_0 P_2^2(t)\right)\right\|_{L^\infty}\lesssim \left\|\sup_{x_0}||D|^{1+\kappa}\nabla^2\Delta^{-1}K_{x_0}(t)| \right\|_{L^2}\int_0^t\|\rho_0\|_{C^1}^2\|u\otimes u(\tau)\|_{L^2}d\tau\nonumber\\
       &\qquad\qquad+\left\|\sup_{x_0}||D|^{1+\kappa}\nabla\Delta^{-1}K_{x_0}(t)| \right\|_{L^2}\|\rho_0\|_{C^2}^2\int_0^t\|u(\tau)\|_{L^\infty}\|u(\tau)\|_{L^2}d\tau\nonumber\\
       &\qquad\lesssim t^{-\frac{5+2\kappa}{4}}\int_0^t\tau^{-\frac{1}{2}}d\tau\|\rho_0\|_{C^1}\mathbf{C}^2(u)\|u\|_{X_T}+t^{-\frac{3+2\kappa}{4}}\|\rho_0\|_{C^2}^2\int_0^t\tau^{-\frac{1}{2}}d\tau\mathbf{C}^2(u)\|u\|_{X_T}\nonumber\\
       &\qquad\lesssim t^{-\frac{3+2\kappa}{4}}\|\rho_0\|_{C^2}^2\mathbf{C}^2(u)\|u\|_{X_T}.\nonumber
   \end{align}
   Now we combine \eqref{S4eq50}, \eqref{S4eq54}, \eqref{S4eq51} and \eqref{S4eq52}, then we can obtain
   \begin{equation}\label{S4eq56}
       \begin{aligned}
           &\sup_{x_0}\sup_{t\leq T}t^{\frac{1}{2}}\|u_{R1,x_0}^{21,P}(t)\|_{L^\infty}\lesssim \|\rho_0\|_{C^2}^2\mathbf{C}^2(u)\left(T^{\frac{5}{8}}+T^{\frac{3}{8}}\|u\|_{X_T}^{\frac{1}{2}}+T^{\frac{1}{4}}\left(\|\rho_0-1\|_{L^\infty}+\|u\|_{X_T}\right) \right),\\
           &\sup_{t\leq T}t^{\frac{5}{12}}\|\tilde{ u}_{R1}^{21,P}(t)\|_{L^{12}}\lesssim \|\rho_0\|_{C^2}^2\mathbf{C}^2(u)\left(T^{\frac{2}{3}}+T^{\frac{5}{12}}\|u\|_{X_T}^{\frac{1}{2}}+T^{\frac{7}{24}}\left(\|\rho_0-1\|_{L^\infty}+\|u\|_{X_T}\right) \right),\\
           &\sup_{x_0}\sup_{t\leq T}t^{\frac{2+\kappa}{2}}\||D|^{1+\kappa}u_{R1,x_0}^{21,P}(t)\|_{L^\infty}\lesssim\|\rho_0\|_{C^2}^2\mathbf{C}^2(u)\left(T^{\frac{5}{8}}+T^{\frac{3}{8}}\|u\|_{X_T}^{\frac{1}{2}}+T^{\frac{1}{4}}\left(\|\rho_0-1\|_{L^\infty}+\|u\|_{X_T}\right) \right).
       \end{aligned}
   \end{equation}
   Finally, \eqref{S4eq55}, \eqref{S4eq53}, \eqref{S4eq18}, \eqref{S4eq19} and \eqref{S4eq56} imply that
   \begin{equation}\label{S4eq57}
       \begin{aligned}
           &\sup_{x_0}\sup_{t\leq T}t^{\frac{1}{2}}\|u_{R1,x_0}^{2}(t)\|_{L^\infty}\lesssim \|\rho_0\|_{C^2}^2\mathbf{C}^2(u)\left(T^{\frac{5}{8}}+T^{\frac{3}{8}}\|u\|_{X_T}^{\frac{1}{2}}+T^{\frac{1}{4}}\left(\|\rho_0-1\|_{L^\infty}+\|u\|_{X_T}\right) \right),\\
           &\sup_{t\leq T}t^{\frac{5}{12}}\|\tilde{ u}_{R1}^{2}(t)\|_{L^{12}}\lesssim \|\rho_0\|_{C^2}^2\mathbf{C}^2(u)\left(T^{\frac{2}{3}}+T^{\frac{5}{12}}\|u\|_{X_T}^{\frac{1}{2}}+T^{\frac{7}{24}}\left(\|\rho_0-1\|_{L^\infty}+\|u\|_{X_T}\right) \right),\\
           &\sup_{x_0}\sup_{t\leq T}t^{\frac{2+\kappa}{2}}\||D|^{1+\kappa}u_{R1,x_0}^{2}(t)\|_{L^\infty}\lesssim\|\rho_0\|_{C^2}^2\mathbf{C}^2(u)\left(T^{\frac{5}{8}}+T^{\frac{3}{8}}\|u\|_{X_T}^{\frac{1}{2}}+T^{\frac{1}{4}}\left(\|\rho_0-1\|_{L^\infty}+\|u\|_{X_T}\right) \right).
       \end{aligned}
   \end{equation}
\iffalse
then by \eqref{defKx0}, \eqref{S4eq11}, \eqref{intpoeq} and Lemma \eqref{eng1}, \ref{eng2}, \ref{eng3}, we have the $L^\infty$ estimate
\begin{equation}\label{S4eq17}
    \begin{aligned}
        \|u_{R1,x_0}^2(t)\|_{L^\infty}&\lesssim \int_0^t\bigl\|K_{x_0}(t)\ast \left(\nabla\rho_0\cdot\left(\nabla u(\tau)-\mathbb{P}(u\otimes u)(\tau)\right)\right)\bigr\|_{L^2}^{\frac{2\kappa}{2\kappa+3}} \\
        &\qquad\times\bigl\|K_{x_0}(t)\ast \left(\nabla\rho_0\cdot\left(\nabla u(\tau)-\mathbb{P}(u\otimes u)(\tau)\right)\right)\bigr\|_{\dot C^{\kappa}}^{\frac{3}{2\kappa+3}} d\tau\\
        &\quad+\int_0^t\sup_{x_0}|\nabla\Delta^{-1}K_{x_0}(t)|\|_{L^2}\|\rho_0\|_{C^2}\|\nabla u(\tau)-\mathbb{P}(u\otimes u)(\tau)\|_{L^2}d\tau \\
        &\lesssim \|\rho_0\|_{C^1}\int_0^t\tau^{-\frac{1}{2}\frac{2\kappa}{2\kappa+3}}t^{-\frac{\kappa}{2}\frac{3}{2\kappa+3}}\tau^{-\frac{3}{2\kappa+3}}d\tau \mathbf{C}(u)^2\|u\|_{X_T}^{\frac{3}{2\kappa+3}}\\
        &\quad+\|\rho_0\|_{C^2}t^{-\frac{1}{4}}\int_0^t\tau^{-\frac{1}{2}}d\tau \mathbf{C}(u)^2\\
        &\lesssim \|\rho_0\|_{C^2}\mathbf{C}^2(u)\left(t^{-\frac{\kappa}{4\kappa+6}}\|u\|_{X_T}^{\frac{1}{2\kappa+3}}+t^{\f14}\right),
    \end{aligned}
\end{equation}
the $L^{12}$ estimate

\fi
So \eqref{uS4eq20} and \eqref{S4eq57} show that
\begin{equation}\label{uR1Linf}
\begin{aligned}
&\sup_{x_0} \sup_{t\leq T}t^{\frac{1}{2}}\|u_{R1,x_0}(t)\|_{L^\infty}\\ &\lesssim \mathbf{C}^2(u)\|\rho_0\|_{C^2}^2\mathbf{C}^2(u)\left(T^{\frac{5}{8}}+T^{\frac{3}{8}}\|u\|_{X_T}^{\frac{1}{2}}+T^{\frac{1}{4}}\left(\|\rho_0-1\|_{L^\infty}+\|u\|_{X_T}\right) \right),
\end{aligned}
\end{equation}
\eqref{S4eq21} and \eqref{S4eq57} show that
\begin{equation}\label{uR1L12}
\begin{aligned}
    &\sup_{t\leq T}t^{\frac{5}{12}}\|\tilde{u}_{R1}(t)\|_{L^{12}}\\
    &\lesssim  \mathbf{C}^2(u) \|\rho_0\|_{C^2}^2\mathbf{C}^2(u)\left(T^{\frac{2}{3}}+T^{\frac{5}{12}}\|u\|_{X_T}^{\frac{1}{2}}+T^{\frac{7}{24}}\left(\|\rho_0-1\|_{L^\infty}+\|u\|_{X_T}\right) \right),
\end{aligned}
\end{equation}
and \eqref{S4eq22} and \eqref{S4eq57} show that
\begin{equation}\label{uR1Hol}
\begin{aligned}
&\sup_{x_0} \sup_{t\leq T}t^{1+\frac{\kappa}{2}}\| |D|^{1+\kappa} u_{R1,x_0}(t) \|_{L^\infty} \\
&\lesssim  \mathbf{C}^2(u) \|\rho_0\|_{C^2}^2\mathbf{C}^2(u)\left(T^{\frac{5}{8}}+T^{\frac{3}{8}}\|u\|_{X_T}^{\frac{1}{2}}+T^{\frac{1}{4}}\left(\|\rho_0-1\|_{L^\infty}+\|u\|_{X_T}\right) \right)%+t^{-\frac{3 + 2\kappa}{4}}\|\rho_0\|_{C^1}\|u_0\|_{L^2}
.
\end{aligned}
\end{equation}

\noindent\textbf{ii) Estimates for time differences.}\smallskip

By \eqref{S4eq11}, for $s<t<T$,
\begin{align*}
    u_{R1,x_0}^2(t,x)-u_{R1,x_0}^2(s,x)=\sum_{i=1}^3\left(u_{R1,x_0}^{2i}(t,x)-u_{R1,x_0}^{2i}(s,x)\right).
\end{align*}
For the latter two terms, since by Proposition \ref{Holtrho} we have
\begin{equation*}
    \|\rho(t)-\rho(s)\|_{L^\infty}\lesssim_\eta \|\rho_0\|_{C^1}\mathbf{C}^2(u)(t-s)^{1-\eta}s^{-\frac{1-\eta}{2}},\qquad\forall \eta\in(0,1).
\end{equation*}
By the estimate above and \eqref{S4eqp}, for $a=t-s\leq \frac{s}{4}$, we have
\begin{align*}
    &\|u_{R1,x_0}^{22}(t)-u_{R1,x_0}^{22}(s)\|_{L^\infty}\lesssim \int_s^t\|\nabla\Delta^{-1}\partial_\tau K_{x_0}(\tau)\|_{L^2}d\tau\|\nabla\rho_0\|_{L^\infty}\|\rho_0-1\|_{L^\infty}\|u_0\|_{L^2}\label{S4eq58}\\
    &\qquad\lesssim s^{-\frac{1}{4}}\min\{1,as^{-1}\}\|\nabla\rho_0\|_{L^\infty}\|\rho_0-1\|_{L^\infty}\|u_0\|_{L^2},\nonumber\\
    &\|\tilde u_{R1}^{22}(t)-\tilde u_{R1}^{22}(s)\|_{L^{12}}\lesssim \int_s^t\|\nabla\Delta^{-1}\partial_\tau K_{x_0}(\tau)\|_{L^\frac{12}{7}}d\tau\|\nabla\rho_0\|_{L^\infty}\|\rho_0-1\|_{L^\infty}\|u_0\|_{L^2}\nonumber\\
    &\qquad\lesssim s^{-\frac{1}{8}}\min\{1,as^{-1}\}\|\nabla\rho_0\|_{L^\infty}\|\rho_0-1\|_{L^\infty}\|u_0\|_{L^2},\nonumber\\
    &\|u_{R1,x_0}^{23}(t)-u_{R1,x_0}^{23}(s)\|_{L^\infty}\lesssim \int_s^t\|\nabla\Delta^{-1}\partial_\tau K_{x_0}(\tau)\|_{L^2}d\tau\|\nabla\rho_0\|_{L^\infty}\|\rho_0-1\|_{L^\infty}\|u_0\|_{L^2}\nonumber\\
    &\qquad\qquad+\left\|\sup_{x_0}|\nabla\Delta^{-1}K_{x_0}(s)|\right\|_{L^2 }\|\nabla\rho_0\|_{L^\infty}\left(\|\rho(t)-\rho(s)\|_{L^\infty}\|u_0\|_{L^2}+\|\rho_0-1\|_{L^\infty}\|u(t)-u(s)\|_{L^2}\right) \nonumber \\
    &\qquad\lesssim a^{\varkappa_1}\|\nabla\rho_0\|_{L^\infty}\|u_0\|_{L^2}\left(s^{-\frac{1}{4}-\varkappa_1}\|\rho_0-1\|_{L^\infty}+s^{\frac{1-\eta}{2}-\varkappa_1-\frac{1}{4}}\right),\nonumber\\
    &\|\tilde u_{R1}^{23}(t)-\tilde u_{R1}^{23}(s)\|_{L^{12}}\lesssim \int_s^t\|\nabla\Delta^{-1}\partial_\tau K_{x_0}(\tau)\|_{L^{\frac{12}{7}}}d\tau\|\nabla\rho_0\|_{L^\infty}\|\rho_0-1\|_{L^\infty}\|u_0\|_{L^2}\nonumber\\
    &\qquad\qquad+\left\|\sup_{x_0}|\nabla\Delta^{-1}K_{x_0}(s)|\right\|_{L^{\frac{12}{7}} }\|\nabla\rho_0\|_{L^\infty}\left(\|\rho(t)-\rho(s)\|_{L^\infty}\|u_0\|_{L^2}+\|\rho_0-1\|_{L^\infty}\|u(t)-u(s)\|_{L^2}\right) \nonumber \\
    &\qquad\lesssim a^{\varkappa_2}\|\rho_0\|_{C^1}^2\|u_0\|_{L^2}\left(s^{-\frac{1}{8}-\varkappa_2}\|\rho_0-1\|_{L^\infty}+s^{\frac{1-\eta}{2}-\varkappa_1-\frac{1}{8}}\right),\nonumber
\end{align*}
for some $\eta\in(0,1)$. So we can take $\eta=\frac{1}{4}$ to get the following difference estimate:
\begin{equation}\label{S4eq59}
    \begin{aligned}
        &\sup_{x_0}\sup_{s<t<T}s^{\frac{1}{2}+\varkappa_1}\frac{\left\|u_{R1,x_0}^{22}(t)-u_{R1,x_0}^{22}(s)\right\|_{L^\infty}}{(t-s)^{\varkappa_1}}+\sup_{x_0}\sup_{s<t<T}s^{\frac{1}{2}+\varkappa_1}\frac{\left\|u_{R1,x_0}^{23}(t)-u_{R1,x_0}^{23}(s)\right\|_{L^\infty}}{(t-s)^{\varkappa_1}}\\
        &+\sup_{s<t<T}s^{\frac{5}{12}+\varkappa_2}\frac{\left\|\tilde u_{R1}^{22}(t)-\tilde u_{R1,}^{22}(s)\right\|_{L^{12}}}{(t-s)^{\varkappa_2}}+\sup_{s<t<T}s^{\frac{5}{12}+\varkappa_2}\frac{\left\|\tilde u_{R1}^{23}(t)-\tilde u_{R1}^{23}(s)\right\|_{L^{12}}}{(t-s)^{\varkappa_2}}\\
        &\qquad\lesssim \|\rho_0\|_{C^1}^2\|u_0\|_{L^2}\left(T^{\frac{5}{8}}+T^{\frac{1}{4}}\|\rho_0-1\|_{L^\infty}\right)
    \end{aligned}
\end{equation}
For $u_{R1,x_0}^{21}$, we still apply the decomposition \eqref{S4eq60}. For $u_{R1,x_0}^{21,m}$, we have
\begin{align*}
    u_{R1,x_0}^{21,m}(t,x)-u_{R1,x_0}^{21,m}(s,x)=&\int_s^t\nabla\Delta^{-1}\nabla\partial_\gamma K_{x_0}(\gamma)d\gamma\ast \int_0^t\nabla\rho_0\left(\nabla u(\tau)-(\rho u\otimes u)(\tau)\right)d\tau\\
    &-\int_s^t\nabla\Delta^{-1}\partial_\gamma K_{x_0}(\gamma)d\gamma\ast \int_0^t\nabla^2\rho_0\left(\nabla u(\tau)-(\rho u\otimes u)(\tau)\right)d\tau\\
    &+\nabla\Delta^{-1}\nabla K_{x_0}(s)\ast \int_s^t\nabla\rho_0\left(\nabla u(\tau)-(\rho u\otimes u)(\tau)\right)d\tau\\
    &-\nabla\Delta^{-1} K_{x_0}(s)\ast \int_s^t\nabla^2\rho_0\left(\nabla u(\tau)-(\rho u\otimes u)(\tau)\right)d\tau
\end{align*}
So by \eqref{esthtk}, Lemma \ref{eng1}, \ref{eng2} and \ref{eng3} we separately have the $L^\infty$ estimate
\begin{align*}
    \|u_{R1,x_0}^{21,m}(t)-u_{R1,x_0}^{21,m}(s)\|_{L^\infty}\lesssim &\int_s^t\|\nabla^2K_{x_0}(\gamma)\|_{L^{\frac{4}{3}}}d\gamma\int_0^t\|\rho_0\|_{C^1}^2\left(\|\nabla u(\tau)\|_{L^4}+\|u(\tau)\|_{L^\infty}\|u(\tau)\|_{L^4}\right)d\tau\\
    &+\int_s^t\|\nabla K_{x_0}(\gamma)\|_{L^{\frac{4}{3}}}d\gamma\int_0^t\|\rho_0\|_{C^2}^2\left(\|\nabla u(\tau)\|_{L^4}+\|u(\tau)\|_{L^\infty}\|u(\tau)\|_{L^4}\right)d\tau\\
    &+\|\rho_0\|_{C^1}^2\|\nabla\Delta^{-1}\nabla K_{x_0}(s)\|_{L^{\frac{4}{3}}}\int_s^t\left(\|\nabla u(\tau)\|_{L^4}+\|u(\tau)\|_{L^\infty}\|u(\tau)\|_{L^4}\right)d\tau\\
    &+\|\rho_0\|_{C^2}^2\|\nabla\Delta^{-1}K_{x_0}(s)\|_{L^{2}}\int_s^t\|\nabla u(\tau)-\mathbb{P}(u\otimes u)(\tau)\|_{L^{2}}d\tau\\
    \lesssim &\|\rho_0\|_{C^2}^2\mathbf{C}^2(u)\left(\int_s^t\gamma^{-1-\frac{3}{8}}d\gamma\int_0^t\tau^{-\frac{3}{4}}d\tau +\int_s^t\gamma^{-\frac{1}{2}-\frac{3}{8}}d\gamma \int_0^t\tau^{-\frac{3}{4}}d\tau\right)\|u\|_{X_T}^{\frac{1}{2}}\\
    &+\|\rho_0\|_{C^2}^2\mathbf{C}^2(u)\left(s^{-\frac{3}{8}}\int_0^t\tau^{-\frac{3}{4}}d\tau\|u\|_{X_T}^{\f12}+s^{-{\f14}}\int_s^t\tau^{-\frac{1}{2}}d\tau\|u_0\|_{L^2} \right)\\
    \lesssim&\|\rho_0\|_{C^2}^2(t-s)^{\varkappa_1} \mathbf{C}^2(u)\left(s^{-\frac{1}{8}-\varkappa_1}\|u\|_{X_T}^{\f12}+s^{\frac{1}{4}-\varkappa_1}\right),
\end{align*}
and the $L^{12}$ estimate
\begin{align*}
    \|\tilde u_{R1}^{21,m}(t)-\tilde u_{R1}^{21,m}(s)\|_{L^{12}}\lesssim &\int_s^t\|\sup_{x_0}|\nabla^2K_{x_0}(\gamma)|\|_{L^{\frac{6}{5}}}d\gamma\int_0^t\|\rho_0\|_{C^1}^2\left(\|\nabla u(\tau)\|_{L^4}+\|u(\tau)\|_{L^\infty}\|u(\tau)\|_{L^4}\right)d\tau\\
    &+\int_s^t\|\sup_{x_0}|\nabla K_{x_0}(\gamma)|\|_{L^{\frac{6}{5}}}d\gamma\int_0^t\|\rho_0\|_{C^2}^2\left(\|\nabla u(\tau)\|_{L^4}+\|u(\tau)\|_{L^\infty}\|u(\tau)\|_{L^4}\right)d\tau\\
    &+\|\rho_0\|_{C^1}^2\|\sup_{x_0}|\nabla\Delta^{-1}\nabla K_{x_0}(s)|\|_{L^{\frac{6}{5}}}\int_s^t\left(\|\nabla u(\tau)\|_{L^4}+\|u(\tau)\|_{L^\infty}\|u(\tau)\|_{L^4}\right)d\tau\\
    &+\|\rho_0\|_{C^2}^2\|\sup_{x_0}|\nabla\Delta^{-1}K_{x_0}(s)|\|_{L^\frac{12}{7}}\int_s^t\|\nabla u(\tau)-(u\otimes u)(\tau)\|_{L^{2}}d\tau\\
    \lesssim &\|\rho_0\|_{C^2}^2\mathbf{C}^2(u)\left(\int_s^t\gamma^{-1-\frac{1}{4}}d\gamma\int_0^t\tau^{-\frac{3}{4}}d\tau +\int_s^t\gamma^{-\frac{1}{2}-\frac{1}{4}}d\gamma \int_0^t\tau^{-\frac{3}{4}}d\tau\right)\|u\|_{X_T}^{\frac{1}{2}}\\
    &+\|\rho_0\|_{C^2}^2\mathbf{C}^2(u)\left(s^{-\frac{1}{4}}\int_0^t\tau^{-\frac{3}{4}}d\tau\|u\|_{X_T}^{\f12}+s^{-{\frac{1}{8}}}\int_s^t\tau^{-\frac{1}{2}}d\tau \right)\\
    \lesssim&\|\rho_0\|_{C^2}^2(t-s)^{\varkappa_2} \mathbf{C}^2(u)\left(s^{-\varkappa_2}\|u\|_{X_T}^{\f12}+s^{\frac{3}{8}-\varkappa_2}\right),
\end{align*}
So we conclude that 
\begin{equation}\label{S4eq43}
    \begin{aligned}
        &\sup_{x_0}\sup_{s<t<T}s^{\frac{1}{2}+\varkappa_1}\frac{\|u_{R1,x_0}^{21,m}(t)-u_{R1,x_0}^{21,m}(s)\|_{L^\infty}}{(t-s)^{\varkappa_1}}\lesssim \|\rho_0\|_{C^2}^2\mathbf{C}^2(u)\left(T^{\frac{3}{8}}\|u\|_{X_T}^{\frac{1}{2}}+T^{\frac{3}{4}} \right)\\
        &\sup_{s<t<T}s^{\frac{5}{12}+\varkappa_2}\frac{\|\tilde u_{R1}^{21,m}(t)-\tilde u_{R1}^{21,m}(s)\|_{L^{12}}}{(t-s)^{\varkappa_2}}\lesssim \|\rho_0\|_{C^2}^2\mathbf{C}^2(u)\left(T^{\frac{5}{12}}\|u\|_{X_T}^{\frac{1}{2}}+T^{\frac{19}{24}} \right)
    \end{aligned}
\end{equation}
For $u_{R1,x_0}^{21,P}$, we apply the decomposition of $\nabla P$ as in \eqref{S4eq61} and \eqref{S4eq63}. Note that by \eqref{S4eqp}, Proposition \ref{Holtrho} and Lemma \ref{eng3}, for any $\eta\in(0,1)$ and $s<t$, we have
\begin{align*}
&\left\|P_1(t)\right\|_{L^2}\lesssim \int_0^t \tau^{-1+\frac{1-\eta}{2}}d\tau\|\rho_0\|_{C^1}\mathbf{C}^2(u)\lesssim t^{\frac{1-\eta}{2}}\|\rho_0\|_{C^1}\mathbf{C}^2(u),\\
  &  \left\|P_1(t)-P_1(s) \right\|_{L^2}\lesssim \int_s^t\tau^{-1+\frac{1-\eta}{2}}d\tau\|\rho_0\|_{C^1}\mathbf{C}^2(u)\lesssim s^{-\frac{1+\eta}{2}}\min\{1,(t-s)s^{-1}\}\|\rho_0\|_{C^1}\mathbf{C}^2(u),\\
  &\left\|P_2^1(t)\right\|_{L^2}\lesssim \int_0^t \tau^{-1+\frac{1-\eta}{2}}d\tau\|\rho_0\|_{C^1}\mathbf{C}^3(u)\lesssim t^{\frac{1-\eta}{2}}\|\rho_0\|_{C^1}\mathbf{C}^3(u),\\
  &  \left\|P_2^1(t)-P_2^1(s) \right\|_{L^2}\lesssim \int_s^t\tau^{-1+\frac{1-\eta}{2}}d\tau\|\rho_0\|_{C^1}\mathbf{C}^3(u)\lesssim s^{-\frac{1+\eta}{2}}\min\{1,(t-s)s^{-1}\}\|\rho_0\|_{C^1}\mathbf{C}^3(u),\\
  &\left\|P_3(t)\right\|_{L^2}\lesssim \|\rho_0-1\|_{L^\infty}\|u_0\|_{L^2},\\
  &  \left\|P_3(t)-P_3(s) \right\|_{L^2}\lesssim \|\rho_0-1\|_{L^\infty}\|u(t)-u(s)\|_{L^2}\lesssim \|\rho_0-1\|_{L^\infty}\|u_0\|_{L^2}\min\{1,(t-s)s^{-1}\}.
\end{align*}
Then we denote the components in $u_{R1,x_0}^{21,P}$ related to $P_1,P_2^1, P_2^2,P_3$ by $u_{R1,x_0}^{21,P1},u_{R1,x_0}^{21,P21}, u_{R1,x_0}^{P22},u_{R1,x_0}^{21,P3}$ separately, and by \eqref{esthtk} we can show that
\begin{align}
    &\left\|u_{R1,x_0}^{21,P1}(t)-u_{R1,x_0}^{21,P1}(s)\right\|_{L^\infty}\lesssim \|\rho_0\|_{C^1}\int_s^t\left\|\sup_{x_0}\partial_\tau\nabla\Delta^{-1}K_{x_0}(\tau)\right\|_{L^2}d\tau \|P_1(t)\|_{L^2} \label{S4eq62}\\
    &\qquad\qquad+\|\rho_0\|_{C^1}\left\|\sup_{x_0}\nabla\Delta^{-1}K_{x_0}(s)\right\|_{L^2}\|P_1(t)-P_1(s)\|_{L^2} \nonumber\\
    &\qquad\lesssim (t-s)^{\varkappa_1}s^{\frac{1}{8}-\varkappa_1}\|\rho_0\|_{C^1}^2\mathbf{C}^2(u),\nonumber\\
    &\left\|u_{R1,x_0}^{21,P21}(t)-u_{R1,x_0}^{21,P21}(s)\right\|_{L^\infty}\lesssim \|\rho_0\|_{C^1}\int_s^t\left\|\sup_{x_0}\partial_\tau\nabla\Delta^{-1}K_{x_0}(\tau)\right\|_{L^2}d\tau \|P_2^1(t)\|_{L^2} \nonumber\\
    &\qquad\qquad+\|\rho_0\|_{C^1}\left\|\sup_{x_0}\nabla\Delta^{-1}K_{x_0}(s)\right\|_{L^2}\|P_2^1(t)-P_2^1(s)\|_{L^2} \nonumber\\
    &\qquad\lesssim (t-s)^{\varkappa_1}s^{\frac{1}{8}-\varkappa_1}\|\rho_0\|_{C^1}^2\mathbf{C}^3(u)\nonumber\\
    &\left\|u_{R1,x_0}^{21,P3}(t)-u_{R1,x_0}^{21,P3}(s)\right\|_{L^\infty}\lesssim \|\rho_0\|_{C^1}\int_s^t\left\|\sup_{x_0}\partial_\tau\nabla\Delta^{-1}K_{x_0}(\tau)\right\|_{L^2}d\tau \|P_3(t)\|_{L^2} \nonumber\\
    &\qquad\qquad+\|\rho_0\|_{C^1}\left\|\sup_{x_0}\nabla\Delta^{-1}K_{x_0}(s)\right\|_{L^2}\|P_3(t)-P_3(s)\|_{L^2} \nonumber\\
    &\qquad\lesssim (t-s)^{\varkappa_1}s^{-\frac{1}{4}-\varkappa_1}\|\rho_0\|_{C^1}^2\|\rho_0-1\|_{L^\infty}\mathbf{C}^2(u),\nonumber\\
    &\left\|\tilde u_{R1}^{21,P1}(t)-\tilde u_{R1}^{21,P1}(s)\right\|_{L^{12}}\lesssim \|\rho_0\|_{C^1}\int_s^t\left\|\sup_{x_0}\partial_\tau\nabla\Delta^{-1}K_{x_0}(\tau)\right\|_{L^{\frac{12}{7}}}d\tau \|P_1(t)\|_{L^2}\nonumber \\
    &\qquad\qquad+\|\rho_0\|_{C^1}\left\|\sup_{x_0}\nabla\Delta^{-1}K_{x_0}(s)\right\|_{L^{\frac{12}{7}}}\|P_1(t)-P_1(s)\|_{L^2}\nonumber \\
    &\qquad\lesssim (t-s)^{\varkappa_2}s^{\frac{1}{4}-\varkappa_2}\|\rho_0\|_{C^1}^2\mathbf{C}^2(u),\nonumber\\
    &\left\|\tilde u_{R1}^{21,P21}(t)-\tilde u_{R1}^{21,P21}(s)\right\|_{L^{12}}\lesssim \|\rho_0\|_{C^1}\int_s^t\left\|\sup_{x_0}\partial_\tau\nabla\Delta^{-1}K_{x_0}(\tau)\right\|_{L^{\frac{12}{7}}}d\tau \|P_2^1(t)\|_{L^2}\nonumber \\
    &\qquad\qquad+\|\rho_0\|_{C^1}\left\|\sup_{x_0}\nabla\Delta^{-1}K_{x_0}(s)\right\|_{L^{\frac{12}{7}}}\|P_2^1(t)-P_2^1(s)\|_{L^2}\nonumber \\
    &\qquad\lesssim (t-s)^{\varkappa_2}s^{\frac{1}{4}-\varkappa_2}\|\rho_0\|_{C^1}^2\mathbf{C}^3(u),\nonumber\\
    &\left\|\tilde u_{R1}^{21,P3}(t)-\tilde u_{R1}^{21,P3}(s)\right\|_{L^{12}}\lesssim \|\rho_0\|_{C^1}\int_s^t\left\|\sup_{x_0}\partial_\tau\nabla\Delta^{-1}K_{x_0}(\tau)\right\|_{L^{\frac{12}{7}}}d\tau \|P_3(t)\|_{L^2}\nonumber \\
    &\qquad\qquad+\|\rho_0\|_{C^1}\left\|\sup_{x_0}\nabla\Delta^{-1}K_{x_0}(s)\right\|_{L^{\frac{12}{7}}}\|P_3(t)-P_3(s)\|_{L^2}\nonumber \\
    &\qquad\lesssim (t-s)^{\varkappa_2}s^{-\frac{1}{8}-\varkappa_2}\|\rho_0\|_{C^1}^2\|\rho_0-1\|_{L^\infty}\mathbf{C}^2(u).\nonumber
\end{align}
Finally, by \eqref{S4eq64}, and the fact that
\begin{align*}
    &\|\int_0^t(u\otimes u)(\tau)d\tau\|_{L^2}\lesssim t^{\frac{1}{2}}\|u\|_{X_T}\mathbf{C}(u),\\
    &\|\int_s^t(u\otimes u)(\tau)d\tau\|_{L^2}\lesssim s^{\frac{1}{2}}\min\{1,(t-s)s^{-1}\}\|u\|_{X_T}\mathbf{C}(u),
\end{align*}
together with \eqref{esthtk}, we have
\begin{align}
    &\left\|u_{R1,x_0}^{21,P22}(t)-u_{R1,x_0}^{21,P22}(s) \right\|_{L^\infty}\lesssim \int_s^t\left\|\sup_{x_0}|\partial_\tau\nabla^2\Delta^{-1}K_{x_0}(\tau)|\right\|_{L^2}d\tau\|\rho_0\|_{C^1}^2t^{\frac{1}{2}}\|u\|_{X_T}\mathbf{C}(u)\label{S4eq65}\\
    &\qquad\qquad+\left\|\sup_{x_0}|\nabla^2\Delta^{-1}K_{x_0}(s)| \right\|_{L^2}s^{\frac{1}{2}}\|\rho_0\|_{C^1}\min\{1,(t-s)s^{-1}\}\|u\|_{X_T}\mathbf{C}(u)\nonumber\\
    &\qquad\qquad+\int_s^t\left\|\sup_{x_0}|\partial_\tau\nabla\Delta^{-1}K_{x_0}(\tau)|\right\|_{L^2}d\tau\|\rho_0\|_{C^2}^2t^{\frac{1}{2}}\|u\|_{X_T}\mathbf{C}(u)\nonumber\\
    &\qquad\qquad+\left\|\sup_{x_0}|\nabla\Delta^{-1}K_{x_0}(s)| \right\|_{L^2}s^{\frac{1}{2}}\|\rho_0\|_{C^2}^2\min\{1,(t-s)s^{-1}\}\|u\|_{X_T}\mathbf{C}(u)\nonumber\\
    &\qquad\lesssim (t-s)^{\varkappa_1}s^{-\frac{1}{4}-\varkappa_1}\|\rho_0\|_{C^2}^2\mathbf{C}^2(u)\|u\|_{X_T},\nonumber\\
    &\left\|\tilde u_{R1}^{21,P22}(t)-\tilde u_{R1}^{21,P22}(s) \right\|_{L^{12}}\lesssim \int_s^t\left\|\sup_{x_0}|\partial_\tau\nabla^2\Delta^{-1}K_{x_0}(\tau)|\right\|_{L^{\frac{12}{7}}}d\tau\|\rho_0\|_{C^1}^2t^{\frac{1}{2}}\|u\|_{X_T}\mathbf{C}(u)\nonumber\\
    &\qquad\qquad+\left\|\sup_{x_0}|\nabla^2\Delta^{-1}K_{x_0}(s)| \right\|_{L^{\frac{12}{7}}}s^{\frac{1}{2}}\|\rho_0\|_{C^1}\min\{1,(t-s)s^{-1}\}\|u\|_{X_T}\mathbf{C}(u)\nonumber\\
    &\qquad\qquad+\int_s^t\left\|\sup_{x_0}|\partial_\tau\nabla\Delta^{-1}K_{x_0}(\tau)|\right\|_{L^{\frac{12}{7}}}d\tau\|\rho_0\|_{C^2}^2t^{\frac{1}{2}}\|u\|_{X_T}\mathbf{C}(u)\nonumber\\
    &\qquad\qquad+\left\|\sup_{x_0}|\nabla\Delta^{-1}K_{x_0}(s)| \right\|_{L^{\frac{12}{7}}}s^{\frac{1}{2}}\|\rho_0\|_{C^2}^2\min\{1,(t-s)s^{-1}\}\|u\|_{X_T}\mathbf{C}(u)\nonumber\\
    &\qquad\lesssim (t-s)^{\varkappa_1}s^{-\frac{1}{8}-\varkappa_1}\|\rho_0\|_{C^2}^2\mathbf{C}^2(u)\|u\|_{X_T}.\nonumber
\end{align}

So by \eqref{S4eq59}, \eqref{S4eq43}, \eqref{S4eq62} and \eqref{S4eq65}, we have
\begin{equation}\label{S4eq27}
    \begin{aligned}
        &\sup_{x_0} \sup_{s < t \leq T} s^{\frac{1}{2} + \varkappa_1} \frac{\|u_{R1,x_0}^2(t) - u_{R1,x_0}^2(s)\|_{L^\infty}}{(t-s)^{\varkappa_1}}+\sup_{x_0} \sup_{s < t \leq T} s^{\frac{5}{12} + \varkappa_2} \frac{\|u_{R1,x_0}^2(t) - u_{R1,x_0}^2(s)\|_{L^{12}}}{(t-s)^{\varkappa_2}} \\
&\qquad\lesssim \mathbf{C}^3(u)\|\rho_0\|_{C^2}^2\left(T^{\frac{1}{4}}(\|u\|_{X_T}+\|\rho_0-1\|_{L^\infty})+T^{\frac{3}{8}}\|u\|_{X_T}^{\f12}+T^{\frac{5}{8}}\right) .
    \end{aligned}
\end{equation}

By summarizing the estimates \eqref{uR1Linf}, \eqref{uR1L12}, \eqref{uR1Hol}, \eqref{uR1bmo} and \eqref{S4eq27}, we conclude the proof of \eqref{S4eq12}.
\end{proof}

	\begin{prop}\label{propuR2}
	Under the assumptions of Proposition \ref{uR1}, one has
\begin{equation}\label{S4eq4}
\begin{aligned}
&\sup_{x_0} \Bigl( \sup_{t \leq T} t^{\frac{1}{2}} \|u_{R2,x_0}(t)\|_{L^\infty} + \|u_{R2,x_0}\|_{V_T^b} + \sup_{s < t < T} s^{\frac{1}{2} + \varkappa_1} \frac{\|u_{R2,x_0}(t) - u_{R2,x_0}(s)\|_{L^\infty}}{(t-s)^{\varkappa_1}} \Bigr) \\
&\quad + \sup_{t \leq T} t^{1 + \frac{\kappa}{2}} \| (\nabla D^\kappa u_{R2,x_0})|_{x_0=x}(t) \|_{L^\infty}
+ \sup_{t \leq T} t^{\frac{5}{12}} \|\tilde{u}_{R2}(t)\|_{L^{12}} \\
&\quad + \sup_{s < t < T} s^{\frac{5}{12} + \varkappa_2} \frac{\|\tilde{u}_{R2}(t) - \tilde{u}_{R2}(s)\|_{L^{12}}}{(t-s)^{\varkappa_2}} \leq CT^{\f38}\|\rho_0\|_{C^1} \mathbf{C}(u)\|u\|_{T}.
\end{aligned}
\end{equation}
\end{prop}

\begin{proof} We divide the proof into the following two steps:

\noindent \textbf{(i) Estimates for spatial integrability.}\smallskip

First in view of \eqref{S4eq0}, we can rewrite $u_{R2,x_0}$ as
\begin{equation*}\label{S4eq5}
u_{R2,x_0} = \int_0^t \partial_\tau \mathbb{P} K_{x_0}(t-\tau) \ast \bigl( (\varrho u)(t) - (\varrho u)(\tau) \bigr) \, d\tau + K_{x_0}(t) \ast \mathbb{P} (\varrho u)(t) = u_{R2,x_0}^1 + u_{R2,x_0}^2.
\end{equation*}

It follows from the proof of Proposition \ref{Holtrho} that for any $\kappa\in (0,1),$
\begin{align*}
\|\rho(t)-\rho_0\|_{L^\infty}\lesssim& \||D|^\kappa\rho\|_{L^\infty_T(L^\infty)}\bigl(\e^{\kappa}+\|u\|_{X_T}\e^{\kappa-1}t^{\f12}\bigr)\\
\lesssim& \|\rho_0\|_{C^1}\bigl(1+\|u\|_{T}\bigr)t^{\f\kappa2},\quad\mbox{if we take }\ \ \e=t^{\f12},
\end{align*}
from which,  Propositions \ref{regdiff} and  \ref{Holtrho},  we deduce that for any $0 < \epsilon' < 1$ and for all $s < t < T$, 
\begin{equation}\label{rhouse}
\begin{split}
&\|\varrho(t)\|_{L^\infty} \lesssim t^{\frac{1-\epsilon'}{2}} \|\rho_0\|_{C^1}\bigl(1+\|u\|_{T}\bigr), \\
&\|\varrho(t) - \varrho(s)\|_{L^\infty} \lesssim \min\bigl\{ (t-s)^{\frac{1-\epsilon'}{2}}, s^{-\frac{1-\epsilon'}{2}} (t-s)^{1-\epsilon'} \bigr\} \|\rho_0\|_{C^1}\bigl(1+\|u\|_{T}\bigr).
\end{split}
\end{equation}

By \eqref{esthtk}, \eqref{refY}, \eqref{estL2u}, and \eqref{rhouse}, we infer
\begin{align*}
\|u^1_{R2,x_0}(t)\|_{L^\infty} 
 \lesssim & \int_0^t \|\partial_t \mathbb{P} K_{x_0}(t-\tau)\|_{L^1} \left( \|\varrho(\tau)(u(t)-u(\tau))\|_{L^\infty} + \|(\varrho(t)-\varrho(\tau)) u(t)\|_{L^\infty} \right) d\tau \\
\leq &C \int_0^t (t-\tau)^{-1} \bigl( \tau^{\frac{1-\epsilon'}{2}} (t-\tau)^{\varkappa_1} \tau^{-\frac{1}{2}-\varkappa_1} + (t-\tau)^{\frac{1-\epsilon'}{2}} t^{-1/2} \bigr) d\tau \mathbf{C}(u) \|\rho_0\|_{C^1} \|u\|_{T}\\
%\leq& C \mathbf{C}(u) \|\rho_0\|_{C^1}\bigl(1+\|u\|_{T}\bigr) t^{-\frac{\epsilon'}{2}} \\
\leq & C\mathbf{C}(u) \|\rho_0\|_{C^1}\|u\|_{T}t^{-\frac{1}{12}}, \quad (\mbox{if we take}\ \epsilon' = 1/24),
\end{align*}
and
\begin{align*}
&\|\na|D|^{\kappa} u^1_{R2,x_0}(t)\|_{L^\infty} \\
&\lesssim \int_0^t \|\partial_t |D|^{1+\kappa} \mathbb{P} K_{x_0}(t-\tau)\|_{L^1}  \left( \|\varrho(\tau)(u(t)-u(\tau))\|_{L^\infty} + \|(\varrho(t)-\varrho(\tau)) u(t)\|_{L^\infty} \right) d\tau \\
&\leq C\int_0^t (t-\tau)^{-1-\frac{1+\kappa}{2}} \bigl( \tau^{\frac{1-\epsilon'}{2}} (t-\tau)^{\varkappa_1} \tau^{-1/2-\varkappa_1} + (t-\tau)^{1-\epsilon'} \tau^{-\frac{1-\epsilon'}{2}} t^{-1/2} \bigr) d\tau \mathbf{C}(u) \|\rho_0\|_{C^1}\|u\|_{T}\\
 &\leq  C\mathbf{C}(u) \|\rho_0\|_{C^1}\|u\|_{T}t^{-\frac{\epsilon' + 1 + \kappa}{2}} \\
 &\leq  C\mathbf{C}(u) \|\rho_0\|_{C^1}\|u\|_{T} t^{-\frac{3}{4}}, \quad (\mbox{if we take}\ \epsilon' = 1/2 - \kappa),
 \end{align*}
and
\begin{align*}
\|\tilde{u}^1_{R2}(t)\|_{L^{12}} 
\lesssim &\int_0^t \bigl\| \sup_{x_0} \partial_t \mathbb{P} K_{x_0}(t-\tau) \bigr\|_{L^1} \left( \|\varrho(\tau)(u(t)-u(\tau))\|_{L^{12}} + \|(\varrho(t)-\varrho(\tau)) u(t)\|_{L^{12}} \right) d\tau \\
\leq & C \int_0^t (t-\tau)^{-1} \bigl( \tau^{\frac{1-\epsilon'}{2}} (t-\tau)^{\varkappa_2} \tau^{-\frac{5}{12}-\varkappa_2} + t^{-\frac{5}{12}} (t-\tau)^{1-\epsilon'} \tau^{-\frac{1-\epsilon'}{2}} \bigr) d\tau \mathbf{C}(u) \|\rho_0\|_{C^1}\|u\|_{T}\\
 \leq & C \mathbf{C}(u) \|\rho_0\|_{C^1}\|u\|_{T}t^{\frac{1}{12} - \frac{\epsilon'}{2}} \\
 \leq &\mathbf{C}(u) \|\rho_0\|_{C^1}\|u\|_{T}, \quad (\mbox{if we take}\ \epsilon' = 1/6).
\end{align*}

Similarly, we deduce from   \eqref{esthtk}, \eqref{refY} and  \eqref{estL2u}  that 
\begin{align*}
\|u_{R2,x_0}^2(t)\|_{L^\infty}\lesssim& \|\mathbb{P} K_{x_0}(t)\|_{L^\frac{12}{11}} \|(\varrho u)(t)\|_{L^{12}}
 \leq  Ct^{-\frac{1}{8}} t^{\frac{1-\epsilon'}{2}}t^{-\frac{5}{12}}  \mathbf{C}(u)\|u\|_{T} \|\rho_0\|_{C^1} \\
\leq &C \mathbf{C}(u) \|u\|_{T}\|\rho_0\|_{C^1} t^{-\frac{1}{6}}, \quad (\mbox{if we take}\ \epsilon' = 1/4),
\end{align*}
and
\begin{align*}
\|\na |D|^{\kappa} u_{R2,x_0}^2\|_{L^\infty} 
\lesssim&  \|\na |D|^{\kappa} K_{x_0}(t)\|_{L^1} \|(\varrho u)(t)\|_{L^\infty} \leq Ct^{-\frac{1+\epsilon}{2}} t^{\frac{1-\kappa}{2}} t^{-1/2} \mathbf{C}(u)\|u\|_{T} \|\rho_0\|_{C^1} \\
 \leq & C\mathbf{C}(u)\|u\|_{T}\|\rho_0\|_{C^1} t^{-\frac{7}{12}}, \quad (\mbox{if we take}\ \epsilon' = 1/6 - \epsilon),
\end{align*}
and
\begin{align*}
\|\tilde{u}^2_{R2}(t)\|_{L^{12}} \lesssim& \|\varrho u(t)\|_{L^{12}} \leq C t^{\frac{1-\epsilon'}{2}} t^{-\frac{5}{12}} \|u\|_{T}\mathbf{C}(u) \|\rho_0\|_{C^1} \\
 \leq& C\mathbf{C}(u)\|u\|_{T}\|\rho_0\|_{C^1}, \qquad (\mbox{if we take}\ \epsilon' = 1/6)
\end{align*}

Thanks to  \eqref{uR1bmo},  we get, by summarizing the above estimates, that
\begin{equation}\label{uR2spa}
\begin{aligned}
\sup_{x_0} \Bigl( & \sup_{t \leq T} \Bigl( t^{\frac{1}{2}} \|u_{R2,x_0}(t)\|_{L^\infty} + t^{1+\f\kappa2} \|\na |D|^{\kappa} u_{R2,x_0}(t)\|_{L^\infty} + t^{\frac{5}{12}} \|\tilde{u}_{R2}(t)\|_{L^{12}} \Bigr) \\
&+\|u_{R2,x_0}\|_{V_T^b} \Bigr) 
 \leq C T^{\frac{3}{8}} \mathbf{C}(u)\|u\|_{T} \|\rho_0\|_{C^1}.
\end{aligned}
\end{equation}

\noindent \textbf{(ii) Estimates for time differences.}\smallskip

We first consider the $L^{12}$ case. Note from \eqref{S4eq0}  that
\begin{equation}\label{S4eq7}
u_{R2,x_0} = \int_0^t \partial_t K_{x_0}(t-\tau) \ast \mathbb{P} (\varrho u)(\tau) \, d\tau - \mathbb{P} (\varrho u)(t).
\end{equation}
By \eqref{rhouse}  and the definition of the $\|\cdot\|_{X_T^2}$ norm, we have
\begin{equation}\label{uR2diffL12}
\begin{split}
\|\mathbb{P} (\varrho u)(t) - \mathbb{P} (\varrho u)(s)\|_{L^{12}}
\leq &\| \varrho(t) - \varrho(s)\|_{L^\infty}\|u(t)\|_{L^{12}}+\|\varrho(s)\|_{L^\infty}\|u(t)-u(s)\|_{L^{12}}\\
\lesssim & s^{-\left(\f{\varkappa_2}2+\f5{12}\right)}(t-s)^{\varkappa_2}\|\rho_0\|_{C^1}\|u\|_{X_T} \mathbf{C}(u)  \quad \forall s < t,
\end{split}
\end{equation}
from which, 
Lemma \ref{CVPDE} and \eqref{S4eq7}, we infer
\begin{align*}
&\sup_{s < t \leq T} s^{\frac{5}{12} + \varkappa_2} \frac{\|\tilde{u}_{R2}(t) - \tilde{u}_{R2}(s)\|_{L^{12}}}{(t-s)^{\varkappa_2}} \\
&\lesssim \sup_{t\leq T}t^{\f5{12}}\|\mathbb{P} (\varrho u)(t)\|_{L^{12}}+\sup_{s < t \leq T}s^{\frac{5}{12} + \varkappa_2}\f{\|\mathbb{P} (\varrho u)(t) - \mathbb{P} (\varrho u)(s)\|_{L^{12}}}{(t-s)^{\varkappa_2}}\\
&\lesssim T^{\f3{8}}\mathbf{C}(u)\|u\|_{X_T} \|\rho_0\|_{C^1}.
\end{align*}

For the $L^\infty$ case, we have for any $\sigma>0,$
\begin{equation*}
\|\mathbb{P} f\|_{L^\infty} \lesssim \|f\|_{L^{12}}^{\frac{4\sigma}{1+4\sigma}} \|f\|_{\dot{C}^{\sigma}}^{\frac{1}{1+4\sigma}}.
\end{equation*}
Taking $\sigma = \f{9}{14}$ in the above inequality, and using \eqref{rhouse}, \eqref{uR2diffL12}, and the fact that $\varkappa_1 = \frac{18}{25} \varkappa_2$, we obtain
\begin{equation}\label{uR2diffLinf}
\begin{split}
&\sup_{s<t\leq T}s^{\f12+\varkappa_1}\f{\|\mathbb{P} (\varrho u)(t) - \mathbb{P} (\varrho u)(s)\|_{L^\infty}}{(t-s)^{\varkappa_1}}\\
& \lesssim \Bigl(\sup_{s<t\leq T}s^{\f5{12}+\varkappa_2}\f{\|\varrho u(t) - \varrho u(s)\|_{L^{12}}}{(t-s)^{\varkappa_1}} \Bigr)^{\frac{18}{25}} \bigl(s^{\f57} \|\varrho u(s)\|_{\dot{C}^{\frac{9}{14}}} + t^{\f57}\|\varrho u(t)\|_{\dot{C}^{\frac{9}{14}}} \bigr)^{\frac{7}{25}}\\
& \lesssim  T^{\frac{3}{8}} \mathbf{C}(u)\|u\|_{X_T} \|\rho_0\|_{C^1},
\end{split}
\end{equation}
where we used the fact that $\|u(s)\|_{\dot{C}^{\frac{9}{14}}}\lesssim s^{-\f{25}{28}}\|u\|_{X_T}$ and
\begin{align*}
\|\varrho u(s)\|_{\dot{C}^{\frac{9}{14}}} \lesssim &\|\varrho(s)\|_{L^\infty}\| u(s)\|_{\dot{C}^{\frac{9}{14}}} 
+\|\varrho(s)\|_{\dot{C}^{\frac{9}{14}}}\|u(s)\|_{L^\infty}\\
\lesssim &\|u\|_{X_T}\|\rho_0\|_{C^1}s^{-\f12}.\end{align*}
Then it follows from Lemma \ref{CVPDE}, \eqref{S4eq7} and \eqref{uR2diffLinf} that
\begin{align*}
\sup_{x_0} \Bigl(\sup_{s < t < T} s^{\frac{1}{2} + \varkappa_1} \frac{\|u_{R2,x_0}(t) - u_{R2,x_0}(s)\|_{L^\infty}}{(t-s)^{\varkappa_1}} \Bigr) \lesssim T^{\f3{8}}\mathbf{C}(u)\|u\|_{X_T} \|\rho_0\|_{C^1}.
\end{align*}
This together with \eqref{uR2spa} completes the proof of \eqref{S4eq4}.
\end{proof}

\begin{prop}\label{uR3}
Under the assumptions of Proposition \ref{uR1}, if $\rho_0\in C^1$, then we have
\begin{equation*}\label{S4eq8}
\begin{aligned}
&\sup_{t \leq T} \Bigl( t^{\frac{1}{2}} \|\tilde{u}_{R3}(t)\|_{L^\infty} + \sup_{t \leq T} t^{1 + \frac{\kappa}{2}} \| (D^{1+\kappa} u_{R3,x_0})(t)|_{x_0=x} \|_{L^{\infty}} + t^{\frac{5}{12}} \|\tilde{u}_{R3}(t)\|_{L^{12}}  \\
&\quad + \sup_{s < t \leq T} s^{\frac{1}{2} + \varkappa_1} \frac{\| \tilde{u}_{R3}(t) - \tilde{u}_{R3}(s) \|_{L^\infty}}{(t-s)^{\varkappa_1}}
+ \sup_{s < t \leq T} s^{\frac{5}{12} + \varkappa_2} \frac{\| \tilde{u}_{R3}(t) - \tilde{u}_{R3}(s) \|_{L^{12}}}{(t-s)^{\varkappa_2}} \Bigr) \\
&\quad + \sup_{z \in \mathbb{R}^3} \sup_{R < b} R^{-\frac{3}{2}} \Bigl( \int_0^{R^2} \int_{B(z,R)} |u_{R3,z}(t,x)|^2 \, dx \, dt \Bigr)^{\frac{1}{2}} \\
&\qquad\leq C  \mathbf{C}^2(u) \|\rho_0\|_{C^1}\left(T^{\frac{3}{8}}\|u\|_{X_T}^{\frac{1}{2}}+T^{\frac{1}{4}}\|\rho_0-1\|_{C^1}\right).
\end{aligned}
\end{equation*}
\end{prop}

\begin{proof} We divide the proof into the following three steps:

\noindent \textbf{(i) Estimates for spatial integrability.}\smallskip

We first observe from \eqref{esthtk} that
\begin{equation}\label{S4eq31}
    \| \sup_{x_0} (|D|^{\eta} \partial_t^i K_{x_0}(t,y) |y|) \|_{L^p_y} \lesssim t^{-i + \frac{1-\eta}{2} -\frac{3}{2}(1-\frac{1}{p})}, 
\end{equation}
for $\eta = 0, 1+\kappa$ and $i = 0, 1$.

We have the formula of $u_{R3,x_0}$ from \eqref{S4eq0} as
\begin{equation}\label{S4eq33}
    \begin{aligned}
        u_{R3,x_0}&=-\int_0^t\int_{\mathbb{R}^3}\partial_\tau K_{x_0}(t-\tau,x-y)\left(\rho_0(x_0)-\rho_0(y)\right)u(\tau,y)dyd\tau\\
        &\qquad+\left(\rho_0(x_0)-\rho_0(x)\right)u(t,x)-\int_{\mathbb{R}^3}K_{x_0}(t,x-y)\left(\rho_0(x_0)-\rho_0(y)\right)u_0(y)dy\\
        &:=\sum_{i=1}^3u_{R3,x_0}^i.
    \end{aligned}
\end{equation}
When we fix $x_0=x$, the second term $u_{R3,x_0}^2$ vanishes. The third term can be written as
\begin{equation}\label{S4eq34}
    \begin{aligned}
        &u_{R3,x_0}^3=\int_{\mathbb{R}^3}K_{x_0}(t,x-y)\left(\rho_0(x_0)-\rho_0(y)\right)u_0(y)dy\\
        &\quad=\int_{\mathbb{R}^3}K_{x_0}(t,x-y)\left(\rho_0(x_0)-\rho_0(y)\right)\left(\left(\rho_0u_0(y)-(\rho u)(t,y)\right)-(\rho_0-1)u_0(y)+(\rho u)(t,y)\right)dy\\
        &\quad=\int_{\mathbb{R}^3}K_{x_0}(t,x-y)\left(\rho_0(x_0)-\rho_0(y)\right)\int_0^t\left(\Delta u(\tau,y)-\nabla\cdot(\rho u\otimes u)(\tau,y)+\nabla P\right)(\tau)d\tau dy\\
        &\qquad+\int_{\mathbb{R}^3}K_{x_0}(t,x-y)\left(\rho_0(x_0)-\rho_0(y)\right)\left((1-\rho_0)u_0(y)+(\rho u)(t,y)\right)dy\\
        &\quad=\int_{\mathbb{R}^3}\nabla\left(K_{x_0}(t,x-y)\left(\rho_0(x_0)-\rho_0(y)\right)\right)\cdot\int_0^t\left(\nabla u(\tau,y)-(\rho u\otimes u)(\tau,y)\right)(\tau)d\tau dy\\
        &\qquad+\int_{\mathbb{R}^3}K_{x_0}(t,x-y)\left(\rho_0(x_0)-\rho_0(y)\right)\int_0^t\nabla P(\tau)d\tau dy\\
        &\qquad+\int_{\mathbb{R}^3}K_{x_0}(t,x-y)\left(\rho_0(x_0)-\rho_0(y)\right)\left((1-\rho_0)u_0(y)+(\rho u)(t,y)\right)dy\\
        &\quad:=\sum_{i=1}^3u_{R3,x_0}^{3i}.\\
    \end{aligned}
\end{equation}
By \eqref{S4eq31}, it is easy to prove
\begin{align}
        &\sup_{t\leq T}t^{\frac{1}{2}}\left(\|\tilde{u}_{R3}^1(t)\|_{L^\infty}+\|\tilde{u}_{R3}^{31}(t)\|_{L^\infty}+\|\tilde{u}_{R3}^{33}(t)\|_{L^\infty}\right)\label{S4eq32}\\
        &\qquad\qquad+\sup_{t\leq T}t^{\frac{5}{12}}\left(\|\tilde{u}_{R3}^1(t)\|_{L^{12}}+\|\tilde{u}_{R3}^{31}(t)\|_{L^{12}}+\|\tilde{u}_{R3}^{33}(t)\|_{L^{12}}\right)\nonumber\\
        &\quad\lesssim \sup_{t\leq T}t^{\frac{1}{2}}\int_0^t(t-\tau)^{-\frac{1}{2}}\tau^{-\frac{1}{2}}d\tau\|u\|_{X_T}+\sup_{t\leq T}t^{\frac{5}{12}}\int_0^t(t-\tau)^{-\frac{1}{2}}\tau^{-\frac{5}{12}}d\tau\|u\|_{X_T}\nonumber\\
        &\qquad+\sup_{t\leq T}t^{\frac{1}{2}}\int_0^t\|\rho_0\|_{C^1}\bigl\|\sup_{x_0}\left|\left(\nabla K_{x_0}(t,y)|y|+K_{x_0}(t,y)\right) \right|\bigr\|_{L^{\frac{4}{3}}_y}\left(\|\nabla u(\tau)\|_{L^4}d\tau+\|u(\tau)\|_{L^\infty}\|u(\tau)\|_{L^4} \right)d\tau\nonumber\\
        &\qquad+\sup_{t\leq T}t^{\frac{5}{12}}\int_0^t\|\rho_0\|_{C^1}\bigl\|\sup_{x_0}\left|\left(\nabla K_{x_0}(t,y)|y|+K_{x_0}(t,y)\right)\right|\bigr\|_{ L^{\frac{6}{5}}_y}\left(\|\nabla u(\tau)\|_{L^4}d\tau+\|u(\tau)\|_{L^\infty}\|u(\tau)\|_{L^4} \right)d\tau\nonumber\\
        &\qquad+\sup_{t\leq T}t^{\frac{1}{2}}\|\sup_{x_0}\left(K_{x_0}(t,y)|y|\right) \bigr\|_{L^{2}_y}\|\rho_0\|_{C^1}\|\rho_0-1\|_{L^\infty}\|u_0\|_{L^2}\nonumber\\
        &\qquad+\sup_{t\leq T}t^{\frac{5}{12}}\|\sup_{x_0}\left(K_{x_0}(t,y)|y|\right) \bigr\|_{L^{\frac{12}{7}}_y}\|\rho_0\|_{C^1}\|\rho_0-1\|_{L^\infty}\|u_0\|_{L^2}\nonumber\\
        &\qquad+\sup_{t\leq T}t^{\frac{1}{2}}\|\sup_{x_0}\left(K_{x_0}(t,y)|y|\right) \bigr\|_{L^{1}_y}\|\rho_0\|_{C^1}\|u(t)\|_{L^\infty}\nonumber\\
        &\qquad+\sup_{t\leq T}t^{\frac{5}{12}}\|\sup_{x_0}\left(K_{x_0}(t,y)|y|\right) \bigr\|_{L^{1}_y}\|\rho_0\|_{C^1}\|u(t)\|_{L^{12}}\nonumber\\
        &\quad\lesssim T^{\frac{3}{8}}\mathbf{C}^2(u)\|\rho\|_{C^1}\|u\|_{X_T}^{\frac{1}{2}}+T^{\frac{1}{4}}\mathbf{C}^3(u)\|\rho\|_{C^1}\|\rho_0-1\|_{L^\infty}\nonumber.
    \end{align}
    For the most subtle term $u_{R3,x_0}^{32}$, we use the decomposition \eqref{S4eq61}. Note that
    \begin{align}
        &\|P_1(t)\|_{L^2}\lesssim \|\rho_0\|_{C^1}\int_0^t\tau^{-1+\frac{1-\eta}{2}}d\tau\mathbf{C}^2(u)\|u_0\|_{L^2}\lesssim t^{\frac{1-\eta}{2}}\|\rho_0\|_{C^1}\mathbf{C}^3(u),\label{S4eq67}\\
        &\|P_2(t)\|_{L^{\frac{12}{7}}}\lesssim \|\rho\|_{L^\infty}\int_0^t\|u(\tau)\|_{L^{12}}\|\nabla u(\tau)\|_{L^2}d\tau\lesssim t^{\frac{1}{12}}\|\rho_0\|_{C^1}\mathbf{C}(u)\|u\|_{X_T},\nonumber\\
        &\|P_3(t)\|_{L^2}\lesssim \|\rho_0-1\|_{L^\infty}\|u_0\|_{L^2},\nonumber
    \end{align}
    which implies
    \begin{align}
        &\sup_{t\leq T}t^{\frac{1}{2}}\left(\|\tilde{u}_{R3}^{32}(t)\|_{L^\infty}\right)+\sup_{t\leq T}t^{\frac{5}{12}}\left(\|\tilde{u}_{R3}^{32}(t)\|_{L^{12}}\right)\label{S4eq66}\\
        &\quad\lesssim \sup_{t\leq T}t^{\frac{1}{2}}\|\sup_{x_0}\left(K_{x_0}(t,y)|y|\right) \bigr\|_{L^{2}_y}\left(\|P_1(t)\|_{L^2}+\|P_3(t)\|_{L^2}\right)\nonumber\\
        &\qquad+\sup_{t\leq T}t^{\frac{1}{2}}\|\sup_{x_0}\left(K_{x_0}(t,y)|y|\right) \bigr\|_{L^{\frac{12}{5}}_y}\|P_2(t)\|_{L^\frac{12}{7}}\nonumber\\
        &\qquad+\sup_{t\leq T}t^{\frac{5}{12}}\|\sup_{x_0}\left(K_{x_0}(t,y)|y|\right) \bigr\|_{L^{\frac{12}{7}}_y}\left(\|P_1(t)\|_{L^2}+\|P_3(t)\|_{L^2}\right)\nonumber\\
        &\qquad+\sup_{t\leq T}t^{\frac{1}{2}}\|\sup_{x_0}\left(K_{x_0}(t,y)|y|\right) \bigr\|_{L^{2}_y}\|P_2(t)\|_{L^\frac{12}{7}}\nonumber\\
        &\quad\lesssim \|\rho_0\|_{C^1}\mathbf{C}^2(u)\left(T^{\frac{3}{8}}\|u\|_{X_T}+T^{\frac{1}{4}}\|\rho_0-1\|_{L^\infty}\right).\nonumber
    \end{align}

    For $C^{1+\kappa}$ estimate, we have
    \begin{align*}
        &|D|^{1+\kappa}u_{R3,x_0}=-\int_0^t\int_{\mathbb{R}^3}\partial_\tau |D|^{1+\kappa}K_{x_0}(t-\tau,x-y)\left(\rho_0(x_0)-\rho_0(y)\right)\left(u(\tau,y)-u(t,y)\right)dyd\tau\label{S4eq35}\\
        &\qquad-\int_{\mathbb{R}^3}|D|^{1+\kappa}K_{x_0}(t,x-y)\left(\rho_0(x_0)-\rho_0(y)\right)\left(u_0(y)-u(t,y)\right)dy.\nonumber
    \end{align*}
   Note that the boundary term $\tau=t$ vanishes. We use similar methods as \eqref{S4eq34} for the second term (but change the $\nabla\left(K_{x_0}(t,x-y)(\rho_0(x)-\rho_0(y))\right)$ to $\nabla\left(|D|^{1+\kappa}K_{x_0}(t,x-y)(\rho_0(x)-\rho_0(y))\right)$), then we have
    \begin{align*}
            \sup_{t\leq T}t^{1+\frac{\kappa}{2}}\||D|^{1+\kappa}u_{R3,x_0}(t)|_{x_0=x}\|_{L^\infty}\lesssim \|\rho_0\|_{C^1}\mathbf{C}^3(u)\left(T^{\frac{3}{8}}\|u\|_{X_T}^\frac{1}{2}+T^{\frac{1}{4}}\|\rho_0-1\|_{L^\infty}\right).
        \end{align*}
As a consequence of the estimate above and \eqref{S4eq32}, \eqref{S4eq66}, we obtain
\begin{equation}\label{uR3spa}
\begin{aligned}
&\sup_{t \leq T} \Bigl( t^{\frac{1}{2}} \|\tilde{u}_{R3}(t)\|_{L^\infty} + \sup_{t \leq T} t^{1 + \frac{\kappa}{2}} \| (\na|D|^{\kappa} u_{R3,x_0})(t)|_{x_0=x} \|_{L^{\infty}} + t^{\frac{5}{12}} \|\tilde{u}_{R3}(t)\|_{L^{12}} \Bigr) \\
&\quad \leq C\|\rho_0\|_{C^1}\mathbf{C}^3(u)\left(T^{\frac{3}{8}}\|u\|_{X_T}^\frac{1}{2}+T^{\frac{1}{4}}\|\rho_0-1\|_{L^\infty}\right).
\end{aligned}
\end{equation}
\noindent \textbf{(ii) Estimates for time differences.}\smallskip

Furthermore, if we denote $a = t-s < \frac{s}{4}$, we use \eqref{S4eq33} and have the decomposition
\begin{align*}
        &\delta_a^tu_{R3,x_0}(t,s,x)=-\int_s^t\int_{\R^3}\partial_\tau K_{x_0}(t-\tau,x-y)\left(\rho_0(x_0)-\rho_0(y)\right)u(\tau,y)dyd\tau\\
    &\quad-\int_0^s\int_{\mathbb{R}^3}\delta_a^t\partial_\tau K_{x_0}(s-\tau,x-y)\left(\rho_0(x_0)-\rho_0(y)\right)\left(u(\tau,y)-u(s,y)\right)dyd\tau\\
    &\quad-\int_{\mathbb{R}^3}K_{x_0}(t-s,x-y)\left(\rho_0(x_0)-\rho_0(y)\right)u(s,y)dy+\left(\rho_0(x_0)-\rho_0(x)\right)u(s,x)\\
    &\quad+\int_{\R^3}\delta_a^tK_{x_0}(s,x-y)\left(\rho_0(x_0)-\rho_0(y)\right)u(s,y)dy+\left(\rho_0(x_0)-\rho_0(x)\right)\delta_a^tu(s,x)\\
    &\quad-\int_{\R^3}\delta_a^tK_{x_0}(s,x-y)\left(\rho_0(x_0)-\rho_0(y)\right)u_0(y)dy\\
    &:=\sum_{i=1}^7Du_{R3,x_0}^i,
\end{align*}
and we will also denote $D\tilde{u}_{R3}^i(t,s,x)=(Du_{R3,x_0}^i)(t,s,x)|_{x_0=x}$. Easy to see that $D\tilde{u}_{R3}^4(t,s,x)=D\tilde{u}_{R3}^6(t,s,x)=0$. \\
Observing  that 
\begin{equation*}
    \begin{aligned}
        \left|\partial_t K_{x_0}(t,y) \right|&=\left|\Bigl(\frac{\rho_0(x_0)}{4\pi t}\Bigr)^{\frac{3}{2}}\exp\left(-\frac{\rho_0(x_0)|y|^2}{4t}\right)\left(-\frac{3}{2}t^{-1}+\frac{\rho_0(x_0)|y|^2}{4t^2}\right)\right|\\
        &\lesssim \frac{C_2^{\frac{3}{2}}}{t^{\frac{5}{2}}}\exp\left(-\frac{C_1|y|^2}{8t}\right),
    \end{aligned}
\end{equation*}
and
\begin{equation*}
    \delta_a^t K_{x_0}(s,y)=\int_s^t\partial_\tau K_{x_0}(\tau,y)d\tau,
\end{equation*}
we find
\begin{equation}\label{S4eq9}
\begin{split}
    \| \sup_{x_0} |\delta_a^t K_{x_0}(s,y)| |y| \|_{L^p_y} \lesssim s^{\frac{1}{2}-\frac{3}{2}(1-\frac{1}{p})}\min\{1, a s^{-1} \}.
\end{split}
\end{equation}
Similarly, observing that 
\begin{equation*}
    \begin{aligned}
        |\delta_a^t\partial_sK_{x_0}(s,y)|=\left|\int_s^t \partial_\tau^2 K_{x_0}(\tau,y)d\tau\right|\lesssim \frac{C^{\frac{3}{2}}}{s^{-\frac{7}{2}}}\exp\Bigl(-\frac{c|y|^2}{8s}\Bigr)
    \end{aligned}
\end{equation*}
for any $x_0$. We get, by a similar derivation of \eqref{S4eq9}, that
\begin{equation}\label{S4eq10}
\begin{split}
&\| \sup_{x_0} |\partial_s\delta_a^t K_{x_0}(s,y)| |y| \|_{L^2_y} \lesssim s^{-1}\min\bigl\{ s^{-\frac{1}{4}}, a^{\frac{3}{4}} s^{-1} \bigr\},\\
& \| \sup_{x_0} |\partial_s\delta_a^t K_{x_0}(s,y)| |y| \|_{L^{12/7}_y} \lesssim s^{-1}\min\bigl\{ s^{-\frac{1}{8}}, a^{\frac{7}{8}} s^{-1} \bigr\}.
\end{split}
\end{equation}
So by \eqref{S4eq10} we have
\begin{align*}
        \|D\tilde{u}_{R3}^2(t,s)\|_{L^\infty}&\lesssim \int_s^t(s-\tau)^{-\frac{1}{2}}\min\{1,a(s-\tau)^{-1}\}\min\{1,(s-\tau)^{\varkappa_1}\tau^{-\varkappa_1}\}\tau^{-\frac{1}{2}}\|u\|_{X_T}\\
        &\lesssim s^{-\varkappa_1}a^{\varkappa_1}\|u\|_{X_T},\\
        \|D\tilde{u}_{R3}^2(t,s)\|_{L^{12}}&\lesssim \int_s^t(s-\tau)^{-\frac{5}{12}}\min\{1,a(s-\tau)^{-1}\}\min\{1,(s-\tau)^{\varkappa_2}\tau^{-\varkappa_2}\}\tau^{-\frac{5}{12}}\|u\|_{X_T}\\
        &\lesssim s^{-\varkappa_2}a^{\varkappa_2}\|u\|_{X_T}.
    \end{align*}
Note that
\begin{equation*}
    D\tilde{ u}_{R3}^3=\int_s^t\partial_\tau K_{x_0}(t-\tau,x-y)|_{x_0=x}\left(\rho_0(x)-\rho_0(y) \right)u(s,y)dyd\tau,
\end{equation*}
we have
\begin{equation*}
    D\tilde{u}_{R3}^1+D\tilde{u}_{R3}^3=-\int_s^t\int_{\R^3}\partial_\tau K_{x_0}(t-\tau,x-y)|_{x_0=x}\left(\rho_0(x)-\rho_0(y)\right)\left(u(\tau,y)-u(s,y)\right)dyd\tau,
\end{equation*}
so
\begin{align*}
        \|D\tilde{u}_{R3}^1(t,s)+D\tilde{u}_{R3}^3(t,s)\|_{L^\infty}&\lesssim \|\rho_0\|_{C^1}\int_s^t(t-\tau)^{-\frac{1}{2}}(s-\tau)^{\varkappa_1}s^{-\frac{1}{2}-\varkappa_1}d\tau\|u\|_{X_T}\\
        &\lesssim s^{-\varkappa_1}a^{\varkappa_1}\|u\|_{X_T}\|\rho_0\|_{C^1},\\
        \|D\tilde{u}_{R3}^1(t,s)+D\tilde{u}_{R3}^3(t,s)\|_{L^{12}}&\lesssim \|\rho_0\|_{C^1}\int_s^t(t-\tau)^{-\frac{1}{2}}(s-\tau)^{\varkappa_2}s^{-\frac{5}{12}-\varkappa_2}d\tau\|u\|_{X_T}\\
        &\lesssim s^{\frac{1}{12}-\varkappa_2}a^{\varkappa_2}\|u\|_{X_T}\|\rho_0\|_{C^1}.
    \end{align*}
Finally, $Du_{R3,x_0}^5$ and $Du_{R3,x_0}^7$ have the following cancellation property,
\begin{align*}
        &D\tilde{u}_{R3}^5+D\tilde{u}_{R3}^7=\int_{\R^3}\nabla_y\big(\delta_a^tK_{x_0}(s,x-y)\big)\big|_{x_0=x}\left(\rho_0(x)-\rho_0(y)\right)\int_0^s\left(\nabla u-\rho u\otimes u\right)(\tau)d\tau\\
        &\qquad+\int_{\R^3}\big(\delta_a^tK_{x_0}(s,x-y)\big)\big|_{x_0=x}\left(\rho_0(x)-\rho_0(y)\right)\int_0^s\nabla P(\tau)d\tau\\
        &\qquad-\int_{\mathbb{R}^3}\big(\delta_a^tK_{x_0}(s,x-y)\left(\rho_0(x)-\rho_0(y) \right)\big)\big|_{x_0=x}(\rho(s,y)-1)u(s,y)dy\\
        &\qquad+\int_{\mathbb{R}^3}\big(\delta_a^tK_{x_0}(s,x-y)\left(\rho_0(x)-\rho_0(y) \right)\big)\big|_{x_0=x}(\rho_0(y)-1)u_0(y)dy.
    \end{align*}
So by \eqref{esthtk} and \eqref{S4eq67}, via s similar procedure as \eqref{S4eq32} and \eqref{S4eq66}, we have
\begin{equation*}
    \begin{aligned}
        &\sup_{s<t<T}s^{\frac{1}{2}+\varkappa_1}\frac{\|D\tilde{u}_{R3}^5(t,x)+D\tilde{u}_{R3}^7(t,s)\|_{L^\infty}}{(t-s)^{\varkappa_1}}+\|D\tilde{u}_{R3}^5(t,x)+\sup_{s<t<T}s^{\frac{5}{12}+\varkappa_2}\frac{D\tilde{u}_{R3}^7(t,s)\|_{L^{12}}}{(t-s)^{\varkappa_2}}\\
        &\quad\lesssim \|\rho_0\|_{C^1}^2\mathbf{C}^3(u)\left(T^{\frac{3}{8}}\|u\|_{X_T}^{\frac{1}{2}}+T^{\frac{1}{4}}\|\rho_0-1\|_{L^\infty}\right).
    \end{aligned}
\end{equation*}

By summarizing the above estimates, we arrive at
\begin{equation}\label{uR3diff}
\begin{aligned}
&\sup_{s < t \leq T} s^{\frac{1}{2} + \varkappa_1} \frac{\| \tilde{u}_{R3}(t) - \tilde{u}_{R3}(s) \|_{L^\infty}}{(t-s)^{\varkappa_1}}
+ \sup_{s < t \leq T} s^{\frac{5}{12} + \varkappa_2} \frac{\| \tilde{u}_{R3}(t) - \tilde{u}_{R3}(s) \|_{L^{12}}}{(t-s)^{\varkappa_2}} \\
&\qquad \lesssim \|\rho_0\|_{C^1}^2\mathbf{C}^3(u)\left(T^{\frac{3}{8}}\|u\|_{X_T}^{\frac{1}{2}}+T^{\frac{1}{4}}\|\rho_0-1\|_{L^\infty}\right)
\end{aligned}
\end{equation}

\noindent \textbf{(iii) Estimates for the $\mathrm{BMO}^{-1}$ norm.}\smallskip

Finally let us turn to the $\mathrm{BMO}^{-1}$ type estimates. For any $x_0 \in \mathbb{R}^3$, by virtue of \eqref{S4eq0}, we write
\begin{align*}
u_{R3,x_0} =& \int_0^t \int_{\R^3} \partial_\tau K_{x_0}(t-\tau, x-y) (\rho_0(x_0)-\rho_0(y)) u(\tau,y) \, dy \, d\tau \\
&+ (\rho_0(x_0)-\rho_0(x)) u(t,x) - \int_{\R^3} K_{x_0}(t, x-y) (\rho_0(x_0)-\rho_0(y)) u_0(y) \, dy \\
=& (\rho_0(x_0)-\rho_0(x)) \Bigl( \int_0^t \int_{\R^3} \partial_\tau K_{x_0}(t-\tau, x-y) u(\tau,y) \, dy \, d\tau + u(x) - K_{x_0}(t) \ast u_0(x) \Bigr) \\
&+ \left(\int_0^t \int_{\R^3} \partial_\tau K_{x_0}(t-\tau, x-y) (\rho_0(x)-\rho_0(y)) u(\tau,y) \, dy \, d\tau \right.\\
&\left. - \int_{\R^3} K_{x_0}(t, x-y) (\rho_0(x)-\rho_0(y)) u_0(y) \, dy\right) \\
%= &%(\rho_0(x_0)-\rho_0(x)) u_{R3,x_0}^1 + u_{R3,x_0}^2 
=:& (\rho_0(x_0)-\rho_0(x)) \sum_{i=1}^3 u_{R3,x_0}^{1i} + u_{R3,x_0}^2.
\end{align*}
Since we still can write
\begin{equation*}
    \begin{aligned}
        &\int_{\R^3} K_{x_0}(t, x-y) (\rho_0(x)-\rho_0(y)) u_0(y) \, dy\\
        &=\int_{\R^3}K_{x_0}(t,x-y)(\rho_0(x_0)-\rho_0(y))\left((\rho_0u_0(y)-(\rho u)(t,y))+(\rho u)(t,y)\right)dy\\
        &\quad+\int_{\R^3}K_{x_0}(t,x-y)(\rho_0(x_0)-\rho_0(y))\left((1-\rho_0)u_0(y)\right)dy\\
        &=\int_{\R^3}\nabla_y\left(K_{x_0}(t,x-y)(\rho_0(x_0)-\rho_0(y))\right)\cdot\int_0^t\left(\nabla u-\rho u\otimes u\right)(\tau,y)d\tau dy\\
        &\quad+\int_{\R^3}\left(K_{x_0}(t,x-y)(\rho_0(x_0)-\rho_0(y))\right)\cdot\int_0^t\nabla P(\tau,y)d\tau dy\\
        &\quad+\int_{\R^3}K_{x_0}(t,x-y)(\rho_0(x_0)-\rho_0(y))\left((\rho u)(t,y)+(1-\rho_0)u_0(y)\right)dy,
    \end{aligned}
\end{equation*}
we can get by a similar derivation of  \eqref{uR3spa} that
\begin{equation*}
\sup_{t \leq T} t^{\frac{5}{12}} \| u_{R3,x_0}^2(t) \|_{L^\infty} \lesssim  \mathbf{C}^2(u)\|\rho_0\|_{C^1}^2\left(T^{\frac{3}{8}} \|u\|_{X_T}^{\frac{1}{2}}+T^{\frac{1}{4}}\|\rho_0-1\|_{L^\infty}\right).
\end{equation*}
It remains to handle the  estimate of the first term. 
%In particular, for any $R < b$ and $z \in \mathbb{R}^3$, we want to estimate
%\begin{equation*}R^{-\frac{3}{2}} \left( \int_0^{R^2} \int_{B(z,R)} | (\rho_0(z)-\rho_0(x)) u_{R3,z}^1(t,x) |^2 \, dx \, dt \right)^{\frac{1}{2}}.\end{equation*}
Note that for any $z \in \mathbb{R}^3$ and $x \in B(z,R)$, we have
\begin{equation*}
|\rho_0(z) - \rho_0(x)| \lesssim R \|\nabla \rho_0\|_{L^\infty}.
\end{equation*}
By the definitions of $u_{R3,z}^{12}$, $u_{R3,z}^{13}$ and the $X_T$ norm, we have
\begin{equation*}
R^{-\frac{3}{2}} \Bigl( \int_0^{R^2} \int_{B(z,R)} | (\rho_0(z)-\rho_0(x)) (u_{R3,z}^{12} + u_{R3,z}^{13})(t,x) |^2 \, dx \, dt \Bigr)^{\frac{1}{2}} \lesssim R^{\frac{1}{2}} \|\rho_0\|_{C^1} \mathbf{C}(u).
\end{equation*}
By Lemma \ref{CVPDE} and Lemma \ref{eng1}, \ref{eng3}, for $u_{R3,z}^{11}$, we have
\begin{equation*}
    \begin{aligned}
       & \sup_{t\leq T}t^{\frac{1}{2}}\|u_{R3,z}^{11}(t)\|_{L^\infty}\lesssim \|u\|_{X_T},\\
       & \sup_{t\leq T}\|u_{R3,z}^{11}(t)\|_{L^2}\lesssim \sup_{t\leq T}\|u(t)\|_{L^2}+\sup_{s<t\leq T}s^{\eta}\frac{\|u(t)-u(s)\|_{L^2}}{(t-s)^\eta}\lesssim \mathbf{C}^3(u) ,
    \end{aligned}
\end{equation*}
for some $\eta\in(0,1)$. Then for $x\in B(z,R)$, we have
\begin{equation*}
\begin{aligned}
&R^{-\frac{3}{2}} \Bigl( \int_0^{R^2} \int_{B(z,R)} | (\rho_0(z)-\rho_0(x)) u_{R3,z}^{11}(t,x) |^2 \, dx \, dt \Bigr)^{\frac{1}{2}} \\
&\quad \lesssim R^{\frac{1}{4}} \|\rho_0\|_{C^1} \Bigl( \int_0^{R^2} \| u_{R3,z}^{11}(t) \|_{L^4}^2 \, dt \Bigr)^{\frac{1}{2}} \\
&\quad \lesssim R^{\frac{3}{4}} \|\rho_0\|_{C^1} \mathbf{C}^3(u)\|u\|_{X_T}^{\frac{1}{2}}.
\end{aligned}
\end{equation*}
Since $R < b$, for $T \sim b^2$ we obtain the following estimate for the $\mathrm{BMO}^{-1}$ norm:
\begin{equation*}
\sup_{R \leq b} \sup_{z \in \mathbb{R}^3} R^{-\frac{3}{2}} \Bigl( \int_0^{R^2} \int_{B(z,R)} | u_{R3,z}(t,x) |^2 \, dt \, dx \Bigr)^{\frac{1}{2}} \lesssim CT^{\frac{3}{8}} \|\rho_0\|_{C^1} \mathbf{C}^3(u)\|u\|_{X_T}^{\frac{1}{2}}.
\end{equation*}
Combining this with \eqref{uR3spa} and \eqref{uR3diff} completes the proof.
\end{proof}

\section{Local existence}\label{sec6}

We first present the following a priori estimates:

\begin{prop}\label{S5prop1}
For any $(\rho_0, u_0)$ satisfying \eqref{bdrho0},  $\|\rho_0\|_{C^1} + \|u_0\|_{L^2} < \infty$, there exists $\varepsilon_0>0$ such that if $\|u_0\|_{\mathrm{BMO}_b^{-1}} < \varepsilon_0$ for some $b < 1$, and $(\rho, u)$ is a smooth enough solution to \eqref{eqinns} with initial data $(\rho_0, u_0)$, then there exists $T \sim b^2$ such that $u \in X_T,$  $u$ satisfies \eqref{estL2u}, and
\begin{equation} \label{S5eq1}
0 < C_1 < \rho(t,x) < C_2\andf \|u\|_{X_T} \leq 8C\epsilon_0^{\frac{5}{6}}\|u_0\|_{L^2}^{\frac{1}{6}}\|\rho_0\|_{C^1}.
\end{equation}
\end{prop}
\begin{proof} Throughout all of our proof, the constant will depend on $C_1$ and $C_2$, and for simplicity we omit them. In view of \eqref{decop}, we get, by applying  \eqref{estnorm} and Propositions \ref{uL}, \ref{uN}, \ref{uR1}, \ref{propuR2} and \ref{uR3},  that for any $T\leq 1$,
\begin{align}
\|u\|_{X_T}\leq C\|\rho_0\|_{C^1}\Bigl(&\|u_0\|_{\mathrm{BMO}_b^{-1}} + \|u_0\|_{\mathrm{BMO}_b^{-1}}^{\frac{5}{6}} \|u_0\|_{L^2}^{\frac{1}{6}}+\|u\|_{X_T}^2\label{S5eq6}\\
&+T^{\f1{4}}\bigl(1+\|u_0\|_{L^2}+\|u\|_{X_T}\bigr)^3\Bigr).\nonumber
\end{align}
We define
\beq \label{S5eq1}
\begin{split}
\ep_0=\min\left\{(100C(1+\|\rho_0\|_{C^1}))^{-1}(1+\|u_0\|_{L^2})^{-1},\|u_0\|_{L^2}\right\}.
\end{split}
\eeq
Since $u_0\in VMO^{-1}$, there exists $b_*$ such that $\|u_0\|_{BMO_{b_*}^{-1}}\leq \ep_0$. Take
\begin{align*}
    &b=\min\left\{(\frac{1}{10}(2+\|u_0\|_{L^2})^{-3}\ep_0)^{10},\frac{b_*}{2}\right\},\\
&T^*=4C_2b^2.
\end{align*}
Then we define 
\begin{equation*}
    T_0:=\sup_{t\leq T^*}\{\|u\|_{X_T}\leq 8C\ep_0^{\frac{5}{6}}\|u_0\|_{L^2}^{\frac{1}{6}}\|\rho_0\|_{C^1}\}
\end{equation*}
We want to prove that $T_0=T^*$. If $T_0<T^*$, then by \eqref{S5eq6} and \eqref{S5eq1}, we have
\begin{equation*}
    \|u\|_{X_{T_0}}\leq 2C\|\rho_0\|_{C^1}\ep_0^{\frac{5}{6}}\|u_0\|_{L^2}^{\frac{1}{6}}+64C^2\|\rho_0\|_{C^1}^2\ep_0^{\frac{5}{3}}\|u_0\|_{L^2}^{\frac{1}{3}}+\frac{1}{10}C\|\rho_0\|_{C^1}\ep_0\leq 4C\ep_0^{\frac{5}{6}}\|u_0\|_{L^2}^{\frac{1}{6}}\|\rho_0\|_{C^1}.
\end{equation*}
This implies that $T_0=T^*$, so for any $T\leq T^*$, if $\|u_0\|_{BMO^{-1}_b}\leq \ep_0$ with $\ep_0,T^*,b$ above, there holds
\begin{align*}\label{S5eq2}
\|u\|_{X_T}\leq 8C\|\rho_0\|_{C^1}\ep_0^{\frac{5}{6}}\|u_0\|_{L^2}^{\frac{1}{6}},
\end{align*}
from which and \eqref{estL2u}, we conclude the proof of Proposition \ref{S5prop1}.
\end{proof}

We are now in a position to complete the proof of Theorem \ref{mainthm}.

\begin{proof}
 First, we take smooth approximating sequences $(\rho_{0,\varepsilon}, u_{0,\varepsilon})$ such that as $\varepsilon \to 0$,
\begin{align*}
&u_{0,\varepsilon} \to u_0 \quad \text{in } L^2 \cap \mathrm{VMO}^{-1}, \qquad
\rho_{0,\varepsilon} \rightharpoonup \rho_0 \quad \text{in } L^\infty \text{ weak-}* \andf \\&\rho_{0,\varepsilon} \to \rho_0 \quad \text{in } L^p_{\rm{loc}}(\R^3)\quad \mbox{for any}\ p<\infty.
\end{align*}
Then, by the classical well-posedness theory for the inhomogeneous Navier–Stokes equations, there exists a unique solution $(\rho^\varepsilon, u^\varepsilon)$ on $[0, T^\varepsilon]$. Furthermore, we deduce from Proposition \ref{S5prop1} that $T^\varepsilon$ has a uniform lower bound, which we denote by $T^\star$, so that \eqref{S5eq1} holds on $[0,T^\star]$ for the approximate solutions $(\rho^\varepsilon, u^\varepsilon).$

By the a priori estimates \eqref{estL2u} and \eqref{S5eq1} for $(\rho^\varepsilon, u^\varepsilon)$, there exists a subsequence (still denoted by $(\rho^\varepsilon, u^\varepsilon)$) such that $(\rho^\varepsilon, u^\varepsilon)$ converges to $(\rho, u)$ in the weak-$*$ sense on $[0,T^\star]$. 
Then we deduce from the proof of Theorem 2.1 of \cite{PL96} that thus obtained $(\rho,u)$ is indeed a weak solution of the system 
\eqref{eqinns}. Furthermore, the solution satisfy  \eqref{estL2u} and \eqref{S1eqw}.  This completes the proof of Theorem \ref{mainthm}.
\end{proof}

 \section{Global existence}\label{sec7}
For the global existence result, we rely on the following theorem from \cite{HSWZZ}.

\begin{thm}\label{HWZ}
Let $p \in (3, \infty)$. There exists a constant $\varepsilon_1 > 0$ such that if
\begin{equation*}
\|\rho_0 - 1\|_{L^\infty} + \|\mathbb{P}(\rho_0 u_0)\|_{\dot{B}_{p,\infty}^{-1 + \frac{3}{p}}} < \varepsilon_1,
\end{equation*}
then the system \eqref{eqinns} has a global strong solution.
\end{thm}

Now we are ready to prove Theorem \ref{thm.global-existence}.

\begin{proof}[Proof of Theorem \ref{thm.global-existence}]
By applying Proposition \ref{uL}, \ref{Prop4.2}–\ref{uR3}, we can prove that there exists a constant $C_0<1$, such that for any $T,b < C_0$, there holds
\begin{align}
\|u\|_{X_T}\leq C\|\rho_0\|_{C^2}\Bigl(&\|u_0\|_{\mathrm{BMO}^{-1}} + \|u_0\|_{\mathrm{BMO}^{-1}}^{\frac{5}{6}} \|u_0\|_{L^2}^{\frac{1}{6}}+\|u\|_{X_T}^2\label{S6eq2}\\
&+C\|\rho_0\|_{C^2} ^2\mathbf{C}^3(u)\left(T^{\frac{1}{4}}\left(\|u\|_{X_T}+\|\rho_0-1\|_{L^\infty}\right)+T^{\frac{3}{8}}\|u\|_{X_T}^{\f12}+T^{\frac{5}{8}}\right).\nonumber
\end{align}
We define 
\begin{equation*}
    \begin{aligned}
        \ep_0=\min\left\{(100C(1+\|\rho_0\|_{C^2}))^{-1}(1+\|u_0\|_{L^2})^{-1},\|u_0\|_{L^2}\right\}.
    \end{aligned}
\end{equation*}
Then, for any $\ep<\ep_0$, define 
\begin{equation*}
    T_\ep:=\ep^{\frac{4}{3}}\|u_0\|_{L^2}^{\frac{2}{9}}\left(10C(2+\|\rho_0\|_{C^2}+\|u_0\|_{L^2})\right)^{-10}.
\end{equation*}
Let 
\begin{equation*}
    T_\ep^0=\sup\left\{T\leq T_\ep: \|u\|_{X_T}\leq 8C\|\rho_0\|_{C^2}\ep^{\f56}\|u_0\|_{L^2}^{\f16}\right\}.
\end{equation*}
By \eqref{S6eq2} and the definition of $T_\ep$, $\ep_0$, if $\|\rho_0-1\|_{L^\infty}+\|u_0\|_{BMO^{-1}}\leq \ep<\ep_0$, $T_\ep^0<T_\ep$, then
\begin{equation*}
    \begin{aligned}
        \|u\|_{X_{T_\ep^0}}&\leq 2C\|\rho_0\|_{C^2}\ep^{\f56}\|u_0\|_{L^2}^{\f16}+\f1{100}\|u\|_{X_{T_\ep^0}}+3C\|\rho_0\|_{C^2}\ep^{\f56}\|u_0\|_{L^2}^{\f16}\\
        &\leq 7C\|\rho_0\|_{C^2}\ep^{\f56}\|u_0\|_{L^2}^{\f16}.
    \end{aligned}
\end{equation*}
This implies $T_\ep^0=T_\ep$. Then by definition of $\|\cdot\|_{X_T}$, we have
\begin{equation}\label{S6eq1}
    \|u(\frac{T_\ep}{2})\|_{L^\infty}\leq (\frac{T_\ep}{2})^{-\frac{1}{2}} \|u\|_{X_{T_\ep}}\lesssim \ep^{\frac{5}{6}-\frac{2}{3}}\|u_0\|_{L^2}^{\f16-\f19}\|\rho_0\|_{C^2},
\end{equation}
which can be arbitrarily small if we take $\ep$ small.

This, together with the fact that $\|\rho(t) - 1\|_{L^\infty} = \|\rho_0 - 1\|_{L^\infty}$ and $\|u(t)\|_{L^2}\leq \|u_0\|_{L^2},$ implies that for any $3 < p < \infty,$
\begin{equation}\label{smallness}
\left\| \mathbb{P}\!\left( \rho\!\left( \frac{T^\ast}{2} \right) u\!\left( \frac{T^\ast}{2} \right) \right) \right\|_{\dot{B}_{p,\infty}^{-1 + \frac{3}{p}}} \leq C \left\| u\!\left( \frac{T^\ast}{2} \right) \right\|_{L^3} \leq C\|u_0\|_{L^2}^{\f23}  \left\| u\!\left( \frac{T^\ast}{2} \right) \right\|_{L^\infty}^{\f13}. \end{equation}
Therefore, we can take $\e_2$ sufficiently small in \eqref{glocond}, such that if $\ep<\ep_2$, then \eqref{S6eq1} and \eqref{smallness} lead to $\left\| \mathbb{P}\!\left( \rho\!\left( \frac{T^\ast}{2} \right) u\!\left( \frac{T^\ast}{2} \right) \right) \right\|_{\dot{B}_{p,\infty}^{-1 + \frac{3}{p}}} \leq \varepsilon_1$ in Theorem \ref{HWZ}, then  the system \eqref{eqinns} with initial data at $T^\ast/2$  generates  a global solution.  This completes the proof of Theorem \ref{thm.global-existence}. \end{proof}
\begin{rmk}\label{rmk6.1}
    The $C^2$ request for $\rho_0$ is used in \eqref{S6eq2}, since if we use the a priori estimate in Section \ref{sec6}, $T^*$ will be related to $\ep_2$, say $\ep_2^{20}$, so it is impossible to keep \eqref{S6eq1} small.
\end{rmk}

    \section{Appendix}\label{append}
  \subsection{Energy estimates}\label{sec3}
In this section, we derive energy estimates for smooth solutions $(\rho, u, P)$ to the inhomogeneous Navier–Stokes system \eqref{eqinns}. For the classical Navier–Stokes equations ($\rho \equiv 1$), it is well known that the equation satisfies the energy inequality $\|u(t)\|_{L^2} \leq \|u_0\|_{L^2}$ for any $t \geq 0$. In the inhomogeneous and non-vacuum case, a similar property holds; see Lemma \ref{eng1}. For the two-dimensional case, the $L^2$ norm is critical, which allows one to obtain global regularity for large initial data (see \cite{HSWZ} and references therein). In three dimensions, however, the $L^2$ norm is supercritical. Nevertheless, we will use the $L^2$ norm to estimate the lower-order remainder terms $u_R$, which enjoy additional regularity, making the loss of regularity from the critical level to $L^2$ manageable. We derive these energy estimates because they help control terms with negative regularity, such as $\|\nabla \Delta^{-1} u\|_{L^\infty} \lesssim \|u\|_{L^2}^{\frac{1}{3}} \|u\|_{L^\infty}^{\frac{2}{3}}$, and because the Leray projection $\mathbb{P}$ is bounded on $L^p$ only for $p \in (1, \infty)$.

Assume that $u_0 \in L^2$ and that
\begin{equation}\label{engasp}
\begin{split}
&\sup_{t \leq T} \left( t^{\frac{1}{2}} \|u(t)\|_{L^\infty} + t^{1+\frac{\eta}{2}} \|u(t)\|_{\dot{C}^{1+\eta}} \right) \leq M \andf\\
& C_1\leq \rho(t,x)\leq C_2\ \ \forall\ (t,x)\in [0,T]\times\R^3,
\end{split}
\end{equation}
for some $M > 1$, $T < +\infty$, and $\eta \in (0,1)$. These quantities arise respectively from \eqref{lowerbd} %Proposition \ref{regdiff}, 
and the definition of the $\|u\|_{X_T^1}$ norm. Below the constant $C$ may depend on $C_1$ and $C_2.$

\begin{lem}[Basic Energy Estimate]\label{eng1}
Let $(\rho, u, P)$ be a smooth enough solution of \eqref{eqinns} with initial data $u_0 \in L^2$. Then
\[
\|u\|_{L_T^\infty L^2} + \|\nabla u\|_{L_T^2 L^2} \leq C \|u_0\|_{L^2}.
\]
\end{lem}

\begin{proof}
Taking the $L^2$ inner product of the momentum equation with $u$ and using $\nabla \cdot u = 0$, we obtain
\[
\frac{1}{2} \frac{d}{dt} \left( \int_{\R^3} \rho(t,x) |u(t,x)|^2 \, dx \right) + \|\nabla u(t)\|_{L^2}^2 = 0.
\]
Integrating in time and applying \eqref{engasp} yields the desired estimate.
\end{proof}

\begin{lem}[Estimates with Time Weights]\label{eng2}
Let $(\rho, u, P)$ be a smooth enough solution of \eqref{eqinns} with $u_0 \in L^2$ such that \eqref{engasp} holds. Then
\[
\sup_{t \leq T} t^{\frac{1}{2}} \|\nabla u(t)\|_{L^2} + \left\| t^{\frac{1}{2}} \partial_t u \right\|_{L_T^2 L^2} \leq C(1+ M) \|u_0\|_{L^2}.
\]
\end{lem}

\begin{proof}
Taking the $L^2$ inner product of the momentum equation with $\partial_t u$, we obtain
\[
\int_{\R^3} \rho |\partial_t u|^2 \, dx + \frac{1}{2} \frac{d}{dt } \|\nabla u(t)\|_{L^2}^2 = -\int_{\R^3} \rho \operatorname{div}(u \otimes u) \cdot \partial_t u \, dx.
\]
Multiplying both sides by $t$ and applying Hölder's inequality yields
\begin{align*}
\frac{d}{dt}\bigl( t \|\nabla u(t)\|_{L^2}^2 \bigr) + t \|\sqrt{\rho} \partial_t u\|_{L^2}^2 \lesssim
& \|\nabla u\|_{L^2}^2 + t \|\rho u \cdot \nabla u\|_{L^2}^2\\
\lesssim &\|\nabla u\|_{L^2}^2 +C_1^2 t \|u(t)\|_{L^\infty}^2 \|\nabla u\|_{L^2}^2.
\end{align*}
Integrating over $(0, T)$ and using \eqref{engasp} together with Lemma \ref{eng1} completes the proof.
\end{proof}

\begin{lem}[Higher Regularity Estimate]\label{eng3}
Let $(\rho, u, P)$ be a smooth enough solution of \eqref{eqinns} with $u_0 \in L^2$ such that \eqref{engasp} holds. Then
\[
\sup_{t \leq T} t \|\partial_t u(t)\|_{L^2} + \left\| t \nabla \partial_t u \right\|_{L_T^2 L^2} \leq C(1+M)^2 M^{2} \|u_0\|_{L^2}.
\]
\end{lem}
\begin{proof}
Differentiating \eqref{eqinns} in time gives
\begin{equation*}
\rho \partial_{tt} u - \partial_t \Delta u = -\partial_t \rho \, \partial_t u - \partial_t \rho \, \nabla \cdot (u \otimes u) - \rho \partial_t u \cdot \nabla u - \rho u \cdot \nabla \partial_t u + \nabla \partial_t P.
\end{equation*}
Taking the $L^2$ inner product with $\partial_t u$ and multiplying the resulting inequality by $t^2$, we obtain
\begin{equation}\label{ener}
\begin{split}
&\frac{1}{2} \frac{d}{dt} \bigl( t^2 \|\sqrt{\rho} \partial_t u(t)\|_{L^2}^2 \bigr) + t^2 \|\partial_t \nabla u\|_{L^2}^2 = t \|\sqrt{\rho} \partial_t u(t)\|_{L^2}^2 \\
&\quad- t^2 \int_{\R^3} \partial_t \rho(t,x) |\partial_t u(t,x)|^2 \, dx  - t^2 \int_{\R^3} \partial_t \rho(t,x) (u \otimes \partial_t u)(t,x) : \nabla u(t,x) \, dx \\
&\quad - t^2 \int_{\R^3} \rho(t,x) (\partial_t u \otimes \partial_t u)(t,x) : \nabla u(t,x) \, dx \\
&:= I(t) + II(t) + III(t) + IV(t).
\end{split}
\end{equation}
By Lemma \ref{eng1}, Lemma \ref{eng2}, and \eqref{engasp}, we have
\begin{equation}\label{eng33}
\begin{split}
\int_0^T \bigl( |I(t)| + |IV(t)| \bigr) \, dt \lesssim & \|t^{\frac12}\sqrt{\rho} \partial_t u(t)\|_{L^2_TL^2}^2
+C_1\|t\na u(t)\|_{L^\infty}\|t^{\frac12}\sqrt{\rho} \partial_t u(t)\|_{L^2_TL^2}^2\\
\leq &C(1+ M)^{2} \|u_0\|_{L^2}^2,
\end{split}
\end{equation}
where we used \eqref{engasp} and interpolation inequality so that
\[ 
\|t\na u(t)\|_{L^\infty}\lesssim \|t^{\f12}u(t)\|_{L^\infty}^{\frac\eta{1+\eta}}\|t^{1+\frac\eta2}u(t)\|_{\dot C^{1+\eta}}^{\frac1{1+\eta}}\lesssim M.\]

For $II$, note that by the transport equation for $\rho$, the Cauchy–Schwarz inequality, and \eqref{engasp}, we obtain
\begin{equation*}
\begin{aligned}
|II(t)| &= \left| 2t^2 \int_{\R^3} \rho(t,x) (u \otimes \partial_t u)(t,x) : \partial_t \nabla u(t,x) \, dx \right| \\
&\leq C_2t^{\frac12}\|u(t)\|_{L^\infty}\|t^{\frac12}\p_tu(t)\|_{L^2}\|t\na\p_t u(t)\|_{L^2}\\
&\leq \frac{1}{8} t^2 \|\partial_t \nabla u(t)\|_{L^2}^2 + C M^2 t \|\partial_t u(t)\|_{L^2}^2,
\end{aligned}
\end{equation*}
which together with Lemma \ref{eng2} yields
\begin{equation}\label{eng34}
\begin{aligned}
\int_0^T |II(t)| \, dt \leq \frac{1}{8} \| t \partial_t \nabla u(t) \|_{L_T^2 L^2}^2 + C (1+M)^2M^{2} \|u_0\|_{L^2}^2.
\end{aligned}
\end{equation}

For $III$, using the transport equation for $\rho$ in \eqref{eqinns}, we have
\begin{equation*}
\begin{aligned}
III &= t^2 \int_{\R^3} \operatorname{div}(\rho u) (u \otimes \partial_t u) : \nabla u \, dx \\
&= -t^2 \int_{\R^3} \rho (u \cdot \nabla u \otimes \partial_t u) : \nabla u \, dx - t^2 \int_{\R^3} \rho (u \otimes u \cdot \nabla \partial_t u) : \nabla u \, dx \\
&\quad - t^2 \int_{\R^3} \rho (u \otimes \partial_t u) : (u \cdot \nabla (\nabla u)) \, dx \\
&:= III_1 + III_2 + III_3.
\end{aligned}
\end{equation*}
By \eqref{engasp}, we deduce
\begin{equation*}
\begin{aligned}
|III_1(t)| &\lesssim \|\sqrt{t} u(t)\|_{L^\infty} \|t \nabla u(t)\|_{L^\infty} \|\sqrt{t} \partial_t u(t)\|_{L^2} \|\nabla u(t)\|_{L^2} \\
&\lesssim M^2 \|\sqrt{t} \partial_t u\|_{L^2} \|\nabla u\|_{L^2}.
\end{aligned}
\end{equation*}
Then Lemma \ref{eng1} and Lemma \ref{eng2} lead to
\begin{equation}\label{III1}
\int_0^T |III_1(t)| \, dt \lesssim (1+M)^2M^{2}  \|u_0\|_{L^2}^2.
\end{equation}
By the Cauchy–Schwarz inequality, \eqref{engasp}, and Lemma \ref{eng1}, we obtain
\begin{equation}\label{III2}
\int_0^T |III_2(t)| \, dt \leq C M^2 \int_0^T \|\nabla u\|_{L^2} \|t \nabla \partial_t u\|_{L^2} \, dt \leq \frac{1}{16} \| t \nabla \partial_t u \|_{L_T^2 L^2}^2 + C M^{4} \|u_0\|_{L^2}^2.
\end{equation}
For $III_3$, we first need an estimate for $\|\nabla^2 u(t)\|_{L^2}$. From \eqref{eqinns}, we have
\begin{equation*}
-\Delta u = -\mathbb{P} (\rho \partial_t u + \rho u \cdot \nabla u).
\end{equation*}
Thus, by singular integral theory, together with \eqref{engasp}, Lemma \ref{eng1}, and Lemma \ref{eng2}, we obtain
\begin{equation*}
\begin{aligned}
\|\sqrt{t} \nabla^2 u\|_{L_T^2 L^2} &\lesssim \|\sqrt{t} \partial_t u\|_{L_T^2 L^2} + \|\sqrt{t} u\|_{L_T^\infty L^\infty} \|\nabla u\|_{L_T^2 L^2} \\
&\lesssim  (1+M) \|u_0\|_{L^2}^2,
\end{aligned}
\end{equation*}
which implies
\begin{equation}\label{III3}
\begin{aligned}
\int_0^T |III_3(t)| \, dt &\lesssim \|\sqrt{t} u\|_{L_T^\infty L^\infty}^2 \|\sqrt{t} \partial_t u\|_{L_T^2 L^2} \|\sqrt{t} \nabla^2 u\|_{L_T^2 L^2} \\
&\lesssim (1+M)^2M^{2} \|u_0\|_{L^2}^2.
\end{aligned}
\end{equation}
It follows from \eqref{III1}, \eqref{III2}, and \eqref{III3} that
\begin{equation}\label{eng35}
\begin{aligned}
\int_0^T |III(t)| \, dt \leq \frac{1}{16} \| t \nabla \partial_t u \|_{L_T^2 L^2}^2 + C (1+M)^2M^{2} \|u_0\|_{L^2}^2.
\end{aligned}
\end{equation}
Substitutining \eqref{eng33}, \eqref{eng34}, and \eqref{eng35}  into \eqref{ener} yields the desired result.
\end{proof}   

\subsection{Proof of Lemma  \ref{lemtta}}\label{sub7.2}
In this appendix we give a self-contained proof of Proposition \ref{lemtta}. The reader is referred to \cite{BCD} for a more detailed exposition. Before proceeding with the proof, we need some preparations.

We denote the linear functional $Lf = (L^m f)_{m=1,2,3}$ as the solution to the following system
\begin{equation}\label{defLm}
\begin{aligned}
&\partial_t L^m f - \Delta L^m f + \nabla P = \partial_m f,\\
&\nabla \cdot L^m f = 0,\\
&L^m f|_{t=0} = 0.
\end{aligned}
\end{equation}

We can prove the following version of \cite[Proposition 5.37]{BCD}.

\begin{lem}\label{lem5.37} 
Let $\Gamma^i(t,x), i=1,2,$ be given by \eqref{defnewker}. There exists $C > 0$ such that for any $R < b$ and $\alpha > 0$,
\begin{equation}\label{5.37h}
\|\Gamma_R^1 \ast_{t,x} f\|_{L^\infty([0,R^2] \times \mathbb{R}^3)} \leq \frac{C}{R^\alpha} \|f\|_{Z_T^{2,b}}.
\end{equation}
For $R < b$ and $0 < \alpha < 2$, there holds
\begin{equation*}\label{5.37l}
\|\Gamma_R^2 \ast_{t,x} f\|_{L^\infty([R^2, \infty] \times \mathbb{R}^3)} \leq \frac{C}{R^\alpha} \|f\|_{Z_T^{2,b}}.
\end{equation*}
\end{lem}

\begin{proof}
First, for any $R < b$, we split the convolution integral into annuli $C(0, 2^p R, 2^{p+1} R)$, which yields
\begin{equation*}
\begin{aligned}
|\Gamma_R^1 \ast f(t,x)| &\leq \sum_{p=0}^\infty \int_0^t \int_{C(0, 2^p R, 2^{p+1} R)} \frac{1}{|y|^{3+\alpha}} |f(\tau, x-y)| \, dy \, d\tau.
\end{aligned}
\end{equation*}
Note that $C(0, 2^p R, 2^{p+1} R)$ can be covered by balls $B(x', R)$ with center $x'$ and radius $R$, and the number of such $x'$ is approximately $2^{3p}$ for any $t \leq R^2$. Consequently,
\begin{equation*}
|\Gamma_R^1 \ast f(t,x)| \leq \sum_{p=0}^\infty (2^p R)^{-3-\alpha} 2^{3p} \sup_{x'} \int_0^t \int_{B(x', R)} |f(\tau, x-y)| \, dy \lesssim \frac{C}{R^\alpha} \|f\|_{Z_T^{2,b}}.
\end{equation*}

Second, we write
\begin{equation*}
\begin{aligned}
|\Gamma_R^2 \ast f(t,x)| &\lesssim \left( \int_{\min\{t/2, R^2\}}^t + \int_0^{\min\{R^2, t/2\}} \right) \int_{B_0(R)} \frac{1}{(\sqrt{t-\tau} + |x-y|)^{3+\alpha}} |f(\tau, y)| \, dy \, d\tau \\
&:= \Gamma_R^{21}(f) + \Gamma_R^{22}(f).
\end{aligned}
\end{equation*}
For $\Gamma_R^{22}(f)$, since in this regime $\frac{1}{(\sqrt{t-\tau} + |x-y|)^{3+\alpha}} \lesssim \frac{1}{t^{\frac{3+\alpha}{2}}}$, we obtain
\begin{equation*}
|\Gamma_R^{22}(f)| \lesssim \frac{R^3}{t^{\frac{3+\alpha}{2}}} \left( \frac{1}{R^3} \int_0^{R^2} \int_{B_0(R)} |f(\tau, y)| \, dy \, d\tau \right) \lesssim \frac{R^3}{t^{\frac{3+\alpha}{2}}} \|f\|_{Z_T^{2,b}} \lesssim \frac{1}{R^\alpha} \|f\|_{Z_T^{2,b}}, \quad \forall t \geq R^2.
\end{equation*}
For the remaining part, note that
\begin{equation*}
\int_{B_0(R)} \frac{1}{(a + |y|)^{3+\alpha}} \, dy \lesssim \frac{1}{a^\alpha} \int_{\mathbb{R}^3} \frac{1}{(1 + |y|)^{3+\alpha}} \, dy.
\end{equation*}
Hence,
\begin{equation*}
\begin{aligned}
|\Gamma_R^{21}(f)| &\lesssim \int_{\min\{t/2, R^2\}}^t \int_{B_0(R)} \frac{1}{(\sqrt{t-\tau} + |x-y|)^{3+\alpha}} \|f(\tau)\|_{L^\infty} \, dy \, d\tau \\
&\lesssim \|f\|_{Z_T^{2,b}} \left( \int_{R^2}^t \frac{1}{(t-\tau)^{\frac{\alpha}{2}}} \frac{d\tau}{\tau} + \int_{t/2}^t \frac{R^3}{t^{\frac{3+\alpha}{2}}} \frac{d\tau}{\tau} \right) \\
&\lesssim \frac{1}{R^\alpha} \|f\|_{Z_T^{2,b}}, \quad \forall t \geq R^2,\; 0 < \alpha < 2.
\end{aligned}
\end{equation*}
This completes the proof.
\end{proof}         

     \begin{lem}\label{conlem}
Assume that $\theta \in \mathcal{S}$. Define $f_\theta(t,\cdot) := t^{-\frac{3}{2}} \theta\!\left(\frac{\cdot}{\sqrt{t}}\right) \ast f(t,\cdot)$. Then
\begin{equation*}
\|f_\theta\|_{Z_T^{2,b}} \lesssim \|f\|_{Z_T^{2,b}}.
\end{equation*}
\end{lem}

\begin{proof}
Without loss of generality, we may assume that $x \in B_0(R)$. For any $R < b$, we have
\begin{equation}\label{conbdd}
\begin{aligned}
|f_\theta(t,x)| &\lesssim t^{-\frac{3}{2}} \int_{\R^3} \theta\!\left(\frac{x-y}{\sqrt{t}}\right) |\mathbf{1}_{B_0(2R)} f(t,y)| \, dy \\
&\quad + t^{-\frac{3}{2}} \int_{\R^3} \frac{1}{(1 + |x-y|/\sqrt{t})^4} \frac{t}{R^2} |f(t,y)| \, dy \\
&\lesssim t^{-\frac{3}{2}} \left| \theta\!\left(\frac{\cdot}{\sqrt{t}}\right) \ast \mathbf{1}_{B_0(2R)} f(t,\cdot) \right| + \frac{1}{R^2} \sup_{t \leq b^2} t \|f(t)\|_{L^\infty}.
\end{aligned}
\end{equation}
Integrating over $[0,R^2] \times B_0(R)$, we obtain
\begin{equation*}
\frac{1}{R^3} \|f_\theta\|_{L^1([0,R^2] \times B_0(R))} \lesssim \|f\|_{Z_T^{2,b}}.
\end{equation*}
The $L^\infty$ estimate also follows from \eqref{conbdd} for any $t < b^2$.
\end{proof}

Finally, to estimate the local $L^1$ norm, we need the following technical lemma.

\begin{lem}\label{lem5.38}
For any $R < b$ and any function $f : [0,R^2] \times B_0(2R) \to \mathbb{R}^3$, there holds
\begin{equation*}
\sup_{R < b} R^{-\frac{3}{2}} \|L(f)(t)\|_{L^2([0,R^2] \times \mathbb{R}^3)} \lesssim \|f\|_{Z_T^{2,b}}.
\end{equation*}
\end{lem}

\begin{proof}
We decompose $f$ into low and high frequencies. Precisely, let $\theta$ be a function whose Fourier transform $\widehat{\theta} \in C_c^\infty(\mathbb{R}^3)$ satisfies $\widehat{\theta}|_{B_0(1)} \equiv 1$. Write
\begin{equation*}
f = f^l + f^h, \qquad f^l = \mathcal{F}^{-1}\bigl( \widehat{\theta}(t^{1/2} \cdot) \widehat{f}(t,\cdot) \bigr)(x).
\end{equation*}
Then
\begin{equation*}
\begin{aligned}
\|f^h\|_{L^2([0,R^2]; \dot{H}^{-1})}^2 &\lesssim \int_0^{R^2}\int_{ \mathbb{R}^3} \frac{|1 - \widehat{\theta}(\sqrt{t} \xi)|^2}{t |\xi|^2} \, t |\widehat{f}(t,\xi)|^2  \, d\xi\, dt \\
&\lesssim \int_0^{R^2} t \|f(t)\|_{L^2}^2 \, dt \lesssim \|f\|_{L^1([0,R^2] \times \mathbb{R}^3)} \sup_{t < R^2} t \|f(t)\|_{L^\infty},
\end{aligned}
\end{equation*}
and then the classical energy estimate for the Stokes equations yields
\begin{equation}\label{4.1}
\|L f^h\|_{L^2([0,R^2] \times \mathbb{R}^3)}^2 \lesssim R^3 \|f\|_{Z_T^{2,b}}^2.
\end{equation}
The hardest part is $L f^l$. Observe that if we define
\[
\widetilde{f^l}(t,x) = \mathcal{F}^{-1}\bigl( e^{t|\cdot|^2} \widehat{\theta}(\sqrt{t} |\cdot|) \widehat{f}(t,\cdot) \bigr)(x),
\]
then
\[
L f^l = \nabla \int_0^t e^{(t-\tau)\Delta} \Bbb{P}f^l(\tau) \, d\tau = \nabla e^{t\Delta} \int_0^t \Bbb{P}\widetilde{f^l}(\tau) \, d\tau,
\]
where $\Bbb{P}$ denotes the Leray projection operator to the divergence free vector field space.
Since $\widehat{\theta}$ is compactly supported, there exists a Schwartz function $\widetilde{\theta}$ such that
\begin{equation}\label{formulalow}
\widetilde{f^l} = t^{-\frac{3}{2}} \widetilde{\theta}\!\left(\frac{\cdot}{\sqrt{t}}\right) \ast f(t,\cdot).
\end{equation}
Direct calculation gives
\begin{equation}\label{4.2}
\begin{aligned}
\mathcal{L}f &:= \|L f^l\|_{L^2([0,R^2] \times \mathbb{R}^3)}^2 \\
&= 2 \int_0^{R^2} \int_0^t \int_0^s \bigl\langle \nabla e^{t\Delta}\Bbb{P}\widetilde{f^l}(s), \nabla e^{t\Delta} \Bbb{P}\widetilde{f^l}(\tau) \bigr\rangle \, d\tau \, ds \, dt \\
&= -2 \int_0^{R^2} \int_0^t \int_0^s \bigl\langle \Delta e^{2t\Delta} \Bbb{P} \widetilde{f^l}(\tau),\widetilde{f^l}(s) \bigr\rangle \, d\tau \, ds \, dt \\
&= -\int_0^{R^2} \int_0^s \left\langle \left( \int_s^{R^2} \frac{d}{dt} e^{2t\Delta} \, dt \right) \Bbb{P} \widetilde{f^l}(\tau), \widetilde{f^l}(s) \right\rangle \, d\tau \, ds \\
&= \int_0^{R^2} \left\langle (e^{2s\Delta} - e^{2R^2\Delta}) \int_0^s  \Bbb{P}\widetilde{f^l}(\tau) \, d\tau, \widetilde{f^l}(s) \right\rangle \, ds \\
&\lesssim \| \widetilde{f^l} \|_{L_T^1([0,R^2] \times \mathbb{R}^3)} \sup_{s < R^2} \left\| (e^{2s\Delta} - e^{2R^2\Delta}) \Bbb{P} \int_0^s \widetilde{f^l}(\tau) \, d\tau \right\|_{L^\infty}.
\end{aligned}
\end{equation}
A simple consequence of \eqref{formulalow} are that
\begin{align*}
 \| \widetilde{f^l} \|_{L_T^1([0,R^2] \times \mathbb{R}^3)} \lesssim  \| {f^l} \|_{L_T^1([0,R^2] \times \mathbb{R}^3)} \lesssim  \|f\|_{Z_T^{2,b}},
 \end{align*}
 and
 \begin{align*}\label{4.3}
\left\| e^{2R^2\Delta} \int_0^s \widetilde{f^l}(\tau) \, d\tau \right\|_{L^\infty} \lesssim &R^{-3}
\left\| e^{2R^2\Delta} \int_0^s \widetilde{f^l}(\tau) \, d\tau \right\|_{L^1}\\
\lesssim &
R^{-3} \|f\|_{L^1([0,R^2] \times \mathbb{R}^3)} \lesssim \|f\|_{Z_T^{2,b}}.
\end{align*}
For the remaining term, note that
\begin{equation*}\label{4.4}
\left| e^{2s\Delta} \int_0^s \widetilde{f^l}(\tau, x) \, d\tau \right|
\leq s^{-\frac{3}{2}} \sum_{{\bf n} \in \mathbb{Z}^3} \int_{B({\bf n}\sqrt{s},\sqrt{s})} \int_0^s e^{-\frac{|x-y|^2}{2s}} \mathbf{1}(y) \left| \widetilde{f^l}(\tau, y) \right| \, d\tau \, dy,
\end{equation*}
where $B({\bf n}\sqrt{s}, \sqrt{s})$ denotes the ball with center ${\bf n}\sqrt{s}$ and radius $\sqrt{s}$. Using translation invariance, we may assume without loss of generality that $x = 0$. Then
\begin{equation}\label{4.5}
\begin{aligned}
\left| e^{2s\Delta} \int_0^s \widetilde{f^l}(\tau) \, d\tau (0)\right|
\lesssim &\sum_{|{\bf n}| > 2} e^{-\frac{|{\bf n}|^2}{4}} s^{-\frac{3}{2}} \int_0^s \int_{B({\bf n}\sqrt{s},\sqrt{s})} |\widetilde{f^l}(\tau, y)| \, dy \, d\tau \\
&+ \sum_{|{\bf n}| \leq 2} s^{-\frac{3}{2}} \int_0^s \int_{B({\bf n}\sqrt{s}, \sqrt{s})} |\widetilde{f^l}(\tau, y)| \, dy \, d\tau \\
&\lesssim \Bigl(1 + \sum_{{\bf n} \in \mathbb{Z}^3} e^{-\frac{|n|^2}{4}}\Bigr) \| \widetilde{f^l} \|_{Z_T^{2,b}} \lesssim \|f\|_{Z_T^{2,b}},
\end{aligned}
\end{equation}
where the last inequality follows from Lemma \ref{conlem}. Here we also used $\sqrt{s} < R < b$, and the fact that $\int_0^s \int_{B{(\bf n}\sqrt{s},\sqrt{s})} |\widetilde{f^l}(\tau, y)| \, dy \, d\tau \lesssim s^{\frac{3}{2}} \| \widetilde{f^l} \|_{Z_T^{2,b}}$ for any $n$. Combining \eqref{4.2}–\eqref{4.5}, we obtain
\begin{equation*}
\|L f^l\|_{L^2([0,R^2] \times \mathbb{R}^3)}^2 \lesssim R^3 \|f\|_{Z_T^{2,b}}^2,
\end{equation*}
which together with \eqref{4.1} completes the proof of Lemma \ref{lem5.38}.
\end{proof}   

  \begin{proof}[Proof of Lemma \ref{lemtta}]
Obviously, it suffices to prove that $L$ (defined in \eqref{defLm}) maps $Z_T^{2,b}$ continuously into $Z_T^{1,b}$. By the Duhamel principle,  $Lf$ admits the explicit representation (see \cite{BCD} for instance)
\begin{equation*}
(L^m f)^k(t,x) = \sum_{\ell=1}^3 \int_0^t \int \Gamma_{m,\ell}^k(t-\tau, x-y) f^\ell(\tau, y) \, dy \, d\tau,
\end{equation*}
where $\Gamma_{j,\ell}^k$ is a kernel satisfying
\begin{equation*}
|\Gamma_{j,\ell}^k(t,x)| \lesssim \Gamma_R^1(t,x) + \Gamma_R^2(t,x), \quad \forall R > 0,
\end{equation*}
with $\Gamma_R^1$ and $\Gamma_R^2$ defined in \eqref{defnewker}. By Lemma \ref{lem5.37} with $\alpha = 1$ and $R = \sqrt{t}$, we obtain
\begin{equation*}
\|Lf(t)\|_{L^\infty} \lesssim t^{-\frac{1}{2}} \|f\|_{Z_T^{2,b}}, \quad \forall t < b^2.
\end{equation*}
Furthermore, we need to estimate the local $L^2$ norm. Without loss of generality, we may restrict our attention to a neighborhood of the origin. Note that
\begin{equation*}
Lf = L(\mathbf{1}_{B(0,2R)} f) + L(\mathbf{1}_{B^c(0,2R)} f).
\end{equation*}
For the exterior part, note that when $|x| < R$ and $|y| > 2R$, we have $\frac{1}{(\sqrt{t} + |x-y|)^4} \lesssim \frac{1}{(\sqrt{t} + |y|)^4}$. The first estimate in Lemma \ref{lem5.37} then yields
\begin{equation*}
R^{-\frac{3}{2}} \| L(\mathbf{1}_{B^c(0,2R)} f) \|_{L^2([0,R^2] \times B(0,R))} \lesssim R \| L(\mathbf{1}_{B^c(0,2R)} f) \|_{L^\infty([0,R^2] \times B(0,R))} \lesssim \|f\|_{Z_T^{2,b}}, \quad \forall R < b.
\end{equation*}
The corresponding estimate for $L(\mathbf{1}_{B_0(2R)} f)$ follows from Lemma \ref{lem5.38}. We thus complete the proof of  Lemma \ref{lemtta}.
\end{proof}

\medskip

\noindent \textbf{Acknowledgement.} Quoc-Hung Nguyen is partially supported by the CAS Project for Young Scientists in Basic Research (Grant No. YSBR-03), the Academy of Mathematics and Systems Science, Chinese Academy of Sciences startup fund, the National Natural Science Foundation of China (Grant Nos. 12288201 and 12494542), and the National Key R\&D Program of China (Grant No. 2021YFA1000800). Ping Zhang is partially supported by the National Key R\&D Program of China (Grant No. 2021YFA1000800) and the National Natural Science Foundation of China (Grant Nos. 12421001, 12494542, and 12288201).

\section*{Declarations}

\subsection*{Conflict of interest} The authors declare that there are no conflicts of interest.

\subsection*{Data availability}
This article has no associated data.

\end{document}